\documentclass[12 pt]{amsart}
\usepackage{fullpage, amsthm, amsmath, amsfonts, amssymb, url, mathrsfs}
\usepackage{esint}
\usepackage{hyperref}
\newcommand{\R}{\mathbb R}

\newcommand{\Op}{\mathcal O}
\def \colonequals {\mathrel{\mathop:}=}

 \numberwithin{equation}{section}

\newtheorem{thm}{Theorem}[section]
\newtheorem{defin}[thm]{Definition}
\newtheorem{lem}[thm]{Lemma}
\newtheorem{cor}[thm]{Corollary}
\newtheorem{prop}[thm]{Proposition}
\newtheorem{rem}[thm]{Remark}
\newtheorem*{thm*}{Theorem}

\title{A Two-Phase Free Boundary Problem for Harmonic Measure}
\date{\today}
\author{Max Engelstein}
\address{Department of Mathematics, University of Chicago, 5734 S. University Avenue, Chicago, IL, 60637}
\email{maxe@math.uchicago.edu}

\begin{document}

\maketitle

\begin{abstract}
We study a 2-phase free boundary problem for harmonic measure first considered by Kenig and Toro \cite{kenigtorotwophase} and prove a sharp H\"older regularity result. The central difficulty is that there is no {\it a priori} non-degeneracy in the free boundary condition. Thus we must establish non-degeneracy by means of monotonicity formulae. 

\end{abstract}

\tableofcontents

\section{Introduction}\label{sec: intro}
In this paper we consider the following two-phase free boundary problem for harmonic measure: let $\Omega^+$ be an unbounded 2-sided non-tangentially accessible (NTA) domain (see Definition \ref{ntadomain}) such that $\log(h)$ is regular, e.g. $\log(h) \in C^{0,\alpha}(\partial \Omega)$. Here $h \colonequals \frac{d\omega^-}{d\omega^+}$ and $\omega^{\pm}$ is the harmonic measure associated to the domain $\Omega^{\pm}$ ($\Omega^-  \colonequals \mathrm{int}((\Omega^+)^c)$).  We ask the question:  what can be said about the regularity of $\partial \Omega$? 

This question was first considered by Kenig and Toro (see \cite{kenigtorotwophase}) when $\log(h) \in \mathrm{VMO}(d\omega^+)$. They concluded, under the initial assumption of $\delta$-Reifenberg flatness, that $\Omega$ is a vanishing Reifenberg flat domain (see Definition \ref{reifenbergflat}). Later, the same problem, without the initial flatness assumption, was investigated by Kenig, Preiss and Toro (see \cite{kenigpreisstoro}) and Badger (see \cite{badgerharmonicmeasure} and \cite{badgerharmonicpolynomial}). Our work is a natural extension of theirs, though the techniques involved are substantially different. 

Our main theorem is:

\begin{thm}\label{maintheorem}
Let $\Omega$ be a 2-sided NTA domain with $\log(h) \in C^{k,\alpha}(\partial \Omega)$ where $k \geq 0$ is an integer and $\alpha \in (0,1)$.
\begin{itemize}
\item When $n =2$: $\partial \Omega$ is locally given by the graph of a $C^{k+1,\alpha}$ function.
\item When $n\geq 3$: there is some $\delta_n > 0$ such that if $\delta < \delta_n$ and $\Omega$ is $\delta$-Reifenberg flat then $\partial \Omega$ is locally given by the graph of a $C^{k+1, \alpha}$ function.
\end{itemize}
Similarly, if $\log(h) \in C^{\infty}$ or $\log(h)$ is analytic we can conclude (under the same flatness assumptions above) that $\partial \Omega$ is locally given by the graph of a $C^\infty$ (resp. analytic) function. 
\end{thm}

When $n> 2$, the initial flatness assumption is needed; if $n \geq 4$, $\Omega = \{X\in \R^n \mid x_1^2 + x_2^2 > x_3^2 + x_4^2\}$ is a 2-sided NTA domain such that $\omega^+ = \omega^-$ on $\partial \Omega$ (where the poles are at infinity). As such, $h \equiv 1$ but, at zero, this domain is not a graph. In $\R^3$, H. Lewy (see \cite{lewy}) proved that, for $k$ odd, there are homogeneous harmonic polynomials of degree $k$ whose zero set divides $\mathbb S^2$ into two domains. The cones over these regions are NTA domains and one can calculate that $\log(h) = 0$. Again, at zero, $\partial \Omega$ cannot be written as a graph. However, these two examples suggest an alternative to the {\it a priori} flatness assumption.

\begin{thm}\label{maintheoremprime}
Let $\Omega$ be a Lipschitz domain (that is, $\partial \Omega$ can be locally written as the graph of a Lipschitz function) and let $h$ satisfy the conditions of Theorem \ref{maintheorem}. Then the same conclusions hold. 
\end{thm}

The corresponding one-phase problem, ``Does regularity of the Poisson kernel imply regularity of the free boundary?", has been studied extensively. Alt and Caffarelli (see \cite{altcaf}) first showed, under suitable flatness assumptions, that $\log(\frac{d\omega}{d\sigma}) \in C^{0,\alpha}(\partial \Omega)$ implies $\partial \Omega$ is locally the graph of a $C^{1,s}$ function. Jerison (see \cite{jerison}) showed $s = \alpha$ above and, furthermore, if $\log(\frac{d\omega}{d\sigma}) \in C^{1,\alpha}(\partial \Omega)$ then $\partial \Omega$ is locally the graph of a $C^{2,\alpha}$ function (from here, higher regularity follows from classical work of Kinderlehrer and Nirenberg, \cite{kinderlehrernirenberg}). Later, Kenig and Toro (see \cite{kenigtoro}) considered when  $\log(\frac{d\omega}{d\sigma}) \in \mathrm{VMO}(d\sigma)$ and concluded that $\partial \Omega$ is a vanishing chord-arc domain (see Definition 1.8 in \cite{kenigtoro}). 

 Two-phase elliptic problems are also an object of great interest. The paper of Alt, Caffarelli and Friedman (see \cite{acf}) studied an ``additive" version of our problem.  Later, Caffarelli (see \cite{caf} for part one of three) studied viscosity solutions to an elliptic free boundary problem similar to our own. This work was then extended to the non-homogenous setting by De Silva, Ferrari and Salsa (see \cite{silvaferrarisalsa}). It is important to note that, while our problem is related to those studied above, we cannot immediately apply any of their results. In each of the aforementioned works there is an {\it a priori} assumption of non-degeneracy built into the problem (either in the class of solutions considered or in the free boundary condition itself). Our problem has no such {\it a priori} assumption. Unsurprisingly, the bulk of our efforts goes into establishing non-degeneracy. 

Even in the case of $n=2$, where the powerful tools of complex analysis can be brought to bear, our non-degeneracy results seem to be new. We briefly summarize some previous work in this area: let $\Omega^+$ be a simply connected domain bounded by a Jordan curve and $\Omega^- = \overline{\Omega^+}^c$. Then $\partial \Omega = G^+ \cup S^+ \cup N^+$ where \begin{itemize}
\item $\omega^+(N^+) = 0$
\item $\omega^+ << \mathcal H^1 << \omega^+$ on $G^+$
\item Every point of $G^+$ is the vertex of a cone in $\Omega^+$. Furthermore, if $C^+$ is the set of all cone points for $\Omega^+$ then $\mathcal H^1(C^+\backslash G^+) = 0 = \omega^+(C^+\backslash G^+)$. 
\item $\mathcal H^1(S^+) = 0$. 
\item For $\omega^+$ a.e $Q \in S^+$ we have $\limsup_{r\downarrow 0} \frac{\omega^+(B(Q,r))}{r} = +\infty$ and $\liminf_{r\downarrow 0} \frac{\omega^+(B(Q,r))}{r} = 0$
\end{itemize}  with a similar decomposition for $\omega^-$. These results are due to works by Makarov, McMillan, Pommerenke and Choi. See Garnett and Marshall \cite{garnettandmarshall}, Chapter 6 for an introductory treatment and more precise references. 

In our context, that is where $\omega^+ << \omega^- << \omega^+$, $\Omega$ is a 2-sided NTA domain and $\log(h) \in C^{0,\alpha}(\partial \Omega)$, one can use the Beurling monotonicity formula (see Lemma 1 in \cite{bishopcarlesongarnettjones}) to show $\limsup_{r\downarrow \infty} \frac{\omega^{\pm}(B(x,r))}{r} < \infty$. Therefore, $\omega^{\pm}(S^+\cup S^-) = 0$ and we can write $\partial \Omega = \Gamma \cup N$ where $\omega^{\pm}(N) = 0$ and $\Gamma$ is 1-rectifiable (i.e. the image of countably many Lipschitz maps) and has $\sigma$-finite $\mathcal H^1$-measure. This decomposition is implied for $n > 2$ by the results of Section \ref{sec: thetanondegenerate}. In order to prove increased regularity one must bound from below $\liminf_{r\downarrow 0} \frac{\omega^+(B(Q,r))}{r}$, which we do in Corollary \ref{uniformlowerbound} and seems to be an original contribution to the literature.

The approach is as follows: after establishing some initial facts about blowups and the Lipschitz continuity of the Green's function (Sections \ref{blowupsonntaandlipschitz} and \ref{sec: u is lipschitz}) we tackle the issue of degeneracy. Our main tools here are the monotonicity formulae of Almgren, Weiss and Monneau which we introduce in Section \ref{sec: thetanondegenerate}. Unfortunately, in our circumstances these functionals are not actually monotonic. However, and this is the key point, we show that they are ``almost monotonic" (see, e.g., Theorem \ref{growthofM}). More precisely, we bound the first derivative from below by a summable function. From here we quickly conclude pointwise non-degeneracy. In Section \ref{sec: c1domain}, we use the quantitative estimates of the previous section to prove uniform non-degeneracy and establish the $C^1$ regularity of the free boundary. 

At this point the regularity theory developed by De Silva et al. (see \cite{silvaferrarisalsa}) and Kinderlehrer et al. (see \cite{kinderlehrernirenberg} and \cite{knstwosided}) can be used to produce the desired conclusion. However, these results cannot be applied directly and some additional work is required to adapt them to our situation. These arguments, while standard, do not seem to appear explicitly in the literature. Therefore, we present them in detail here. Section  \ref{sec: initialholderregularity} adapts the iterative argument of De Silva, Ferrari and Salsa \cite{silvaferrarisalsa} to get $C^{1,s}$ regularity for the free boundary. In Section \ref{sec: higherholder} we first describe how to establish optimal $C^{1,\alpha}$ regularity and then $C^{2,\alpha}$ regularity (in analogy to the aforementioned work of Jerison \cite{jerison}). This is done through an estimate in the spirit of Agmon et al. (\cite{adn1} and \cite{adn2}) which is proven in the appendix. Higher regularity then follows easily.

\medskip

\noindent {\bf Acknowledgements:} This research was partially supported by the Department of Defense's National Defense Science and Engineering Graduate Fellowship as well as by the National Science Foundation's Graduate Research Fellowship, Grant No. (DGE-1144082). We thank the anonymous referee for several helpful comments and corrections. The author would also like thank Professor Carlos Kenig for his guidance, support and, especially, boundless patience. 

\section{Notation and Definitions}\label{sec: prelimsanddefs}

Throughout this article $\Omega \subset \R^n$ is an open set and our object of study. For simplicity, $\Omega^+ \colonequals \Omega$ and $\Omega^- \colonequals \overline{\Omega}^c$. To avoid technicalities we will assume that $\Omega^{\pm}$ are both unbounded and let $u^{\pm}$ be the Green's function of $\Omega^{\pm}$ with a pole at $\infty$ (our methods and theorems apply to finite poles and bounded domains). Let $\omega^{\pm}$ be the harmonic measure of $\Omega^{\pm}$ associated to $u^{\pm}$; it will always be assumed that $\omega^- << \omega^+ << \omega^-$. Define $h = \frac{d\omega^-}{d\omega^+}$ to be the Radon-Nikodym derivative and unless otherwise noted, it will be assumed that $\log(h) \in C^{0,\alpha}(\partial \Omega)$. 

Finally, for a measurable $f: \R^n \rightarrow \R$, we write $f^+(x) \colonequals |f(x)| \chi_{\{f > 0\}}(x)$ and $f^-(x) \colonequals |f(x)|\chi_{\{f < 0\}}(x)$. In particular, $f(x) = f^+(x) - f^-(x)$. Define $u^{\pm}$ outside of $\Omega^{\pm}$ to be identically zero and set $u(x) \colonequals u^+(x)- u^-(x)$ (so that these two notational conventions comport with each other). 

Recall the definition of an non-tangentially accessible (NTA) domain.

\begin{defin}\label{ntadomain}[See \cite{jerisonandkenig} Section 3]
A domain $\Omega \subset \R^n$ is {\bf non-tangentially accessible}, (NTA), if there are constants $M > 1, R_0 > 0$ for which the following is true:
\begin{enumerate}
\item $\Omega$ satisfies the corkscrew condition: for any $Q \in \partial \Omega$ and $0 < r < R_0$ there exists $A = A_r(Q) \in \Omega$ such that $M^{-1}r < \mathrm{dist}(A, \partial \Omega) \leq |A-Q| < r$.
\item $\overline{\Omega}^c$ satisfies the corkscrew condition.
\item $\Omega$ satisfies the Harnack chain condition: let $\varepsilon > 0, x_1, x_2\in \Omega\cap B(R_0/4, Q)$ for a $Q\in \partial \Omega$ with $\mathrm{dist}(x_i, \partial \Omega) > \varepsilon$ and $|x_1-x_2| \leq 2^k \varepsilon$. Then there exists a ``Harnack chain" of overlapping balls contained in $\Omega$ connecting $x_1$ to $x_2$. Furthermore we can ensure that there are no more than $Mk$ balls and that the diameter of each ball is bounded from below by $M^{-1}\min_{i=1,2}\{\mathrm{dist}(x_i, \partial \Omega)\}$
\end{enumerate}

When $\Omega$ is unbounded we also require that $\R^n\backslash \partial \Omega$ has two connected components and that $R_0 = \infty$.

We say that $\Omega$ is {\bf $2$-sided NTA} if both $\Omega$ and $\overline{\Omega}^c$ are NTA domains. The constants $M, R_0$ are referred to as the ``NTA constants" of $\Omega$. 
\end{defin}

It should be noted that our analysis in this paper will be mostly local. As such we need only that our domains be ``locally NTA" (i.e. that $M, R$ can be chosen uniformly on compacta). However, for the sake of simplicity we will work only with NTA domains. We now recall the definition of a Reifenberg flat domain. 

\begin{defin}\label{reifenbergflat}
For $Q\in \partial \Omega$ and $r > 0$, $$\theta(Q,r) \colonequals \inf_{P \in G(n, n-1)} D[\partial \Omega \cap B(Q,r), \{P + Q\} \cap B(Q,r)],$$ where $D[A,B]$ is the Hausdorff distance between $A,B$. 

For $\delta > 0, R> 0$ we then say that $\Omega$ is {\bf $(\delta, R)$-Reifenberg flat} if for all $Q \in \partial \Omega, r < R$ we have $\theta(Q,r) \leq \delta.$ When $\Omega$ is unbounded we say it is {\bf $\delta$-Reifenberg flat} if the above holds for all $0 < r < \infty$. 

Additionally, if $K \subset \subset \R^n$ we can define $$\theta_K(r) = \sup_{Q\in K\cap \partial \Omega} \theta(Q,r).$$ Then we say that $\Omega$ is {\bf vanishing Reifenberg flat} if for all $K \subset \subset \R^n$, $\limsup_{r\downarrow 0} \theta_K(r) = 0.$
\end{defin}

\begin{rem}\label{lipschitzdomain}
Recall that a $\delta$-Reifenberg flat NTA domain is not necessarily a Lipschitz domain, and a Lipschitz domain need not be $\delta$-Reifenberg flat. However, all Lipschitz domains are (locally) 2-sided NTA domains (see \cite{jerisonandkenig} for more details and discussion). \end{rem}

Finally, let us make two quick technical points regarding $h$.

\begin{rem}\label{rnderivative}
For every $Q \in \partial \Omega$, we have $\lim_{r\downarrow 0} \frac{\omega^-(B(Q,r))}{\omega^+(B(Q,r))} = h(Q)$ (in particular the limit exists for every $Q\in \partial \Omega$).
\end{rem}

\begin{proof}[Justification of Remark]
By assumption, $\frac{d\omega^-}{d\omega^+}$ agrees with a H\"older continuous function $h$ where defined (i.e. $\omega^+$-almost everywhere). For any $Q\in \partial \Omega$ we can rewrite $\lim_{r\downarrow 0} \frac{\omega^-(B(Q,r))}{\omega^+(B(Q,r))} = \lim_{r\downarrow 0}\fint_{B(Q,r)} \frac{d\omega^-}{d\omega^+}(P) d\omega^+(P) = \lim_{r\downarrow 0}\fint_{B(Q,r)} h(P)d\omega^+(P)$. This final limit exists and is equal to $h(Q)$ everywhere because $h$ is continuous.
\end{proof}

We also note that $h$ is only defined on $\partial \Omega$. However, by Whitney's extension theorem, we can extend $h$ to $\tilde{h}: \R^n \rightarrow \R$ such that $\tilde{h} = h$ on $\partial \Omega$ and $\log(\tilde{h}) \in C^\alpha(\R^n)$ (or, if $\log(h) \in C^{k,\alpha}(\partial \Omega)$ then $\log(\tilde{h}) \in C^{k,\alpha}(\R^n)$). For simplicity's sake, we will abuse notation and let $h$ refer to the function defined on all of $\R^n$. 

\section{Blowups on NTA and Lipschitz Domains}\label{blowupsonntaandlipschitz}

For any $Q \in \partial \Omega$ and any sequence of $r_j \downarrow 0$ and $Q_j \in \partial \Omega$ such that $Q_j\rightarrow Q$, define the {\bf pseudo-blowup} as follows:

\begin{equation}\label{blowup}
\begin{aligned}
\Omega_j &\colonequals \frac{1}{r_j}(\Omega-Q_j) \\
u^{\pm}_j(x) & \colonequals  \frac{u^{\pm}(r_jx + Q_j) r_j^{n-2}}{\omega^{\pm}(B(Q_j,r_j))}\\
\omega^{\pm}_j(E) & \colonequals  \frac{\omega^\pm(r_jE + Q_j)}{\omega^{\pm}(B(Q_j, r_j))}.
\end{aligned}
\end{equation}

A pseudo-blowup where $Q_j \equiv Q$, is a {\bf blowup}. Kenig and Toro characterized pseudo-blowups of 2-sided NTA domains when $\log(h) \in \mathrm{VMO}(d\omega^+)$. 

\begin{thm}\label{blowuplimits}[\cite{kenigtorotwophase}, Theorem 4.4]
Let $\Omega^{\pm}\subset \R^n$ be a 2-sided NTA domain, $u^{\pm}$ the associated Green's functions and $\omega^{\pm}$ the associated harmonic measures.  Assume $\log(h)\in \mathrm{VMO}(d\omega^+)$. Then, along any pseudo-blowup,  there exists a subsequence (which we shall relabel for convenience) such that (1) $\Omega_j \rightarrow \Omega_\infty$ in the Hausdorff distance uniformly on compacta, (2) $u_j^{\pm} \rightarrow u_\infty^{\pm}$ uniformly on compact sets (3) $\omega_j^{\pm} \rightharpoonup \omega_\infty^{\pm}$. Furthermore, $u_\infty \colonequals u_\infty^+ - u_\infty^-$ is a harmonic polynomial (whose degree is bounded by some number which depends on the dimension and the NTA constants of $\Omega$) and $\partial \Omega_\infty = \{u_\infty = 0\}$.

Additionally, if $n=2$ or $\Omega$ is a $\delta$-Reifenberg flat domain with $\delta > 0$ small enough (depending on $n$) then $u_\infty(x)= x_n$ (possibly after a rotation). In particular, $\Omega$ is vanishing Reifenberg flat.
\end{thm}

This result plays a crucial role in our analysis. In particular, the key estimate in \eqref{slowerthanpoly} follows from vanishing Reifenberg flatness. Therefore, in order to prove Theorem \ref{maintheoremprime} we must establish an analogous result when $\Omega$ is a Lipschitz domain. 

\begin{cor}\label{blowupsforeveryone}
Let $\Omega \subset \R^n$ be as in Theorem \ref{maintheoremprime}. Then, along any pseudo-blowup we have (after a possible rotation) that $u_\infty(x) = x_n$. In particular, $\Omega^{\pm}$ is a vanishing Reifenberg flat domain. 
\end{cor}

\begin{proof}
We first recall Remark \ref{lipschitzdomain}, which states that any Lipschitz domain is a (locally) 2-sided NTA domain. Therefore, the conditions of Theorem \ref{blowuplimits} are satisfied. A result of Badger (Theorem 6.8 in \cite{badgerharmonicpolynomial}) says that, under the assumptions of Theorem \ref{blowuplimits}, the set of points where all {\it blowups} are 1-homogenous polynomials is in fact vanishing Reifenberg flat (``locally Reifenberg flat with vanishing constant" in the terminology of \cite{badgerharmonicpolynomial}). Additionally, graph domains (i.e. domains whose boundaries are locally the graph of a function) are closed under blowups, so all blowups of $\partial \Omega$ can be written locally as the graph of a some function. Observe that the zero set of a $k$-homogenous polynomial is a graph domain if and only if $k =1$. In light of all the above, it suffices to show that all blowups of $\partial \Omega$ are given by the zero set of a homogenous harmonic polynomial. We now recall another result of Badger. 
 
 \begin{thm*}[\cite{badgerharmonicmeasure}, Theorem 1.1]
If $\Omega$ is an NTA domain with harmonic measure $\omega$ and $Q\in \partial \Omega$, then $\mathrm{Tan}(\omega, Q) \subset P_d \Rightarrow \mathrm{Tan}(\omega, Q) \subset F_k$ for some $1 \leq k \leq d$. $P_d$ is the set of harmonic measures associated to a domain of the form  $\{h > 0\}$, where $h$ is a harmonic polynomial of degree $\leq d$. $F_k$ is the set of harmonic measures associated to a domain of the form $\{h > 0\}$, where $h$ is a homogenous harmonic polynomial of degree $k$.
 \end{thm*}  
 
In other words, if every blowup of an NTA domain is the zero set of a degree $\leq d$ harmonic polynomial, then every blowup of that domain is the zero set of a $k$-homogenous harmonic polynomial. This result, combined with Theorem \ref{blowuplimits}, immediately implies that all blowups of $\partial \Omega$ are given by the zero set of a $k$-homogenous harmonic polynomial. By the arguments above, $k =1$ and $\partial \Omega$ is vanishing Reifenberg flat. 

That $u_\infty = x_n$ (as opposed to $kx_n$ for some $k \neq 1$) follows from the fact that $\omega_\infty(B(0,1)) =\lim_i \omega_i(B(0,1)) \equiv 1$, and that $u_\infty^{\pm}$ is the Green's function associated to $\omega_\infty$.  
\end{proof}

Hereafter, we can assume, without loss of generality, that $\Omega$ is a vanishing Reifenberg flat domain and that all pseudo-blowups are 1-homogenous polynomials. 

\section{$u$ is Lipschitz}\label{sec: u is lipschitz}

The main aim of this section is to prove that $u$ is locally Lipschitz.\footnote{NB: In this section we need only assume that $\log(h) \in C(\partial \Omega)$.}  We adapt the method of Alt, Caffarelli and Friedman (\cite{acf}, most pertinently Section 5) which uses the following monotonicity formula to establish Lipschitz regularity for an ``additive" two phase free boundary problem. 

\begin{thm}\label{monotonicity}[\cite{acf}, Lemma 5.1] Let $f$ be any function in $C^0(B(x_0,R))\cap W^{1,2}(B(x_0,R))$ where $f(x_0) = 0$ and $f$ is harmonic in $B(x_0,R) \backslash \{f = 0\}$. Then $$J(x,r) \colonequals \frac{1}{r^2}\left(\int_{B(x,r)} \frac{|\nabla f^+|^2}{|x-y|^{n-2}}dy\right)^{1/2}\left(\int_{B(x,r)} \frac{|\nabla f^-|^2}{|x-y|^{n-2}}dy\right)^{1/2}$$ is increasing in $r \in (0, R)$ and is finite for all $r$ in that range.
\end{thm}

In a 2-sided NTA domain, $u\in C^0(B(Q,R))\cap W^{1,2}(B(Q,R))$ for any $Q \in \partial \Omega$ and any $R$ (as such domains are ``admissible" see \cite{kenigpreisstoro}, Lemma 3.6). This monotonicity immediately implies upper bounds on $\frac{\omega^{\pm}(B(Q,r))}{r^{n-1}}$.

\begin{cor}\label{upperboundondensity}
Let $K\subset \subset \R^n$ be compact. There is a $0 < C\equiv C_{K,n} < \infty$ such that $$\sup_{0 < r \leq 1}\sup_{Q\in K \cap \partial \Omega} \frac{\omega^{\pm}(B(Q,r))}{r^{n-1}} < C.$$
\end{cor}

\begin{proof}
Using the Theorem \ref{monotonicity} one can prove that $$\frac{\omega^{+}(B(Q, r))}{r^{n-1}}\frac{\omega^-(B(Q,r))}{r^{n-1}} \leq C\|u\|_{L^2(B(Q, 4))},\; \forall 0 < r \leq 1,$$ (see Remark 3.1 in \cite{kenigpreisstoro}). Note that $$\sup_{1 \geq r > 0, Q\in \partial \Omega \cap K} \left(\frac{\omega^{\pm}(B(Q, r))}{r^{n-1}}\right)^2 = \sup_{1 \geq r > 0, Q\in \partial \Omega \cap K} \frac{\omega^+(B(Q,r))}{r^{n-1}}\frac{\omega^-(B(Q,r))}{r^{n-1}} \frac{\omega^{\pm}(B(Q,r))}{\omega^{\mp}(B(Q,r))} \leq $$$$\sup_{P\in \partial \Omega,\; \mathrm{dist}(P,K) \leq 1}h^{\mp 1}(P) \sup_{1 \geq r > 0, Q\in \partial \Omega \cap K} \frac{\omega^{+}(B(Q,r))}{r^{n-1}}\frac{\omega^-(B(Q, r))}{r^{n-1}}.$$  By continuity, $\log(h)$ is  bounded on compacta and so we are done.
\end{proof}

Blowup analysis connects the Lipschitz continuity of $u$ to the boundedness of $\frac{\omega^{\pm}(B(Q,r))}{r^{n-1}}$. 

\begin{lem}\label{boundedaverage}
Let $K\subset \subset \R^n$ be compact, $Q \in K\cap \partial \Omega$ and $1 \geq r > 0$. Then there is a constant $C > 0$ (which depends only on dimension and $K$) such that $$\frac{1}{r}\fint_{\partial B(Q, r)} |u| < C.$$
\end{lem}

\begin{proof}
We rewrite $\frac{1}{r}\fint_{\partial B(Q, r)} |u| = \frac{1}{r} \fint_{\partial B(0, 1)} |u(ry + Q)| d\sigma(y)$. Standard estimates on NTA domains imply $u^{\pm}(ry + Q) \leq C_K u^{\pm}(A_\pm(Q, r)) \leq C_K\frac{\omega^{\pm}(B(Q, r))}{r^{n-2}}$ (see \cite{jerisonandkenig}, Lemmas 4.4 and 4.8). So $$\frac{1}{r}\fint_{\partial B(Q, r)} |u| \leq C_K \left(\frac{\omega^+(B(Q,r))}{r^{n-1}} + \frac{\omega^-(B(Q, r))}{r^{n-1}}\right).$$

Corollary \ref{upperboundondensity} implies the desired result.
\end{proof}

We then prove Lipschitz continuity around the free boundary. 

\begin{prop}\label{locallylipschitz}
If $K\subset\subset \R^n$ is compact then $|Du(x)| < C \equiv C(n,K) < \infty$ a.e. in $K$.
 \end{prop}
 
 \begin{proof}
As $u$ is analytic away from $\partial \Omega$ and $u \equiv 0$ on $\partial \Omega$ we can conclude that $Du$ exists a.e.
 
Pick $x \in K$ and, without loss of generality, let $x\in \Omega^+$. Define $\rho(x) \colonequals \mathrm{dist}(x, \partial \Omega)$ and let $Q \in \partial \Omega$ be such that $\rho(x) = |x-Q|$. If $\rho > 1/5$ then elliptic regularity implies $|Du(x)| \leq C(n, K)$. 

So we may assume that $\rho < 1/5$. A standard estimate yields  \begin{equation}\label{dubound}
|Du(x)| \leq \frac{C}{\rho}\fint_{\partial B(x, \rho)}|u(y)|d\sigma(y).
\end{equation}

We may pick $3\rho < \sigma < 5\rho$ such that $y\in \partial B(x,\rho) \Rightarrow y\in B(Q,\sigma)$. As $|u|$ is subharmonic and $\mathrm{dist}(y,\partial B(Q,\sigma)) > \sigma/3$ we may estimate $$|u(y)| \leq c(n)\int_{\partial B(Q, \sigma)} \frac{\sigma^2 - |y-Q|^2}{\sigma |y-z|^n}|u(z)|d\sigma(z) \leq c\fint_{\partial B(Q,\sigma)} |u(z)|d\sigma(z) \stackrel{\mathrm{Lem}\;\ref{boundedaverage}}{\leq} C\sigma \leq C'\rho.$$

This estimate, with \eqref{dubound}, implies the Lipschitz bound.
 \end{proof}
 
Consider any pseudo-blowup $Q_j \rightarrow Q, r_j \downarrow 0$. It is clear that $u_j$ is a Lipschitz function (though perhaps not uniformly in $j$). If $\phi \in C_c^\infty(B_1; \R^n)$ then Corollary \ref{blowupsforeveryone} implies (after a possible rotation) $$\int \phi \cdot \nabla u_{j}^{\pm} = -\int (\nabla \cdot \phi) u_{j}^{\pm} \stackrel{j\rightarrow \infty}{\rightarrow} -\int(\nabla \cdot \phi) (x_n)^{\pm} = \int \phi \cdot e_n \chi_{\mathbb H^{\pm}}.$$ 

Because $\nabla u_j^{\pm}$ converges in the weak-$*$ topology on $L^\infty(B_1; \R^n)$, $|\nabla u_j^{\pm}|$ is bounded in $L^\infty(B_1)$. Therefore, $|\nabla u_j^{\pm}|$ converges in the weak-$*$ topology on $L^\infty(B_1)$ to some function $f$. However, as $\nabla u_j^{\pm}\stackrel{*}{\rightharpoonup} e_n  \chi_{\mathbb H^{\pm}}$ it must be true that $|\nabla u_j^{\pm}|$ converges pointwise to $\chi_{\mathbb H^{\pm}}$ and thus $f= \chi_{\mathbb H^{\pm}}$ (more generally, converges to the indicator function of some half space which may depend on the blowup sequence taken). 
 
The existence of this weak-$*$ limit allows us to prove that $\Theta^{n-1}(\omega^{\pm},Q) \colonequals \lim_{r\downarrow 0}\frac{\omega^{\pm}(B(Q,r))}{r^{n-1}}$ exists, and is finite, everywhere on $\partial \Omega$ (as opposed to $\mathcal H^{n-1}$-almost everywhere). Let $r_j \downarrow 0$; one can compute that $J(Q,r_j) = \frac{\omega^{+}(B(Q,r_j))}{r_j^{n-1}}\frac{\omega^-(B(Q,r_j))}{r_j^{n-1}} J_{Q,r_j}(0,1)$ where $$J_{Q,r_j}(0,s) \colonequals  \frac{1}{s^2}\left(\int_{B(0,s)} \frac{|\nabla u_{j}^+(y)|^2}{|y|^{n-2}}dy\right)^{1/2}\left(\int_{B(0,s)} \frac{|\nabla u_{j}^-(y)|^2}{|y|^{n-2}}dy\right)^{1/2}$$ and $u_j$ is a blowup along the sequence $Q_j \equiv Q$ and $r_j \downarrow 0$. By the arguments above, $|\nabla u_j^{\pm}|^2$ converges in the weak-$*$ topology to the indicator function of some halfspace. Therefore, $J_{Q,r_j}(0,1) \stackrel{j\rightarrow \infty}{\rightarrow} c(n)$, where $c(n)$ is some constant independent of $r_j\downarrow 0$ (the halfspace may depend on the sequence, but the integral does not).   Furthermore, by Theorem \ref{monotonicity} $J(Q,0) \colonequals \lim_{r\downarrow 0} J(Q,r)$ exists. It follows that $$\lim_{r\downarrow 0} \frac{\omega^{+}(B(Q,r))}{r^{n-1}}\frac{\omega^-(B(Q,r))}{r^{n-1}} = \frac{J(Q, 0)}{c(n)}.$$ In particular, the limit on the left exists for every $Q\in \partial \Omega$, which (given Remark \ref{rnderivative}) implies $\Theta^{n-1}(\omega^{\pm},Q)$ exists for every $Q\in \partial \Omega$. 

\section{Non-degeneracy of $\Theta^{n-1}(\omega^\pm, Q)$}\label{sec: thetanondegenerate}

In this section we show  $\Theta^{n-1}(\omega^\pm, Q) >0$ for all $Q\in \partial \Omega$ (Proposition \ref{nondegeneracy}). Let  \begin{equation}\label{vq} v^{(Q)}(x) \colonequals h(Q)u^+(x) - u^-(x),\; Q\in \partial \Omega.\end{equation} For any $r_j \downarrow 0$, we define the blowup of $v^{(Q)}$ along $r_j$ to be $v^{(Q)}_j(x) \colonequals \frac{r_j^{n-2} v^{(Q)}(r_jx + Q)}{\omega^-(B(Q, r_j))}$. Let us make some remarks concerning $v^{(Q)}$ and its blowups. 
\begin{rem}\label{vobservations} The following hold for any $Q\in \partial \Omega$. 
\begin{itemize}
\item For any compact $K$, we have $\sup_{Q\in K\cap \partial \Omega} \|v^{(Q)}\|_{W_{loc}^{1,\infty}(\R^n)} < \infty$. 
\item  $v^{(Q)}_j(x) \rightarrow x\cdot e_n$ uniformly on compacta (after passing to a subsequence and a possible rotation). Additionally (as above), we have $|\nabla v_j^{(Q)}|\stackrel{*}{ \rightharpoonup} 1$ in $L^\infty$. 
\item If the non-tangential limit of $|\nabla v^{(Q)}|$ at $Q$ exists it is equal to $\Theta^{n-1}(\omega^-, Q)$. 
\end{itemize}
\end{rem}

\begin{proof}[Justification of Remarks]
The first two statements follow from the work in Section \ref{sec: u is lipschitz}.  

To prove the third statement we first notice  \begin{equation}\label{blowupv}\nabla v^{(Q)}_j(x) = \frac{r_j^{n-1} \nabla v^{(Q)}(r_jx+Q)}{\omega^-(B(Q, r_j))}.\end{equation} The second statement implies $\lim_{j\rightarrow \infty} |\nabla v^{(Q)}_j(x)| = 1$ almost everywhere. The result follows.
 \end{proof}


\subsection{Almgren's Frequency Formula}\label{formulasandcomputations} Remark \ref{vobservations} hints at a connection between the degeneracy of $\Theta^{n-1}(\omega^-, Q)$ and that of the non-tangential limit of $\nabla v^{(Q)}$.  This motivates the use of Almgren's frequency function (first introduced in \cite{almgren}). 

\begin{defin}\label{almgrenfrequency}
 Let $f \in H^1_{\mathrm{loc}}(\R^n)$ and pick $x_0\in \{f= 0\}$. Define $$H(r, x_0, f) = \int_{\partial B_r(x_0)} f^2,$$ $$D(r,x_0,f) = \int_{B_r(x_0)} |\nabla f|^2,$$ and finally $$N(r,x_0,f) = \frac{rD(r,x_0,f)}{H(r,x_0,f)}.$$
\end{defin}

Almgren first noticed that when $f$ is harmonic, $r\mapsto N(r, x_0, f)$ is absolutely continuous and monotonically decreasing as $r\downarrow 0$. Furthermore, $N(0, x_0, f)$ is an integer and is the order to which $f$ vanishes at $x_0$ (these facts first appear in \cite{almgren}. See \cite{almgrennotes} for proofs and a gentle introduction). 

Throughout the rest of this subsection we consider $v \equiv v^{(Q)}$ and, for ease of notation, set $Q = 0$. $v$ may not be harmonic and thus $N(r, 0, v)$ may not be monotonic. However, in the sense of distributions, the following holds: \begin{equation}\label{distributionderivative}
\Delta v(x) = (h(0)d\omega^+-d\omega^-)|_{\partial \Omega} = \left(\frac{h(0)}{h(x)} - 1\right)d\omega^-|_{\partial \Omega}.
\end{equation}

Therefore, $\log(h) \in C^{\alpha}(\partial \Omega)$ implies that $|\Delta v(x)| \leq C|x|^\alpha d\omega^-|_{\partial \Omega}$. That  $v$ is ``almost harmonic" will imply that $N$ is ``almost monotonic" (see Lemma \ref{almostmonotonic}). 

When estimating $N'(r, 0, v)$ we reach a technical difficulty; {\it a priori} $v$ is merely Lipschitz, and so $\nabla v$ is not defined everywhere. To address this, we will work instead with $v_\varepsilon = v*\varphi_\varepsilon$, where $\varphi$ is a $C^\infty$ approximation to the identity (i.e. $\mathrm{supp}\; \varphi \subset B_1$ and $\int \varphi = 1$). Let $N_\varepsilon(r) \colonequals N(r, 0, v_\varepsilon)$ and similarly define $H_\varepsilon, D_\varepsilon$. 

\begin{rem}\label{basicalmgrenfacts}
The following are true:
\begin{equation*}
\begin{aligned}
&\lim_{r\downarrow 0} N(r,0,v) = 1\\
&D_\varepsilon(r) = \int_{\partial B_r}v_\varepsilon(v_\varepsilon)_\nu d\sigma - \int_{B_r} v_\varepsilon \Delta v_\varepsilon\\
&\frac{d}{dr}D_\varepsilon(r) = \frac{n-2}{r}\int_{B_r}|\nabla v_\varepsilon|^2dx + 2\int_{\partial B_r} (v_\varepsilon)_\nu^2 - \frac{2}{r}\int_{B_r} \left\langle x, \nabla v_\varepsilon\right\rangle \Delta v_\varepsilon dx \\
&\frac{d}{dr}H_\varepsilon(r) = \frac{n-1}{r}H_\varepsilon(r) + 2\int_{\partial B_r} v_\varepsilon (v_\varepsilon)_\nu d\sigma.
\end{aligned}
\end{equation*}
\end{rem}

\begin{proof}
The second equation follows from integration by parts and the third (originally observed by Rellich) can be obtained using the change of variables $y = x/r$. The final equation can be proven in the same way as the third. 

To establish the first equality we take blowups. Pick any $r_j \downarrow 0$. One computes, $$N(r_j, 0, v) = \frac{\int_{B_1} |\nabla v_j|^2}{\int_{\partial B_1} v_j^2}.$$ Recall Remark \ref{vobservations}; $v_j \rightarrow x_n$ uniformly on compacta and $|\nabla v_j| \stackrel{*}{\rightharpoonup} 1$ in $L^\infty$ (perhaps passing to subsequences and rotating the coordinate system). Therefore, $\lim_{j\rightarrow \infty} N(r_j, 0,v) = \lim_{j\rightarrow \infty} N(1, 0, v_j) = N(1, 0, x_n)$. Almgren (in \cite{almgren}) proved that if $p$ is a 1-homogenous polynomial then $N(r, 0, p) \equiv 1$ for all $r$. It follows that $\lim_{j \rightarrow \infty} N(r_j, 0, v) = 1$.
\end{proof}

With these facts in mind we  calculate $N'_\varepsilon(r)$. 

\begin{equation}\label{derivativeofN}\begin{aligned}
H_\varepsilon^2(r)N'_{\varepsilon}(r) &= 2r\left(\int_{\partial B_r}(v_\varepsilon)_\nu^2d\sigma \int_{\partial B_r} v_\varepsilon^2d\sigma -\left[\int_{\partial B_r} v_\varepsilon (v_\varepsilon)_\nu d\sigma\right]^2\right)\\ &+ 2r\int_{B_r} v_\varepsilon \Delta v_\varepsilon dx \int_{\partial B_r} v_\varepsilon (v_\varepsilon)_\nu d\sigma -  2H_\varepsilon(r) \int_{B_r} \left\langle x, \nabla v_\varepsilon\right\rangle \Delta v_\varepsilon dx
\end{aligned}
\end{equation}

\begin{proof}[Derivation of \eqref{derivativeofN}]
By the quotient rule
$$H_\varepsilon^2(r)N'_{\varepsilon}(r) = D_\varepsilon(r)H_\varepsilon(r) + rD'_\varepsilon(r)H_\varepsilon(r) - rD_\varepsilon(r)H'_\varepsilon(r).$$ 

Using the formulae for $H'_\varepsilon, D'_\varepsilon$ found in Remark \ref{basicalmgrenfacts} we rewrite the above as $$H_\varepsilon^2(r)N'_{\varepsilon}(r) = D_\varepsilon(r)H_\varepsilon(r) + rH_\varepsilon(r)\left(\frac{n-2}{r}\int_{B_r}|\nabla v_\varepsilon|^2dx + 2\int_{\partial B_r} (v_\varepsilon)_\nu^2 - \frac{2}{r}\int_{B_r} \left\langle x, \nabla v_\varepsilon\right\rangle \Delta v_\varepsilon dx\right)$$$$ - rD_\varepsilon(r)\left(\frac{n-1}{r}H_\varepsilon(r) + 2\int_{\partial B_r} v_\varepsilon (v_\varepsilon)_\nu d\sigma\right).$$

Distribute and combine terms to get $$H_\varepsilon^2(r)N'_{\varepsilon}(r) = \left(D_\varepsilon(r)H_\varepsilon(r) + (n-2)H_\varepsilon(r) \int_{B_r}|\nabla v_\varepsilon|^2dx- (n-1)D_\varepsilon(r)H_\varepsilon(r)\right)$$$$+2r\left(H_\varepsilon(r)\int_{\partial B_r} (v_\varepsilon)_\nu^2d\sigma - D_\varepsilon(r)\int_{\partial B_r}v_\varepsilon (v_\varepsilon)_\nu d\sigma\right)-2H_\varepsilon(r)\int_{B_r}\left\langle x, \nabla v_\varepsilon\right\rangle \Delta v_\varepsilon dx.$$

The first set of parenthesis above is equal to zero (recalling the definition of $D_\varepsilon(r)$). In the second set of parenthesis use the formula for $D_\varepsilon(r)$ found in Remark \ref{basicalmgrenfacts}. This gives us $$H_\varepsilon^2(r)N'_{\varepsilon}(r) = 2r\left(H_\varepsilon(r)\int_{\partial B_r}(v_\varepsilon)_\nu^2 d\sigma - \left(\int_{\partial B_r}(v_\varepsilon)_\nu v_\varepsilon d\sigma\right)^2\right)$$$$ + 2r\int_{B_r} v_\varepsilon \Delta v_\varepsilon dx \int_{\partial B_r} v_\varepsilon (v_\varepsilon)_\nu d\sigma -  2H_\varepsilon(r) \int_{B_r} \left\langle x, \nabla v_\varepsilon\right\rangle \Delta v_\varepsilon dx.$$
\end{proof}

 The difference in parenthesis on the right hand side of \eqref{derivativeofN} is positive by the Cauchy-Schwartz inequality. Thus, to establish a lower bound on $N_\varepsilon'(r)$, it suffices to consider the other terms in the equation. 
 
\begin{lem}\label{rellichsidentity}
Let $\varepsilon < r$ and  define $E_\varepsilon(r) = \int_{B_r} \left\langle x, \nabla v_\varepsilon\right\rangle \Delta v_\varepsilon dx$. Then there exists a constant $C$ (independent of $r, \varepsilon$) such that $ |E_\varepsilon(r)|  \leq C r^{1+\alpha} \omega^-(B(0,r))$.
\end{lem}

\begin{proof}
Since $\Delta v_\varepsilon = (\Delta v)*\varphi_\varepsilon$ in terms of distributions, we can move the convolution from one term to the other:  $$\int_{B_r} \left\langle x, \nabla v_\varepsilon\right\rangle \Delta v_\varepsilon dx = \int  [(\chi_{B_r}(x)\left\langle x, \nabla v_\varepsilon \right\rangle)*\varphi_\varepsilon] \Delta  vdx.$$ Evaluate $\Delta v$, as in \eqref{distributionderivative}, to obtain $$\left|\int_{B_r} \left\langle x, \nabla v_\varepsilon\right\rangle \Delta v_\varepsilon dx \right| = \left|\int (\chi_{B_r}(x)\left\langle x, \nabla v_\varepsilon \right\rangle)_\varepsilon \left(\frac{h(0)}{h(x)}-1\right)d\omega^-\right|$$$$\leq C r^{1+\alpha} \int_{B_{r+\varepsilon}} (|\nabla v|_\varepsilon)_\varepsilon d\omega^-,$$ where the last inequality follows from $\log(h) \in C^\alpha$, and $|x| < C(r + \varepsilon) < Cr$ on the domain of integration. The desired estimate then follows from the Lipschitz continuity of $v$ and that the harmonic measure of an NTA domain is doubling (see \cite{jerisonandkenig}, Theorem 2.7).
\end{proof}

\begin{lem}\label{lowerboundonH}
Let $\varepsilon << r$. Then $H_\varepsilon(r) > c\frac{\omega^-(B(0,r))^2}{r^{n-3}}$ for some constant $c > 0$ independent of $r, \varepsilon > 0$. 
\end{lem}

\begin{proof}
By the corkscrew condition (see Definition \ref{ntadomain} condition (1)) on $\Omega$, there is a point $x_0 \in \partial B_r \cap \Omega$ such that $\mathrm{dist}(x_0, \partial \Omega) > cr$ ($c$ depends only on the NTA properties of $\Omega$). The Harnack chain condition (see Definition \ref{ntadomain} condition (3)) gives $v(x_0)\sim v(A_r(0))$.  The Harnack inequality then implies that, for $\varepsilon << r$ there is a universal $k$ such that for $y\in B(x_0, kr)$ we have $v_\varepsilon(y) \sim v(x_0) \sim v(A_r(0))$. 

Therefore, there is a subset of $\partial B_r$ (with surface measure $\approx k|\partial B_r|$) on which $v_\varepsilon$ is proportional to $v(A_r(0))$.  We then recall that in an NTA domain we have $v(A_r(0)) \sim \frac{\omega^-(B(0,r))}{r^{n-2}}$ (\cite{jerisonandkenig}, Lemma 4.8), which proves the desired result.
\end{proof}

It is useful now to establish bounds on the growth rate of $\omega^{\pm}(B(Q,r))$. As $\Omega$ is vanishing Reifenberg flat, $\omega^{\pm}$ is asymptotically optimally doubling (\cite{kenigtoroduke}, Corollary 4.1). This implies a key estimate:  for any $\delta > 0$ and $Q\in \partial \Omega$ we have \begin{equation}\label{slowerthanpoly}\lim_{r\downarrow 0} \frac{r^{n-1+\delta}}{\omega^-(B(Q,r))} = 0.\end{equation}

\begin{lem}\label{almostmonotonic}
Let $\varepsilon << R$. There exists a function, $C(R, \varepsilon)$, such that 
\begin{eqnarray} 
\forall R/4 < r < R,\; N_\varepsilon(R) + C(R, \varepsilon)(R-r) &\geq& N_\varepsilon(r)\\
C(R,\varepsilon)R &\leq& kR^{\alpha/2}
\end{eqnarray}
where $k > 0$ is a constant independent of $\varepsilon, R$ (as long as $\varepsilon << R$). 
\end{lem}

\begin{proof}
If $C(R, \varepsilon) \colonequals \sup_{R/4 < r < R} (N_\varepsilon(r)')^-$, the first claim of our lemma is true by definition.

Recall \eqref{derivativeofN}: $$H_\varepsilon^2(r)N'_{\varepsilon}(r) = 2r\left(\int_{\partial B_r}(v_\varepsilon)_\nu^2d\sigma \int_{\partial B_r} v_\varepsilon^2d\sigma -\left[\int_{\partial B_r} v_\varepsilon (v_\varepsilon)_\nu d\sigma\right]^2\right) + 2r\int_{B_r} v_\varepsilon \Delta v_\varepsilon dx \int_{\partial B_r} v_\varepsilon (v_\varepsilon)_\nu d\sigma$$$$ -2H_\varepsilon(r) E_\varepsilon(r).$$

As mentioned above, the difference in parenthesis is positive by the Cauchy-Schwartz inequality. Therefore $$(N'_{\varepsilon}(r))^- \leq 2\left|\frac{E_\varepsilon(r)}{H_{\varepsilon}(r)}\right| + \left|\frac{2r\int_{B_r} v_\varepsilon \Delta v_\varepsilon dx \int_{\partial B_r} v_{\varepsilon} (v_{\varepsilon})_\nu d \sigma }{H_{\varepsilon}(r)^2}\right|.$$ 

{\bf (A) Estimating $2r\int_{B_r} v_\varepsilon \Delta v_\varepsilon dx \int_{\partial B_r} v_{\varepsilon} (v_{\varepsilon})_\nu d \sigma$}:  On $\partial B_r, |(v_\varepsilon)_\nu| < C, |v_\varepsilon| < Cr$ by Lipschitz continuity. Therefore, arguing as in Lemma \ref{rellichsidentity}, we can estimate $$\left|2r\int_{B_r} v_\varepsilon \Delta v_\varepsilon dx \int_{\partial B_r} v_{\varepsilon} (v_{\varepsilon})_\nu d\sigma\right| \leq Cr^{n+1} \int_{B_{r+\varepsilon}\cap \partial \Omega} |v_\varepsilon|_\varepsilon \left(\frac{h(0)}{h(x)} - 1\right)d\omega^-$$$$\leq Cr^{n+\alpha+2}\omega^-(B(0,r)),$$ where the last inequality follows from $|v_\varepsilon|_\varepsilon \leq C\varepsilon < Cr$ on $\partial \Omega$ (by Lipschitz continuity). 

From Lemma \ref{lowerboundonH} it follows that $$\left|\frac{2r\int_{B_r} v_\varepsilon \Delta v_\varepsilon dx \int_{\partial B_r} v_{\varepsilon} (v_{\varepsilon})_\nu d \sigma }{H_{\varepsilon}(r)^2}\right| \leq \frac{Cr^{n+\alpha+2}r^{2n-6}}{\omega^-(B(0,r))^3} = C \left(\frac{r^{n-1 + \alpha/6}}{\omega^-(B(0,r))}\right)^3 r^{\alpha/2-1}.$$

{\bf (B) Estimating $2\left|\frac{E_\varepsilon(r)}{H_{\varepsilon}(r)}\right|$}: Lemma \ref{rellichsidentity} and Lemma \ref{lowerboundonH} imply $$2\left|\frac{E_\varepsilon(r)}{H_{\varepsilon}(r)}\right| \leq C\frac{r^{n-2+\alpha}}{\omega^-(B(0,r))} = C\left(\frac{r^{n-1 + \alpha/2}}{\omega^-(B(0,r))}\right)r^{\alpha/2-1}.$$

From \eqref{slowerthanpoly} we can conclude $$\frac{R^{n-1+\alpha/2}}{\omega^-(B(0,R))}, \frac{R^{3n-3+\alpha/2}}{\omega^-(B(0,R))^3}\stackrel{R\downarrow 0}{\rightarrow} 0.$$ Combine the estimates in (A) and (B)  to conclude that $C(\varepsilon, R)R \leq o_R(1)R^{\alpha/2}$. \end{proof}

 We can now prove a lower bound on the size of $N_\varepsilon(r)$ for small $r$. 

\begin{cor}\label{growthrateforW}
$\limsup_{\varepsilon \downarrow 0} \frac{1}{r}(N_\varepsilon(r) -1)  > -Cr^{\alpha/2-1}$. 
\end{cor}

\begin{proof}
As $\lim_{s\downarrow 0} N(s) = 1$ there is some $r' << r$ such that $|N(r') - 1| < Cr^{\alpha/2}$. Now pick $\varepsilon <<r'$ small enough that Lemma \ref{almostmonotonic} applies for $\varepsilon$ and all $r' < R < r$ and such that $|N_\varepsilon(r') - N(r')| < Cr^{\alpha/2}$(recall $N_\varepsilon(\rho) \rightarrow N(\rho)$ for fixed $\rho$ as $\varepsilon \downarrow 0$). 

Let $j$ be such that $2^{-j}r < r' < 2^{-j+1}r$. Then $$N_\varepsilon(r) - N_\varepsilon(r') \geq \sum_{\ell=0}^{j-2} (N_\varepsilon(2^{-\ell}r)-N_\varepsilon(2^{-\ell-1})r) + N_\varepsilon(2^{-j+1}r) - N_\varepsilon(r') \geq $$$$-C(2^{-j+1}r, \varepsilon)(2^{-j+1}r-r')-\frac{1}{2}\sum_{\ell=0}^{j-2} C(2^{-\ell}r, \varepsilon)2^{-\ell}r \stackrel{\mathrm{Lem} \ref{almostmonotonic}}{\geq}-kr^{\alpha/2} \sum_{\ell=0}^{j-1}(2^{-\ell})^{\alpha/2} \geq -C_\alpha r^{\alpha/2}.$$

Combining all the inequalities above we have that $N_\varepsilon(r) - 1 > -C r^{\alpha/2}$ for small $\varepsilon > 0$.
\end{proof}

\subsection{Monneau Monotonicity and Non-degeneracy}Our main tool here will be the Monneau potential, defined for $f\in H^1_{\mathrm{loc}}(\R^n)$ and $p\in C^{\infty}(\R^n)$, \begin{equation}M^{x_0}(r, f, p)\colonequals \frac{1}{r^{n+1}}\int_{\partial B_r} (f(x+ x_0)-p)^2d\sigma(x).\end{equation} Monneau, \cite{monneau}, observed that if $f$ is a harmonic function vanishing to first order at $x_0$ and $p$ is a 1-homogenous polynomial then $M^{x_0}$ is monotonically decreasing as $r\downarrow 0$.

We follow closely the methods of Garofalo and Petrosyan (\cite{gandp}, see specifically Sections 1.4-1.5) who studied issues of  non-degeneracy in an obstacle problem. Their program, which we adapt to our circumstances, has two steps: first relate the growth of the Monneau potential to the growth of Almgren's frequency function. Second, use this relation to establish lower bounds on the growth of $M$ and the existence of a limit at zero for $M$. As before, $v \equiv v^{(Q)}$ and without loss of generality, $Q= 0 \in \partial \Omega$.  Additionally, $p$ will always be a 1-homogenous polynomial. We drop the dependence of $M$ on $Q$ and $v$ when no confusion is possible. Again $v_\varepsilon = v*\varphi_\varepsilon$, where $\varphi$ is an approximation to the identity. Naturally, $M_\varepsilon(r,p) \colonequals M^{0}(r, v_{\varepsilon}, p)$. 

First we derive equations \eqref{derivativeofM} and \eqref{rewritingW}. 

\begin{equation}\label{derivativeofM}
M'_\varepsilon(r,p) = \frac{2}{r^{n+2}} \int_{\partial B_r} (v_\varepsilon - p)(x\cdot \nabla (v_\varepsilon - p) - (v_\varepsilon -p))d\sigma.
\end{equation}

\begin{proof}[Derivation of \eqref{derivativeofM}]
Let $x = ry$ so that $M_\varepsilon(r,p) = \int_{\partial B_1} \left(\frac{v_\varepsilon(ry)}{r} - \frac{p(ry)}{r}\right)^2 d\sigma(y)$. Differentiating under the integral gives $$M'_\varepsilon(r,p) = \int_{\partial B_1} 2\left(\frac{v_\varepsilon(ry)}{r} - \frac{p(ry)}{r}\right)\left(\frac{y}{r}\cdot \nabla_{x} [v_\varepsilon(ry) - p(ry)] - \frac{1}{r^2}(v_\varepsilon(ry) - p(ry))\right)d\sigma(y).$$

Changing back to $x$ we have that $$M'_\varepsilon(r,p) = \frac{2}{r^{n+2}} \int_{\partial B_r}(v_\varepsilon - p)(x\cdot \nabla (v_\varepsilon - p) - (v_\varepsilon -p))d\sigma(x).$$
\end{proof}

Next we establish a relation between the derivative of $M$ and the growth rate of $N$ (we emphasize that \eqref{rewritingW} is true only when $p$ is a 1-homogenous polynomial). 

\begin{equation}\label{rewritingW}
\frac{H_\varepsilon(r)}{r^{n+1}}\left(N_\varepsilon(r) - 1\right) = -\frac{1}{r^n} \int_{B_r}(v_\varepsilon - p)\Delta v_\varepsilon dx + rM'_\varepsilon(r,p)/2
\end{equation}

\begin{proof}[Derivation of \eqref{rewritingW}]
Recall for all 1-homogenous polynomials $p$ we have $N(r,x_0, p) \equiv 1$. We ``add zero" and distribute to rewrite $$\frac{H_\varepsilon(r)}{r^{n+1}}\left(N_\varepsilon(r) - 1\right) =  \frac{1}{r^n} \int_{B_r} |\nabla(v_\varepsilon - p)|^2 + 2\nabla v_\varepsilon \cdot \nabla p dx - \frac{1}{r^{n+1}} \int_{\partial B_r} (v_\varepsilon - p)^2 + 2v_\varepsilon p d\sigma.$$ 

Transform the first integral on the right hand side using integration by parts,  $$\frac{H_\varepsilon(r)}{r^{n+1}}\left(N_\varepsilon(r) - 1\right)= \frac{1}{r^n} \int_{\partial B_r} \frac{x}{r}\cdot \nabla (v_\varepsilon - p)(v_\varepsilon - p) + 2\left(\frac{x}{r}\cdot \nabla p\right)v_\varepsilon $$$$- \frac{1}{r^n}\int_{B_r} (v_\varepsilon - p) \Delta(v_\varepsilon -p) + 2v_\varepsilon \Delta p dx - \frac{1}{r^{n+1}} \int_{\partial B_r}(v_\varepsilon - p)^2 + 2v_\varepsilon p d\sigma.$$

As $p$ is a 1-homogenous polynomial, $\Delta p = 0$ and $x\cdot \nabla p - p = 0$. The above simplifies to $$\frac{H_\varepsilon(r)}{r^{n+1}}\left(N_\varepsilon(r) - 1\right) = -\frac{1}{r^n} \int_{B_r}(v_\varepsilon - p)\Delta v_{\varepsilon} + \frac{1}{r^{n+1}} \int_{\partial B_r}( x \cdot \nabla( v_\varepsilon - p)-(v_\varepsilon - p))(v_\varepsilon - p)d\sigma.$$ In light of \eqref{derivativeofM}, we are finished.  
\end{proof}

The above two equations, along with Corollary \ref{growthrateforW}, allow us to control the growth of $M$ from below. 

\begin{lem}\label{growthofM}
Let $p$ be any 1-homogenous polynomial. Then for any $R> 0$ there exists a constant $C$ (independent of $R$ and $p$) such that $$M(R, p) - M(r, p) \geq -(C+C\|p\|_{L^\infty(\partial B_1)})R^{\alpha/2}$$ for any $r\in [R/4, R]$. 
\end{lem}

\begin{proof}
Recall \eqref{rewritingW}, $$rM'_\varepsilon(r, p)/2 = \frac{H_\varepsilon(r)}{r^{n+1}}\left(N_\varepsilon(r) - 1\right) + \frac{1}{r^n} \int_{B_r}(v_\varepsilon - p)\Delta v_{\varepsilon}.$$ Consider first the integral on the right hand side and argue as before to estimate, $$\left|\frac{1}{r^n} \int_{B_r}(v_\varepsilon - p)\Delta v_{\varepsilon}\right| \leq \frac{1}{r^n} \int_{\partial \Omega \cap B_{r+\varepsilon}}|v_\varepsilon - p|_\varepsilon \left(\frac{h(x)}{h(0)} - 1\right)d\omega^- \leq C(1+\|p\|_{L^\infty(\partial B_1)})\frac{\omega^-(B(0,r))r^{1+\alpha}}{r^{n}},$$ where $|v_\varepsilon| < Cr$ on $\partial \Omega$ because $v$ is Lipschitz and $|p(x)| \leq C\|p\|_{L^\infty(\partial B_1)}r$ because $p$ is 1-homogenous.  By Corollary \ref{upperboundondensity}, $\frac{\omega^-(B(Q,r))}{r^{n-1}}$ is bounded uniformly in $r < 1$ and in $Q\in \partial \Omega$ on compacta. Therefore, $|\frac{1}{r^n} \int_{B_r}(v_\varepsilon - p)\Delta v_{\varepsilon}| \leq C(1+\|p\|_{L^\infty(\partial B_1)}) r^{\alpha}$. 

Returning to \eqref{rewritingW}, $$\limsup_{\varepsilon \downarrow 0} \sup_{R/4 < r < R} (M_\varepsilon(r,p)')^- \leq C(1+\|p\|_{L^\infty(\partial B_1)}) R^{\alpha-1} + \limsup_{\varepsilon \downarrow 0}\sup_{R/4 < r < R} \frac{1}{r}(N_\varepsilon(r) - 1).$$  The bounds on the growth of $N$ (Corollary \ref{growthrateforW}) imply $$\limsup_{\varepsilon \downarrow 0} \sup_{R/4 < r < R} (M_\varepsilon(r,p)')^- \leq (C+C\|p\|_{L^\infty(\partial B_1)})R^{\alpha/2-1},$$ which is equivalent to the desired result.  
\end{proof}

When it is not relevant to the analysis (e.g. in the proofs of Lemma \ref{limitofM} and Proposition \ref{nondegeneracy} below), we omit the dependence of the constant in Lemma \ref{growthofM} on $\|p\|_{L^\infty(\partial B_1)}$.  

\begin{lem}\label{limitofM}
Let $p$ be any 1-homogenous polynomial. Then $M(0,p) \colonequals \lim_{r\downarrow 0}M(r, p)$ exists. 
\end{lem}

\begin{proof}
Let $a \colonequals \limsup_{r\downarrow 0} M(r,p)$. That $a < \infty$ follows from Lemma \ref{growthofM}, applied iteratively (as $r^{\alpha/2-1}$ is integrable at zero). We claim that there exists a constant $C < \infty$ such that $M(r,p) - a > -Cr^{\alpha/2}$ for any $0 < r \leq 1$.  

On the other hand, $a- M(r,p) > -o(1)$ as $r \downarrow 0$ by the definition of $\limsup$. This, with the claim above, implies that $\lim_{r\downarrow 0} M(r, p) = a$.

Let us now address the claim: take $r_0 < r$. Let $k$ be such that $2^{-k}r \geq r_0 \geq 2^{-k-1}r$. Then, by Lemma \ref{growthofM}, we have $$M(r, p) - M(r_0, p) = \sum_{\ell = 0}^{k-1} (M(2^{-\ell}r, p) - M(2^{-\ell - 1}r, p)) + M(2^{-k}r, p) - M(r_0, p)$$$$ \geq -Cr^{\alpha/2} \sum_{\ell = 0}^{\infty} (2^{\alpha/2})^{-\ell} \geq -C_\alpha r^{\alpha/2}.$$ The claim follows if we pick $r_0$ small so that $M(r_0,p)$ is arbitrarily close to $a$.
\end{proof}

Finally, we can establish the pointwise non-degeneracy of $\Theta^{n-1}(\omega^\pm, Q)$. 

\begin{prop}\label{nondegeneracy}
For all $Q \in \partial \Omega$ we have $\Theta^{n-1}(\omega^{\pm}, Q) > 0$. 
\end{prop}

\begin{proof}
It suffices to assume $Q = 0$ and to prove $\Theta^{n-1}(\omega^-, 0) > 0$. 

We proceed by contradiction. Pick some $r_j \downarrow 0$ so that $v_j \rightarrow p$ uniformly on compacta (where $p$ is a 1-homogenous polynomial given by Corollary \ref{blowupsforeveryone}).  Lemma \ref{limitofM} implies $$M(0, p) = \lim_{j\rightarrow \infty} M(r_j, p) = \lim_{j\rightarrow \infty}\int_{\partial B_1} \left(v_j(x)\frac{\omega^-(B(0,r_j))}{r_j^{n-1}} - \frac{p(r_j x)}{r_j}\right)^2 d\sigma(x).$$ As  $\frac{p(r_jx)}{r_j}= p(x)$ and $\Theta^{n-1}(\omega^-, 0) = 0$, by assumption, we can conclude $M(0,p) = \int_{\partial B_1} p^2d\sigma$.

For any $j$, the homogeneity of $p$ implies $$M(r_j, p) -M(0,p) = \frac{1}{r_j^{n+1}}\int_{\partial B_{r_j}} (v - p)^2 - \int_{\partial B_1} p^2= \int_{\partial B_1}\left(v_j(y)\frac{\omega^-(B(0,r_j))}{r_j^{n-1}} - p\right)^2-p^2d\sigma$$$$=\int_{\partial B_1} \left(v_j(y)\frac{\omega^-(B(0,r_j))}{r_j^{n-1}}\right)^2 - 2v_j(y)\frac{\omega^-(B(0,r_j))}{r_j^{n-1}}p(y)d\sigma \geq -Cr_j^{\alpha/2},$$ where the last inequality follows from iterating Lemma \ref{growthofM} (as in the proof of Lemma \ref{limitofM}). 

Rewrite the above equation as $$\frac{\omega^-(B(0,r_j))}{r_j^{n-1}}\int_{\partial B_1} v_j(y)^2 \frac{\omega^-(B(0,r_j))}{r_j^{n-1}} - 2v_j(y)p(y)d\sigma \geq -Cr_j^{\alpha/2}.$$ Divide by $\omega^-(B(0,r_j))/r_j^{n-1}$ and let $j\rightarrow \infty$. By \eqref{slowerthanpoly} the right hand side vanishes and, by assumption, $\omega^-(B(0,r_j))/r_j^{n-1}\rightarrow 0$. In the limit we obtain $-2\int_{\partial B_1} p^2 \geq 0$, a contradiction. \end{proof}

At this point we have proven that $\infty > \Theta^{n-1}(\omega^-, Q) > 0$ everywhere on $\partial \Omega$ and that $\Theta^{n-1}(\omega^-, Q)$ is bounded uniformly from above on compacta. Using standard tools from geometric measure theory this implies, for all dimensions, the decomposition mentioned in the introduction (for $n=2$): $\partial \Omega = \Gamma \cup N$, where $\omega^{\pm}(N) = 0$ and $\Gamma$ is a $(n-1)$-rectifiable set with $\sigma$-finite $\mathcal H^{n-1}$ measure. 

\section{Uniform non-degeneracy and initial regularity}\label{sec: c1domain}

\subsection{$\Theta^{n-1}(\omega^{\pm}, Q)$ is bounded uniformly away from $0$.}

In order to establish greater regularity for $\partial \Omega$ we need a uniform lower bound. Again the method of Garofalo and Petrosyan (\cite{gandp}, specifically Theorems 1.5.4 and 1.5.5) guides us. Our first step is to show that there is a unique tangent plane at every point. 

\begin{lem}\label{uniqueblowup}
For each $Q \in \partial \Omega$ there exists a unique 1-homogenous polynomial, $p^{Q}$, such that for any $r_j \downarrow 0$ we have $v_j \rightarrow p^{Q}$ uniformly on compacta (i.e. the limit described in Corollary  \ref{blowupsforeveryone} is unique).
\end{lem}

\begin{proof}
We prove it for $Q = 0$. Pick $r_j \downarrow 0$ so that $v_{r_j} \rightarrow p$ uniformly on compacta for some 1-homogenous polynomial $p$. Let $\tilde{r}_j \downarrow 0$ be another sequence so that $v_{\tilde{r}_j} \rightarrow \tilde{p}$, where $\tilde{p}$ is also a 1-homogenous polynomial. 

By Lemma \ref{limitofM}, $M(0, \Theta^{n-1}(\omega^-, 0)p)$ exists. Therefore, \begin{equation}\label{computationforlimitat0}\begin{aligned}M(0, \Theta^{n-1}(\omega^-, 0)p) =& \lim_{j\rightarrow \infty} M(r_j, \Theta^{n-1}(\omega^-,0)p)\\ =& \lim_{j\rightarrow \infty} \int_{\partial B_1}\left(\frac{\omega^-(B(0,r_j))}{r_j^{n-1}}v_{r_j}(x) - \Theta^{n-1}(\omega^-,0)p\right)^2d\sigma\\ =& 0.\end{aligned}\end{equation} The last equality above follows by the dominated convergence theorem and that $v_{r_j}\rightarrow p$. 

Similarly, \begin{equation*}\begin{aligned} M(0, \Theta^{n-1}(\omega^-, 0)p) =& \lim_{j\rightarrow \infty} M(\tilde{r}_j, \Theta^{n-1}(\omega^-, 0)p)\\ =& \lim_{j\rightarrow \infty} \int_{\partial B_1}\left(\frac{\omega^-(B(0,\tilde{r}_j))}{\tilde{r}_j^{n-1}}v_{\tilde{r}_j}(x) - \Theta^{n-1}(\omega^-,0)p\right)^2 \\ =& (\Theta^{n-1}(\omega^-, 0))^2 \int_{\partial B_1} (\tilde{p} - p)^2 d\sigma.\end{aligned}\end{equation*} Again the last equality follows by dominated convergence theorem and that $v_{\tilde{r}_j}\rightarrow \tilde{p}$.  As $\Theta^{n-1}(\omega^-, 0) > 0$ (Proposition \ref{nondegeneracy}), we have $p = \tilde{p}$.
\end{proof}

We should note that Lemma \ref{limitofM} (the existence of a limit at 0) and Lemma \ref{growthofM} (estimates on the derivatives of $M$) both hold for $M^{Q}(r, v^{(Q)}, p)$ where $p$ is any 1-homogenous polynomial and (as before) $v^{(Q)}(y)= h(Q)u^+(y) - u^-(y)$. Furthermore the constants in Lemma \ref{growthofM} are uniform for $Q$ in a compact set.
We now prove the main result of this subsection.

\begin{prop}\label{continuityofblowup}
The function $Q \mapsto \tilde{p}^{Q} \colonequals \Theta^{n-1}(\omega^-, Q)p^{Q}$ is a continuous function from $\partial \Omega \rightarrow C(\R^n)$. 
\end{prop}

\begin{proof}
As $\tilde{p}^{Q}$ is a 1-homogenous polynomial, it suffices to show that $Q \mapsto \tilde{p}^{Q}$ is a continuous function from $\partial \Omega \rightarrow L^2(\partial B_1)$. 

Pick $\varepsilon > 0$ and $Q \in \partial \Omega$. Equation \ref{computationforlimitat0} implies that $M^{Q}(0,v^{(Q)}, \tilde{p}^{Q}) = 0$. In particular, there is a $r_\varepsilon > 0$ such that if $r \leq r_\varepsilon$ then $M^{Q}(r, v^{(Q)}, \tilde{p}^{Q}) < \varepsilon$. Shrink $r_\varepsilon$ so that $r_\varepsilon^{\alpha/2} < \varepsilon$. 

$v^{(Q)}\in W^{1,\infty}_{loc}(\R^n)$ (uniformly for  $Q$ in a compact set) and $h\in C^\alpha(\partial \Omega)$, so there exists a $\delta = \delta(r_\varepsilon, \varepsilon) > 0$ such that for all $P\in B_\delta(Q)$ and $x\in B_1(0)$ we have \begin{equation}\label{vqnearvp} |v^{(Q)}(x+ Q) - v^{(P)}(x+P)| < \varepsilon r_\varepsilon.\end{equation}

Since $\sup_{P \in B_\delta(Q)} \|v^{(P)}(-+P)\|_{L^\infty(\partial B_{r_\varepsilon})} < r_\varepsilon$, \eqref{vqnearvp} immediately implies that $$\left|M^{Q}(r_\varepsilon,v^{(Q)}, \tilde{p}^{Q}) - \frac{1}{r_\varepsilon^{n+1}}\int_{\partial B_{r_\varepsilon}} (v^{(P)}(x+P) - \tilde{p}^{Q})^2\right| < C\varepsilon,\; \forall P \in B_\delta(Q).$$ By definition, $M^{Q}(r_\varepsilon,v^{(Q)}, \tilde{p}^{Q}) < \varepsilon$, so it follows that
$$M^P(r_\varepsilon, v^{(P)}, \tilde{p}^{Q}) \equiv \frac{1}{r_\varepsilon^{n+1}}\int_{\partial B_{r_\varepsilon}} (v^{(P)}(x+P) - \tilde{p}^{Q})^2 < C\varepsilon,\; \forall P \in B_\delta(Q).$$  

Repeated application of Lemma \ref{growthofM} yields, $$M^{P}(r_\varepsilon, v^{(P)}, \tilde{p}^{Q}) - M^{P}(0, v^{(P)}, \tilde{p}^{Q}) > -(C+C\|\tilde{p}^{Q}\|_{L^\infty(\partial B_1)})r_\varepsilon^{\alpha/2},\; \forall P\in B_\delta(Q)\Rightarrow$$$$C\varepsilon > M^P(0,v^{(P)}, \tilde{p}^Q)= \int_{\partial B_1} (\tilde{p}^{P} - \tilde{p}^Q)^2,\; \forall P\in B_\delta(Q).$$ That the first line implies the second follows from $\|\tilde{p}^{Q}\|_{L^\infty(\partial B_1)} = \Theta^{n-1}(\omega^-, Q) < C$ uniformly on compacta, $r_\varepsilon^{\alpha/2} < \varepsilon$ and $M^{P}(r_\varepsilon, v^{(P)}, \tilde{p}^{Q}) < C\varepsilon$. The equality in the second line follows from the standard blowup argument (see the proof of Lemma \ref{uniqueblowup}) and allows us to conclude that $Q \mapsto \tilde{p}^{Q}$ is continuous from $\partial \Omega \rightarrow L^2(\partial B_1)$. \end{proof}

\begin{cor}\label{continuityofdensity}
The function $Q \mapsto \Theta^{n-1}(\omega^-, Q)$ is continuous. Additionally, the function $Q \mapsto \{p^{Q} = 0\}$ is continuous (from $\partial \Omega$ to $\mathrm{G}(n, n-1)$). 
\end{cor}

\begin{proof}
Clearly the first claim, combined with Proposition \ref{continuityofblowup}, implies the second. 

For $Q_1, Q_2\in \partial \Omega$, if $P_1 = \{p^{Q_1} = 0\}, P_2 = \{p^{Q_2}=0\}$ are distinct hyperplanes with normals $\hat{n}_1, \hat{n}_2$, then both $(\hat{n}_1+\hat{n}_2)^\perp$ and $(\hat{n}_1- \hat{n}_2)^\perp$ consist of points equidistant from $P_1$ and $P_2$. Elementary geometry then shows that there is some constant $c > 0$ such that $$\max \{D[(\hat{n}_1+\hat{n}_2)^\perp\cap B_1(0), P_1\cap B_1(0)], D[(\hat{n}_1-\hat{n}_2)^\perp\cap B_1(0), P_1\cap B_1(0)]\} \geq c.$$ Let $P_3(Q_1,Q_2)$ be the plane which achieves this maximum.  If $P_1 = P_2$ then pick $P_3(Q_1,Q_2)$ to be any hyperplane such that $D[P_1\cap B_1(0), P_3\cap B_1(0)] \geq c$. 

 Recall Corollary \ref{blowupsforeveryone}, which implies that $p^{Q}$ is a monic 1-homogenous polynomial for all $Q\in \partial \Omega$. So if $y \in P_3(Q_1,Q_2)\cap \partial B_1(0)$, there is an universal $\tilde{c} > 0$ such that $\tilde{c} < |p^{Q_1}(y)| = |p^{Q_2}(y)|$. 

Therefore, $$\begin{aligned}\|\tilde{p}^{Q_1}- \tilde{p}^{Q_2}\|_{L^\infty(\partial B_1)} \geq&  |\Theta^{n-1}(\omega^-, Q_1)p^{Q_1}(y) -  \Theta^{n-1}(\omega^-, Q_2)p^{Q_2}(y)|\\ \geq& \tilde{c} |\Theta^{n-1}(\omega^-, Q_1) - (\mathrm{sgn}\; p^{Q_1}(y)p^{Q_2}(y))\Theta^{n-1}(\omega^-, Q_2)|.\end{aligned}$$ If $\mathrm{sgn}\; p^{Q_1}(y)p^{Q_2}(y) = -1$ ($p^{Q_1}(y)$ and $p^{Q_2}(y)$ have opposite signs), then $$\|\tilde{p}^{Q_1}- \tilde{p}^{Q_2}\|_{L^\infty(\partial B_1)} \geq \tilde{c}(\Theta^{n-1}(\omega^-, Q_1) + \Theta^{n-1}(\omega^-, Q_2)) \geq \tilde{c}\Theta^{n-1}(\omega^-, Q_1).$$ Letting $Q_2 \rightarrow Q_1$, the continuity of $Q\mapsto \tilde{p}^{Q}$ (Proposition \ref{continuityofblowup}) implies that $0 \geq \tilde{c} \Theta^{n-1}(\omega^-, Q_1)$. This contradicts the non-degeneracy of $\Theta^{n-1}(\omega^-, Q_1)$ (Proposition \ref{nondegeneracy}).

 On the other hand, if $\mathrm{sgn}\; p^{Q_1}(y)p^{Q_2}(y) = 1$ ($p^{Q_1}(y)$ and $p^{Q_2}(y)$ share the same sign), then $$\|\tilde{p}^{Q_1}- \tilde{p}^{Q_2}\|_{L^\infty(\partial B_1)} \geq \tilde{c}|\Theta^{n-1}(\omega^-, Q_1) - \Theta^{n-1}(\omega^-, Q_2)|,$$ and the continuity of $Q \mapsto \tilde{p}^{Q}$  implies that $Q\mapsto \Theta^{n-1}(\omega^-, Q)$ is continuous.  
\end{proof}

Uniform non-degeneracy immediately follows.

\begin{cor}\label{uniformlowerbound}
For any $K \subset\subset \R^n$ there is a $c = c(K) > 0$ such that, for all $Q \in K\cap \partial \Omega$, $$\Theta^{n-1}(\omega^{\pm}, Q) > c.$$
\end{cor}

\subsection{$\partial \Omega$ is a $C^1$ domain} We define for $Q_0\in \partial \Omega$ and $r > 0$ \begin{equation}\label{definitionofbeta}\beta(Q_0,r) = \inf_{P} \frac{1}{r}\sup_{Q\in \partial \Omega \cap B_r(Q_0)} \mathrm{dist}(Q, P)\end{equation} where the infimum is taken over all $(n-1)$-dimensional hyperplanes through $Q_0$ (these are a variant of Jones' $\beta$-numbers, see \cite{jones}). David, Kenig and Toro (see \cite{davidkenigtoro}, Proposition 9.1) show that, under suitable assumptions, $\beta(Q_0,r) \lesssim r^\gamma$  implies that $\partial \Omega$ is locally the graph of a $C^{1,\gamma}$ function for any $1 > \gamma > 0$. We will adapt this proof to show that $\partial \Omega$ is locally the graph of a $C^1$ function. 

For any $Q_0 \in \partial \Omega$, $$P(Q_0) \colonequals \{p^{Q_0} = 0\}$$ (where $p^{Q_0}$ is the 1-homogenous polynomial guaranteed to exist by Corollary \ref{blowupsforeveryone} and which is unique by Lemma \ref{uniqueblowup}). By the definition of blowups, we know that $P(Q_0)+Q_0$ approximates $\partial \Omega$ near $Q_0$. The following lemma shows that this approximation is uniformly tight in $Q_0$. 

\begin{lem}\label{uniformplaneconvergence}[Compare to \cite{davidkenigtoro}, equation 9.14]
Let $K \subset \subset \R^n$ and $\varepsilon > 0$. Then there is an $R = R(K, \varepsilon) > 0$ such that $r < R$ and $Q_0 \in K \cap \partial \Omega$ implies \begin{equation}\label{planeconvergence}\sup_{Q\in \partial \Omega \cap B_r(Q_0)} \frac{1}{r}\mathrm{dist}\left(Q-Q_0, P(Q_0)\right) < \varepsilon.\end{equation}
\end{lem}

\begin{proof}
The proof hinges on the following estimate (see \cite{gandp} Theorem 1.5.5); for any $K$ compact there exists a modulus of continuity $\sigma_K$ with $\lim_{t\downarrow 0}\sigma_K(t) = 0$ such that  \begin{equation}\label{taylorexpansion}
|v^{(Q_0)}(x+Q_0) - \tilde{p}^{Q_0}(x)| \leq \sigma_K(|x|)|x|
\end{equation} for any $Q_0 \in K \cap \partial \Omega$. 

Assume this estimate is true; let $Q\in B_r(Q_0)\cap \partial \Omega$ and write $Q= Q_0 + x$. As $\Theta^{n-1}(\omega^-, Q_0) > c$ for all $Q_0 \in K\cap \partial \Omega$ (Corollary \ref{uniformlowerbound}) it follows that $\mathrm{dist}(Q-Q_0, P(Q_0)) \lesssim |\tilde{p}^{Q_0}(x)|$. Then \eqref{taylorexpansion} yields that $\mathrm{dist}(Q-Q_0, P(Q_0)) \lesssim |\tilde{p}^{Q_0}(x)| \leq |x| \sigma_K(|x|) =r \sigma_K(r)$. Set $R$ to be small enough so that $r < R$ implies $\sigma_K(r) < \varepsilon$ to prove \eqref{planeconvergence}.

Thus it suffices to establish \eqref{taylorexpansion}. Let $|x| = r$ and write $x = ry$ with $|y| = 1$. If we divide by $r$, \eqref{taylorexpansion} is equivalent to \begin{equation}\label{taylorexpansiontransformed} |v^{(Q_0)}(ry+Q_0)/r - \tilde{p}^{Q_0}(y)| \leq \sigma_K(r).\end{equation}

As $v^{(Q)}(ry+Q)/r$ is locally Lipschitz (uniformly in $Q$ on compacta), the uniform estimate \eqref{taylorexpansiontransformed} follows from an $L^2$ estimate: for all $\varepsilon > 0$, there exists a $R = R_{K,\varepsilon} > 0$ such that if $r < R$ and $Q_0 \in K\cap \partial \Omega$ then $$M^{Q_0}(r, v^{(Q_0)}, \tilde{p}^{Q_0}) \equiv \|v^{(Q_0)}(ry + Q_0)/r - \tilde{p}^{Q_0}(y)\|^2_{L^2(\partial B_1)}  <  \varepsilon.$$
 
For each point $Q\in K\cap \partial \Omega$ we can find an $R = R_\varepsilon(Q)$ such that $R << \varepsilon$ and for all $r < R,\; |M^{Q}(r, v^{(Q)}, \tilde{p}^{Q}) |< \varepsilon/4$. Furthermore, for every $r > 0$ there is a $\delta(r) > 0$ such that for $Q,Q' \in K\cap \partial \Omega$ we have $$
|Q-Q'| < \delta(r) \Rightarrow |M^{Q'}(r, v^{(Q')}, \tilde{p}^{Q'}) - M^{Q}(r, v^{(Q)}, \tilde{p}^{Q})| < \varepsilon/4.$$ The existence of $\delta(r)$ follows from the uniform Lipschitz continuity of $v^{(Q)}$, the H\"older continuity of $h$ and the continuity of $Q\mapsto \tilde{p}^{Q}$.

 As $K$ is compact we can find $Q_1,...,Q_n \in K\cap \partial \Omega$ such that if $\delta_1\colonequals \delta(R_\varepsilon(Q_1)),...,\delta_n\colonequals \delta(R_\varepsilon(Q_n))$ then $K\cap \partial \Omega \subset \bigcup B_{\delta_i}(Q_i)$. By the definition of $\delta_i$, if $Q'\in B_{\delta_i}(Q_i)$, then $M^{Q'}(R_\varepsilon(x_i), v^{(Q')}, \tilde{p}^{Q'}) < \varepsilon/2$. Recall, $R_\varepsilon(Q_i) << \varepsilon$ and Lemma \ref{growthofM} to conclude that for all $Q' \in B_{\delta_i}(Q_i)$ and $r < R_\varepsilon(Q_i), M^{Q'}(r, v^{(Q')}, \tilde{p}^{Q'}) < \varepsilon$.  Therefore, setting $R_{K,\varepsilon} \equiv \min_i \{R_{\varepsilon}(Q_i)\}$ gives the $L^2$ estimate $r < R_{K,\varepsilon}, Q' \in K \cap \partial \Omega \Rightarrow M^{Q'}(r, v^{(Q')}, \tilde{p}^{Q'}) < \varepsilon$. 
  \end{proof}

We should note, \eqref{taylorexpansion} (along with the Whitney extension theorem) allows for an alternative proof that $\partial \Omega$ is a $C^1$ domain (see \cite{gandp} Theorem 1.3.8).  We will, however, continue our proof in the vein of \cite{davidkenigtoro}.

\begin{prop}\label{boundaryisc1}
Let $\Omega \subset \R^n$ satisfy the conditions of Theorem \ref{maintheorem} or Theorem \ref{maintheoremprime}. If $\log(h) \in C^{0,\alpha}(\partial \Omega)$ then $\Omega$ is a $C^1$ domain.
\end{prop}

\begin{proof}
For $Q_0 \in \partial \Omega$, equation \ref{planeconvergence} shows that $P(Q_0) + Q_0$ is a tangent plane to $\partial \Omega$ at $Q_0$. Furthermore, $Q_0 \mapsto P(Q_0)$ is continuous (Corollary \ref{continuityofdensity}). Under the assumptions of Theorem \ref{maintheoremprime}, $\Omega$ is a Lipschitz domain with a tangent plane at every $Q\in \partial \Omega$ that varies continuously in $Q$; thus we are done. 

If we simply assume that $\Omega$ is Reifenberg flat (Theorem \ref{maintheorem}), we still need to show that $\Omega$ is a graph domain (in fact we will show it is a Lipschitz domain).  Let $R = R_{K,\varepsilon} > 0$ be chosen later and let $r < R$. If $R$ is small enough, vanishing Reifenberg flatness (Corollary \ref{blowupsforeveryone}), along with Lemma \ref{uniformplaneconvergence}, implies $$\pi(\{\partial \Omega \cap B(Q_0, r)-Q_0\}) \supset P(Q_0) \cap B(0, \frac{r}{2}),\; \forall Q_0 \in K\cap \partial \Omega, r < R.$$  Here $\pi: \R^n \rightarrow P(Q_0)$ is a projection (for more details see the proof of \cite{davidkenigtoro} Lemma 8.3 or \cite{kenigtoroduke} Remark 2.2). 

We need only to show that $\pi^{-1}$ is a well defined function with bounded Lipschitz norm on $P(Q_0)\cap B(0,r/2)$. Let $\Sigma \colonequals (\partial \Omega-Q_0) \cap B(0, r) \cap \pi^{-1}(B(0, r/2))$ and pick distinct $Q_1,Q_2\in \Sigma$. Perhaps shrinking $R$ again, the continuity of $Q\mapsto P(Q)$, combined with Lemma \ref{uniformplaneconvergence}, implies \begin{equation}\label{planeestimate}\frac{1}{|Q_1-Q_2|} \mathrm{dist}(Q_1-Q_2, P(Q_0)) < \varepsilon.\end{equation} Therefore, $\pi^{-1}$ is well defined and $\|\pi^{-1}\|_{\mathrm{Lip}(P(Q_0)\cap B(0,r/2))} < (1-\varepsilon)^{-1}$.
\end{proof}

It should be noted that if $\Omega$ is a $C^1$ domain it is not necessarily true that $u \in C^1(\overline{\Omega})$ (see \cite{pommeranke}, pg 45). However, as $\Theta^{n-1}(\omega^\pm, Q)$ is continuous, we can establish the following.

\begin{cor}\label{c1extension}
Let $\Omega, \log(h)$ be as in Proposition \ref{boundaryisc1}. Then $u^\pm \in C^1(\overline{\Omega^{\pm}})$. 
\end{cor}

\begin{proof}
For $Q\in \partial \Omega$, let $\nu(Q)$ be the inward pointing normal to $\Omega$ at $Q$.  We will prove that $$\lim_{\stackrel{X\rightarrow Q}{X\in \Omega^+}} D_iu^+(X) = (\nu(Q)\cdot e_i)\Theta^{n-1}(\omega^+,Q), \forall i = 1,...,n.$$ The desired result follows from $\Theta^{n-1}(\omega^+,-), \nu(-) \in C(\partial \Omega)$ (Corollary \ref{continuityofdensity} and Proposition \ref{boundaryisc1}). The proof for $u^-$ is identical.

Pick $r$ small so that $B(Q,r) \cap \partial \Omega$ can be written as the graph of a $C^1$ function. Then construct a bounded NTA domain $\Omega_B \subset \Omega$ such that $\partial \Omega_B\cap \partial \Omega = B(Q,r)\cap \partial \Omega$ (see \cite{jerisonandkenig} Lemma 6.3 and \cite{kenigtoro} Lemma A.3.3). For $X_0\in \Omega_B$, let $\omega_B^{X_0}$ be the harmonic measure of $\Omega_B$ with a pole at $X_0$. By local Lipschitz continuity, $|D_iu^+| < C$ on $\Omega_B$ and, therefore, $D_iu^+$ has a non-tagential limit $g(P)$ for $\omega_B^{X_0}$-a.e. $P$ in $\partial \Omega_B$ (see Section 5 in \cite{jerisonandkenig}).  Furthermore, if $K(X,P) \colonequals \frac{d\omega_B^X}{d\omega_B^{X_0}}(P)$ we have the following representation (see \cite{jerisonandkenig} Corollary 5.12), $$D_iu^+(X) = \int_{\partial \Omega_B}g(P)K(X,P)d\omega_B^{X_0}(P).$$ 

Using blowup analysis, one computes $g(P) = (\nu(P)\cdot e_i)\Theta^{n-1}(\omega^+, P)$ for $P \in \partial \Omega_B\cap B(Q,r/2)$. As $g(P)$ is continuous on $B(Q,r/2)\cap \partial \Omega_B$, there is some $s < r/2$ such that $P\in B(Q, s)\cap \partial \Omega_B \Rightarrow |g(P) -g(Q)| < \varepsilon$. On the other hand, Jerison and Kenig (Lemma 4.15) proved that $\lim_{X\in \Omega_B, X\rightarrow Q} \sup_{P\in \partial \Omega_B\backslash B(Q,s)} K(X,P) = 0.$ This allows us to estimate, $$\lim_{\stackrel{X\rightarrow Q}{X\in \Omega^+}} |D_iu^+(X) - (\nu(Q)\cdot e_i)\Theta^{n-1}(\omega^+,Q)| = \lim_{\stackrel{X\rightarrow Q}{X\in \Omega_B}} |D_iu^+(X) - g(Q)| $$$$\leq \lim_{\stackrel{X\rightarrow Q}{X\in \Omega_B}}\int_{\partial \Omega_B\backslash B(Q,s)} K(X,P)|g(P)-g(Q)| d\omega_B^{X_0}(P) + \int_{\partial \Omega_B\cap B(Q,s)} K(X,P) |g(P)-g(Q)|d\omega_B^{X_0}(P)$$$$\leq  \lim_{\stackrel{X\rightarrow Q}{X\in \Omega_B}} C\omega_B^{X_0}(\partial \Omega_B\backslash B(Q,s))\sup_{P\notin B(Q,s)} K(X,P) + \varepsilon \omega_B^{X}(B(Q,s)) \leq \varepsilon.$$ The first equality follows from the fact that any sequence in $\Omega^+$ approaching $Q$ must, apart from finitely many terms, be contained in $\Omega_B$. The last line follows  first from $|g(P)|< C$ and then from the fact that $\omega_B^{X}$ is a probability measure for any $X\in \Omega_B$. 
\end{proof} 

\section{Initial H\"older regularity: $\partial \Omega$ is $C^{1,s}$}\label{sec: initialholderregularity} In this section we will prove that $\partial \Omega$ is locally the graph of a $C^{1,s}$ function for some $0 < s \leq \alpha$. Note that, in general, the best one can hope for is $s = \alpha$ (if $\partial \Omega$ is the graph of a $C^{1,\alpha}$ function then $\log(h) \in C^{0,\alpha}$). 

Here we will borrow heavily from the arguments of De Silva et al. \cite{silvaferrarisalsa}, who prove $C^{1,\gamma}$ regularity for a wide class of non-homogenous free boundary problems. We cannot immediately apply their results, as they assume a non-degeneracy in the free boundary condition that our problem does not have (see condition (H2) in Section 7 of \cite{silvaferrarisalsa}). It should also be noted that our main result in this section is not immediately implied by the remark at the end of Caffarelli's paper, \cite{caf}. Indeed, Caffarelli's free boundary condition also contains an {\it a priori} non-degeneracy condition (see condition (a) at the top of page 158 in \cite{caf}) which our problem lacks. 

\subsection{The Iterative Argument} In this section we shall state the main lemma and show how that lemma, through an iterative argument, implies our desired result. First we need two definitions.

\begin{defin}\label{freeboundaryproblem}
Let $g: \R^n \rightarrow \R$. Then $w \in C(B_1(0))$ is {\bf a solution to the free boundary problem associated to $g$} if:
\begin{itemize}
\item $w \in C^2(\{w > 0\}) \cap C^2(\{w < 0\})$
\item $w \in C^1(\overline{ \{w > 0\} }) \cap C^1(\overline{ \{w < 0\} })$
\item $w$ satisfies, in $B_1(0)$, the following: 
\begin{equation}\label{classicalsolution}
\begin{aligned}
\Delta w(x) &= 0, \; x\in \{w \neq 0\}\\
(w^+)_{\nu_x}(x)g(x) &= -(w^-)_{\nu_x}(x),\; x\in \{w = 0\}
\end{aligned}
\end{equation}
 where $\nu_x$ is the normal to $\{w=0\}$ at $x$.
\end{itemize}
\end{defin}

One observes that Corollary \ref{c1extension} implies that $u$ is a solution to the free boundary problem associated to $h$. We now need the notion of a ``two-plane solution".

\begin{defin}\label{twophasesolution}
Let $\gamma >0$ and $g:\R^n\rightarrow \R$. Then for any $x_0 \in B_1(0)$ we can define the {\bf two-plane solution associated to $g$ at $x_0$}: $$U_\gamma^{(x_0)}(t) \colonequals \gamma t^+ - g(x_0)\gamma t^-,\; t\in \R.$$  When no confusion is possible we drop the dependence on $x_0$. It should also be clear from context to which function $g$ our $U$ is associated. 
\end{defin}

The following remark, which follows immediately from Corollary  \ref{blowupsforeveryone} and \eqref{taylorexpansion}, elucidates the relationship between a two-plane solution and our function $u$.

\begin{rem}\label{convergencetotwoplane}
Let $x_0 \in \partial \Omega$. As $r\rightarrow 0$ it is true that  $$u_{r,x_0}(x) \colonequals \frac{u(rx + x_0)}{r}\rightarrow U^{(x_0)}_{\Theta^{n-1}(\omega^+, x_0)}(x\cdot \nu_{x_0})$$ uniformly on compacta. Here $U$ is the two-plane solution associated to $h$. Furthermore, the rate of this convergence is independent of $x_0\in K\cap \partial \Omega$ for $K$ compact. \end{rem}

Intuitively, the faster the rate of this convergence, the greater the regularity of $\partial \Omega$. This relationship motivates the following lemma (compare with \cite{silvaferrarisalsa}, Lemma 8.3),  which says roughly that if $u$ is close to a two-plane solution in a large ball, then $u$ is in fact even closer to a, possibly different, two-plane solution in a smaller ball. 

\begin{lem}\label{iterationstep}
Let $\infty > C_1, c_1 > 0$ and $\tilde{k} > 0$.  Let $v$ be a solution to a free boundary problem associated to $g$ such that $\inf_{x\in B_2(x_0)} g(x) \geq \tilde{k} > 0$ and such that $v(x_0) = 0$. Let $\varepsilon > 0, C_1 > \gamma > c_1, \nu \in \mathbb S^{n-1}$ and assume \begin{equation}\label{squeeze1}
U_\gamma^{(x_0)}(x\cdot \nu - \varepsilon) \leq v(x+x_0) \leq U_\gamma^{(x_0)}(x\cdot \nu + \varepsilon), \; x\in B_1(0).
\end{equation}
Also, assume that $\sup_{x,y \in B_1(x_0)}\frac{|g(x) - g(y)|}{|x-y|^\alpha} < \varepsilon^2$. 

Then there exists some $R_0 = R_0(C_1, c_1, n) > 0$ such that for all $r < R_0$ there is a $\tilde{\varepsilon} = \tilde{\varepsilon}(r, C_1, c_1, n) > 0$ so that if the $\varepsilon$ above satisfies $\varepsilon \leq\tilde{\varepsilon}$ then \begin{equation}\label{squeeze2}
U_{\gamma'}^{(x_0)}(x\cdot \nu' - r\frac{\varepsilon}{2}) \leq v(x+x_0) \leq U^{(x_0)}_{\gamma'}(x\cdot \nu' + r \frac{\varepsilon}{2}),\; x\in B_r(0),
\end{equation} where $|\nu'| = 1, |\nu' - \nu| \leq \tilde{C}\varepsilon$ and $|\gamma -\gamma'| \leq \tilde{C}\gamma \varepsilon$. Here $\tilde{C} = \tilde{C}(C_1, c_1, n) > 0$.
\end{lem}

With this lemma we can prove H\"older regularity by way of an iterative argument. 

\begin{prop}\label{initialholderregularity}
Let $\Omega \subset \R^n$ be a 2-sided NTA domain with $\log(h) \in C^{0,\alpha}(\partial \Omega)$. 
\begin{itemize}
\item If $n =2$, then $\Omega$ is a $C^{1,s}$ domain for some $s > 0$. 
\item If $n \geq 3$, assume that either $\Omega$ is a $\delta$-Reifenberg flat domain for some $0 < \delta$ small enough or that $\Omega$ is a Lipschitz domain. Then $\Omega$ is a $C^{1,s}$ domain for some $s > 0$. \end{itemize} \end{prop}

\begin{proof}[Proof of Proposition \ref{initialholderregularity} assuming Lemma \ref{iterationstep}] Without loss of generality let $0 \in \partial \Omega$ and $e_n$ be the inward pointing normal to $\Omega$ at $x_0\in B_1(0)\cap \partial \Omega$. We will show that $\beta(x_0, t) \leq C''t^s$ for some $s > 0$ and some $C'' > 0$ independent of $t > 0, x_0 \in \partial \Omega \cap B_1(0)$. A theorem of David, Kenig and Toro (\cite{davidkenigtoro}, Proposition 9.1) then implies that $\partial \Omega$ is locally the graph of a $C^{1,s}$ function.  

Set $\gamma = \Theta^{n-1}(\omega^+, x_0)$ and let $$C_1 \colonequals 2\sup_{z \in \partial \Omega \cap B_4(0)} \Theta^{n-1}(\omega^+, z),\: c_1 \colonequals \frac{1}{2} \inf_{z \in \partial \Omega \cap B_4(0)} \Theta^{n-1}(\omega^+, z).$$ By Corollary \ref{uniformlowerbound} and the work of Section \ref{sec: u is lipschitz} we have $\infty > C_1 \geq c_1 > 0.$

 Lemma \ref{iterationstep} gives us an $R_0$. Pick $0 < \overline{r} \leq R_0$ small enough so that $\overline{r}^\alpha < \frac{1}{4}$. We then get a $\tilde{\varepsilon} > 0$ depending on $\overline{r}$. Pick $\varepsilon < \tilde{\varepsilon}$ such that $$1/2 \leq \left(\prod_{k=0}^\infty (1-\tilde{C}\varepsilon/2^k)\right) < \left(\prod_{k=0}^\infty (1+\tilde{C}\varepsilon/2^k)\right) \leq 2$$ where $\tilde{C}$ is the constant from Lemma \ref{iterationstep}.

Recall Remark \ref{convergencetotwoplane}, that $u_{\rho, x_0}(x) \rightarrow U_{\gamma}(x_n)$ for $x\in B_1$ as $\rho \downarrow 0$. Thus, for small enough $\rho$, we have $$\|u_{\rho, x_0}(x) - U_{\gamma}(x_n)\|_{L^\infty(B_1)} < K\varepsilon,$$ where $K \leq \min\{c_1, \inf_{x\in B_1}|h(x)|c_1\}$. This implies $$U_{\gamma}(x_n - \varepsilon) \leq u_{\rho, x_0}(x) \leq U_{\gamma}(x_n + \varepsilon), x\in B_1(0).$$ $u_{\rho, x_0}$ is a solution to the free boundary problem associated to $g(x) = h(\rho x+ x_0)$. In particular, if $\rho$ is small enough such that $\rho^{\alpha}\|h\|_{C^{0,\alpha}} < \varepsilon^2$ then $g$ satisfies the growth and lower bound assumptions of Lemma \ref{iterationstep}.
 
 If $u^{0}(x) \colonequals u_{\rho, x_0}(x)$, then we can apply Lemma \ref{iterationstep} to $u^{0}$ in direction $e_n$ with $\gamma, C_1, c_1, \overline{r}, \varepsilon$ as above. This gives us a $\nu_1 \in \mathbb S^{n-1}$ and a $\gamma_1 > 0$ such that $$U_{\gamma_1}(x\cdot \nu_1 - \overline{r}\frac{\varepsilon}{2}) \leq u^{0}(x) \leq U_{\gamma_1}(x\cdot \nu_1 + \overline{r} \frac{\varepsilon}{2}),\; x\in B_{\overline{r}}(0).$$ Write $x = \overline{r}y$ and divide the above equation by $\overline{r}$ to obtain, $$U_{\gamma_1}(y\cdot \nu_1 - \frac{\varepsilon}{2}) \leq u^{0}(\overline{r}y)/\overline{r} \leq U_{\gamma_1}(y \cdot \nu_1 +\frac{\varepsilon}{2}),\; y\in B_1(0).$$ Let $u^1(z) \colonequals u^{0}(\overline{r}z)/\overline{r}$ so that $$U_{\gamma_1}(y\cdot \nu_1-\varepsilon/2) \leq u^1(y) \leq U_{\gamma_1}(y\cdot \nu_1+\varepsilon/2),\; y\in B_1(0)$$
Apply Lemma \ref{iterationstep} to $u^1$ in direction $\nu_1$ with $C_1, c_1, \gamma_1,\varepsilon/2, \overline{r}$ and iterate. 
 
 In this way, we create a sequence of $u^k(y), \theta_k, \gamma_k, \nu_k$ such that $$U_{\gamma_k}(y\cdot \nu_k - \varepsilon/2^k) \leq u^k(y) \leq U_{\gamma_k}(y\cdot \nu_k + \varepsilon/2^k),\; y\in B_1(0)$$ and $|\nu_k - \nu_{k+1}| <\tilde{C}\varepsilon/2^k$. We must prove that it is valid to apply Lemma \ref{iterationstep} at each step. 
 
 By Lemma \ref{iterationstep} and construction, $$c_1 \leq \frac{1}{2}\gamma \leq \prod_{i=0}^{k-1}(1-\tilde{C}\varepsilon/2^k)\gamma \leq \gamma_k  \leq \prod_{i=0}^{k-1}(1+\tilde{C}\varepsilon/2^k)\gamma \leq 2\gamma \leq C_1,$$ so $\gamma_k$ is always in the acceptable range for another application of Lemma \ref{iterationstep}. Also in the $k$th step we apply the lemma with $\varepsilon/2^k < \varepsilon < \tilde{\varepsilon}$ and the same $\overline{r}$.
 
 Finally, in the $k$th step we have $u^k(y) = u_{\rho \overline{r}^k, x_0}(y)$. Thus we need to make sure that $(\rho \overline{r}^k)^\alpha \|h\|_{C^{0,\alpha}} < (\varepsilon/2^k)^2.$ By construction, $\rho^\alpha \|h\|_{C^{0,\alpha}}< \varepsilon^2$ and $\overline{r}^{k\alpha} \leq \frac{1}{4}^k$ and so the conditions of Lemma \ref{iterationstep} are satisfied for each $k$.  
 
After $k$ steps, $$U_{\gamma_k}(y\cdot \nu_k - \varepsilon/2^k) \leq u^k(y) \leq U_{\gamma_k}(y\cdot \nu_k + \varepsilon/2^k),\; y\in B_1(0)\Rightarrow$$$$U_{\gamma_k}(x\cdot \nu_k -\rho\overline{r}^k \varepsilon/2^k) \leq u(x+x_0) \leq U_{\gamma_k}(x\cdot \nu_k+ \rho\overline{r}^k\varepsilon/2^k),\; x\in B_{\rho\overline{r}^k}(0).$$ If $x\in B_{\rho\overline{r}^k}(0)$ is taken such that $x+x_0 \in \partial \Omega$ then the above equation implies $$x\cdot \nu_k - \rho\overline{r}^k\varepsilon/2^k < 0 < x\cdot\nu_k + \rho\overline{r}^k\varepsilon/2^k\Rightarrow$$$$|x\cdot \nu_k| \leq \rho\overline{r}^k\varepsilon/2^k \Rightarrow \beta(x_0, \rho \overline{r}^k) \leq \varepsilon/2^k.$$

If $s \colonequals -\log_{\overline{r}}(2) > 0$, we have shown $\beta(x_0, \rho\overline{r}^k) \leq \frac{\varepsilon}{\rho^s}(\rho\overline{r}^k)^s \leq C'(\rho\overline{r}^k)^s$ (Remark \ref{convergencetotwoplane} implies that we can we can take $\rho$ uniformly in $x_0\in B_1(0)$). If $t$ is such that $\rho\overline{r}^{k+1} < t \leq \rho\overline{r}^k$ we can estimate $$\beta(x_0, t) < \frac{\rho\overline{r}^k}{t} \beta(x_0, \rho\overline{r}^k) <C'\frac{\overline{\rho r}^k}{t} (\rho\overline{r}^k)^s = C'\frac{\overline{\rho r}^k}{t} t^s \left(\frac{\overline{\rho r}^k}{t}\right)^s \leq \frac{C'}{\overline{r}^{1+s}} t^s \equiv C'' t^s,$$ where we used that $\frac{\rho\overline{r}^k}{t} < \frac{1}{\overline{r}}$.
\end{proof}

It is worthwhile to note that the condition $\overline{r}^\alpha < 1/4$ implies $s = -\log_{\overline{r}}(2) < \alpha/2$. So this argument does not give optimal H\"older regularity.

\subsection{Harnack Inequalities} It remains to prove Lemma \ref{iterationstep}. We first define a subsolution to the free boundary problem (see Definition \ref{freeboundaryproblem}).

\begin{defin}\label{subsoltn}
Let $\Op$ be an open set in $\R^n$ and $g:\R^n\rightarrow \R$. We say that $z \in C(\overline{\Op})$ is a {\bf strict-subsolution to the free boundary problem associated with $g$} in $\Op$ if:
\begin{itemize}
\item $\{z =0\}$ is locally the graph of a $C^2$ function.
\item $z \in C^1(\overline{ \{z > 0\}\cap \Op}) \cap C^1(\overline{ \{z < 0\}\cap \Op})$. 
\item On the set $\{z \neq 0\}$ we have $\Delta z > 0$.
\item For $x_0 \in \{z = 0\}$ we have $$g(x_0)(z^+)_{\nu_{x_0}}(x_0) + (z^-)_{\nu_{x_0}}(x_0) > 0,$$ where $\nu_{x_0}$ is the inward pointing normal at $x_0$ to $\{z > 0\}$. 
\end{itemize}

We define a strict supersolution analogously. 
\end{defin}

 With this definition we need a comparison principle (note that this comparison principle can also be taken to be the definition of a sub/super solution, see e.g. \cite{silvaferrarisalsa}).

\begin{lem}\label{comparisonlem}[Compare to \cite{cafandsalsa} Lemma 2.1, \cite{silvaferrarisalsa} Definition 7.2]
Let $\Op$ be an open set in $\R^n$. Let $w, z$ be a solution and strict subsolution respectively to the free boundary problem associated to a positive $g$ in $\Op$. If  $w\geq z$ in $\overline{\Op}$ then $w > z$ in $\Op$. 

The analogous statement holds for supersolutions.
\end{lem}

\begin{proof}
We proceed by contradiction and let $\tilde{x}$ be a touching point. There are three cases:

\medskip

{\bf Case 1:} $\tilde{x} \in \{z=0\}$.  $\{z=0\}$ is locally the graph of a $C^2$ function so there is a tangent ball $B \subset \{z  >0\}$ with $B\cap \{z =0\} = \tilde{x}$. Since $\{z > 0\} \subseteq \{w > 0\}$ we have $B \cap \{w = 0\} = \tilde{x}$ and $B \subset \{w > 0\}$. As such $\{z=0\}, \{w = 0\}$ share a normal vector $\nu$ at $\tilde{x}$.

 Since $w \geq z, z\neq w$ we have that $z-w$ attains a local maximum at $\tilde{x}$. Thus $(z^+-w^+)_\nu \leq 0$ and $(-z^-+w^-)_{-\nu} = (z^- - w^-)_\nu \leq 0$. We then have $0 \geq g(\tilde{x})(z^+-w^+)_\nu + (z^--w^-)_\nu = g(\tilde{x})(z^+)_\nu + (z^-)_\nu > 0$ a contradiction.
 
 \medskip

{\bf Case 2:} $\tilde{x} \in \{z > 0\}$. As $\{z > 0\} \subseteq \{w >0\}$, both $-w,z$ are subharmonic on $\{z > 0\}$. So $z-w$ cannot attain a local maximum on $\{z > 0\}$ which implies $w > z$ on $\{z > 0\} \cap \Op$.  

\medskip

{\bf Case 3:} $\tilde{x} \in \{z < 0\}$. In this case $\tilde{x} \in \{w < 0\}$.  As $\{w < 0\} \subseteq \{z < 0\}$, we have $-w, z$ are both subharmonic on $\{w < 0\}$. We can then argue as in {\bf Case 2}.  
\end{proof}

With this comparison lemma we can prove a ``one-sided" Harnack type inequality. 

\begin{lem}\label{harnacklemma}[Compare with \cite{silvaferrarisalsa}, Lemmas 4.3 and 8.1]
Let $w$ be a solution to the free boundary problem associated to a positive continuous function $g$ on $B_1(0)$ (see Definition \ref{freeboundaryproblem}). Let $\tilde{k} > 0$ and assume $\inf_{x\in B_1(0)} g(x) \geq \tilde{k}$. Also assume  $w$ satisfies $$w(x) \geq U^{(0)}_\gamma(x\cdot \nu),\; x\in B_1(0)$$ (where $\nu \in \mathbb S^{n-1}$ and $\gamma > 0$) and that at $\overline{x} = \frac{1}{5}\nu$ \begin{equation}\label{inequalitygap}w(\overline{x}) \geq U^{(0)}_\gamma(1/5 + \varepsilon).\end{equation} Finally, assume that $\sup_{x\in B_1} |g(0) - g(x)| \leq 10 \varepsilon^2$. 

 Then there exists $\overline{\varepsilon}> 0$ and $0 < c < 1$ (which depend only on the dimension and $k$), such that if the above $\varepsilon < \overline{\varepsilon}$ we can conclude $$w(x) \geq U^{(0)}_\gamma(x\cdot \nu + c\varepsilon),\; x\in \overline{B}_{1/2}(0).$$ 

Analogously, if $w(x) \leq U_\gamma(x\cdot \nu), \; x\in B_1$ and $w(\overline{x}) \leq U_\gamma(1/5-\varepsilon)$ then $w(x) \leq U_\gamma(x\cdot \nu - c\varepsilon)$ in $\overline{B}_{1/2}(0)$. 
\end{lem}

\begin{proof}
For ease of notation we will drop the dependence of $U$ on $\gamma, 0$ and let $\nu = e_n$. We prove the inequality from below; the inequality from above, and the result for general $\nu$, is proven similarly. Our first step is to widen the gap between $w$ and $U$:

\medskip

{\bf Claim:} There exists a universal $c_1 > 0$ such that $w(x) \geq (1+c_1\varepsilon)\gamma x_n^+ - g(0)\gamma x_n^-$ for all $x\in \overline{B}_{19/20}(0)$ and for universal $c_1 > 0$.

\medskip

{\it Proof of Claim:} In $\overline{B}_{1/20}(\overline{x})$ there is a universal constant $c_0 > 0$ such that $w(x)-U_\gamma(x) \geq c_0\gamma \varepsilon\geq c_0\gamma \varepsilon x_n$ by the Harnack inequality and \eqref{inequalitygap}. 

Define $\Op = (B_1 \cap \{x_n > 0\})\backslash \overline{B}_{1/20}(\overline{x})$ and let $\phi$ be the harmonic function in $\Op$ such that $\phi = 0$ on $\partial (B_1 \cap \{x_n > 0\})$ and $\phi = 1$ on $\partial B_{1/20}(\overline{x})$. 

We have $$w(x) - \gamma x_n \geq 0  = \gamma c_0\phi(x)\varepsilon/2, x\in \partial (B_1\cap \{x_n > 0\}).$$ Also, note $$w(x) -\gamma x_n \geq  c_0\gamma \varepsilon  \geq \gamma c_0 \varepsilon \phi(x)/2, x\in \partial B_{1/20}(\overline{x}).$$ As $w-\gamma x_n$ and $\gamma c_0 \varepsilon \phi(x)/2$ are both harmonic on $\Op$ we have that $w-\gamma x_n \geq \gamma c_0 \varepsilon \phi(x)/2$ on all of $\Op$.  Finally, by the boundary Harnack principle there is a $\tilde{c} > 0$ such that $\phi \geq \tilde{c}x_n$ on $\overline{\Op}\cap B_{19/20}$.  Therefore, $c_1 = \min\{c_0, c_0\tilde{c}/2, 5/2\}$ is such that $w-\gamma x_n^+ \geq \gamma \varepsilon c_1 x_n^+$ on $\overline{B}_{19/20}$, proving the claim.

\medskip

Recall $w(\overline{x})-U(\overline{x}_n) \geq \gamma \varepsilon > 0$. Thus $w(\overline{x}) - (1+c_1\varepsilon)\gamma(\overline{x}_n)^+ \geq \gamma\varepsilon - c_1 \gamma \varepsilon/5 \geq \gamma \varepsilon/2$.  The Harnack inequality tells us that $$w(x)- (1+c_1\varepsilon)\gamma (x_n)^+ \geq c' \varepsilon \gamma,\; x\in \overline{B}_{1/20}(\overline{x}),$$ for $c'$ universal depending on dimension. If $c_2$ is small enough that $(1+c_1 \varepsilon)c_2\leq c'$, then
\begin{equation}\label{insideinequality}w(x)- (1+c_1\varepsilon)\gamma (x_n + c_2\varepsilon)^+ \geq 0,\; x\in \overline{B}_{1/20}(\overline{x}).\end{equation} 

 Now we create a strict subsolution in the annulus $$A\colonequals B_{3/4}(\overline{x})\backslash \overline{B}_{1/20}(\overline{x})$$ and then use this subsolution to transfer the gap in \eqref{insideinequality} to a neighborhood of $0$. 

Let $$\psi(x) \colonequals 1 - c(|x-\overline{x}|^{-n} - (3/4)^{-n}),\; x\in A,$$ where $c$ is such that $\psi = 0$ on $\partial B_{1/20}(\overline{x})$. Then $0 \leq \psi \leq 1$ and $-\Delta \psi \geq k(n) > 0$ in $A$. We can extend $\psi \equiv 0$ on $B_{1/20}(\overline{x})$. 

For $t\geq 0$ we write \begin{equation}\label{familyofsubsolutions}v_t(x) \colonequals (1+c_1\varepsilon)\gamma(x_n - \varepsilon c_2\psi(x) + t\varepsilon)^+ - g(0)\gamma(x_n - \varepsilon c_2\psi(x) + t\varepsilon)^-,\; x\in \overline{B}_{3/4}(\overline{x}).\end{equation} We will prove later that this is a family of strict subsolutions.

By the claim, $v_0(x) \leq (1+c_1\varepsilon)\gamma x_n^+ - g(0)\gamma x_n^- \leq w(x)$ for $x\in \overline{B}_{3/4}(\overline{x})$. So we can define $t^* = \sup \{t \mid v_t(x) \leq w(x),\; \forall x\in \overline{B}_{3/4}(\overline{x})\}$. If $t^* \geq c_2$ we get $$w(x) \geq v_{c_2}(x) \geq U_\gamma(x_n - \varepsilon c_2\psi + c_2\varepsilon) \geq U_\gamma(x_n + c\varepsilon), \; x\in B_{1/2}(0)$$ where $c \colonequals c_2(1-\sup_{x\in B_{1/2}} \psi)$. This is the desired result. 

Assume, to obtain a contradiction, $t^* < c_2$.  There must be some point $\tilde{x} \in \overline{B}_{3/4}(\overline{x})$ such that $v_{t^*}(\tilde{x}) = w(\tilde{x})$ (and everywhere else in $\overline{B}_{3/4}(\overline{x})$ we have $v_{t^*}(x) \leq w(x)$) .

\medskip

{\bf Case 1:} $\tilde{x} \in \partial B_{3/4}(\overline{x})$. As $\psi(\tilde{x})= 1$, $$v_{t^*}(\tilde{x}) = (1+c_1\varepsilon)\gamma(\tilde{x}_n + (t^*-c_2)\varepsilon)^+ - g(0)\gamma(\tilde{x}_n + (t^*-c_2)\varepsilon)^-$$$$ < (1+c_1\varepsilon)\gamma (\tilde{x}_n)^+ - g(0)\gamma (\tilde{x}_n)^-.$$ Note, $\overline{B}_{3/4}(\overline{x}) \subset \overline{B}_{19/20}$, so the claim implies $w(\tilde{x}) \geq (1+c_1\varepsilon)\gamma (\tilde{x}_n)^+ - g(0)\gamma (\tilde{x}_n)^- > v_{t^*}(\tilde{x}),$ a contradiction.

\medskip

{\bf Case 2:} $\tilde{x} \in \overline{B}_{1/20}(\overline{x})$. Here $\psi \equiv 0$ so $v_{t^*}(\tilde{x}) =(1+ c_1\varepsilon)\gamma (\tilde{x}_n + t^*\varepsilon)^+ < (1+ c_1\varepsilon)\gamma (\tilde{x}_n + c_2\varepsilon)^+,$ as $t^* < c_2$.  But \eqref{insideinequality} implies $w(\tilde{x}) \geq (1+ c_1\varepsilon)\gamma (\tilde{x}_n + c_2\varepsilon)^+$, which is a contradiction.

\medskip

{\bf Case 3:} $\tilde{x} \in A$. If $v_{t}$ is a strict subsolution to the free boundary problem associated with $g$ in $A$, then Lemma \ref{comparisonlem} (the comparison lemma) gives the desired contradiction. 

\medskip

{\bf Proof that $v_t$ is a strict subsolution}: Note that in $(\{v_{t^*} > 0\} \cap A)\cup (\{v_{t^*} < 0\} \cap A)$ we have $\Delta v_{t^*} \geq -m \varepsilon c_2 \Delta \psi \geq m\varepsilon c_2k(n) > 0$ where $m = \gamma \min\{1, \tilde{k}\}$. 

We then need to show that $\{v_{t^*} = 0\}$ is locally the graph of a $C^2$ function. Observe $\{v_{t^*} = 0\} = \{x_n - \varepsilon c_2 \psi(x) + t^*\varepsilon = 0\}$. As $\psi \in C^\infty(\overline{A})$ it suffices to show that $|e_n - \varepsilon c_2 \nabla \psi(x)| \neq 0$ on $A$. But this is accomplished simply by picking $\overline{\varepsilon} < \frac{1}{c_2 M}$ where $M = \sup_{x\in A} |\nabla \psi(x)|$. $M$ depends only on dimension so $\overline{\varepsilon}$ can still be chosen universally.

To verify the boundary condition, let $x_0 \in \{v_{t} = 0\}$ and $\nu$ the unit normal pointing into $\{v_{t} > 0\}$ at $x_0$. Then $g(x_0)(v_{t}^+)_\nu + (v_{t}^-)_\nu = ((1+c_1\varepsilon)g(x_0)\gamma - g(0)\gamma)(e_n -\varepsilon c_2 \nabla \psi)\cdot \nu.$ As $\nu$ points into $\{v_{t} > 0\}$ it must be the case that $(e_n -\varepsilon c_2 \nabla \psi)\cdot \nu > 0$. So it is enough to prove that $(1+c_1\varepsilon)g(x_0) - g(0) > 0$. By assumption $|g(x_0) - g(0)| \leq 10 \varepsilon^2$ which means it suffices to show $c_1 \varepsilon g(x_0) > 10\varepsilon^2$. By picking $\overline{\varepsilon} > 0$ small enough (now depending on $\tilde{k}$)  this is true on $B_1(0)$ and we are done. 
\end{proof}

Using the one-sided Harnack inequality we can prove a two-sided Harnack type inequality. 

\begin{lem}\label{twosidedharnackinequality}[Compare with \cite{silvaferrarisalsa}, Theorem 4.1]
Let $\tilde{k} > 0$ and let $g \in C(B_2(0))$ such that $\inf_{x\in B_2(0)} g(x) \geq \tilde{k}$. Let $w$ be a solution to the free boundary problem associated to $g$ in $B_2(0)$.  Assume $w$ satisfies at some point $x_0 \in B_2$, \begin{equation}\nonumber U^{(0)}_\gamma(x\cdot \nu + a_0) \leq w(x) \leq U^{(0)}_\gamma(x\cdot \nu + b_0), \; \forall x\in B_r(x_0) \subset B_2(0)\end{equation} where $\nu \in \mathbb S^{n-1}, \gamma >0$ and $b_0 - a_0 \leq \varepsilon r, \sup_{x\in B_2}|g(x) - g(0)| \leq \varepsilon^2$ for some $\varepsilon > 0$. 

Then there exists some $\overline{\varepsilon} = \overline{\varepsilon}(n, \tilde{k}) > 0$ such that if $\varepsilon \leq \overline{\varepsilon}$ we can conclude $$U^{(0)}_{\gamma}(x\cdot \nu + a_1) \leq w(x) \leq U^{(0)}_\gamma(x\cdot \nu + b_1), \; \forall x\in B_{r/20}(x_0),$$ where $a_0 \leq a_1 \leq b_1 \leq b_0$ and $b_1 - a_1 \leq (1-c)\varepsilon r$. Here $c = c(n, \tilde{k}) > 0$. 
\end{lem}

\begin{proof}
Without loss of generality $x_0 = 0, r = 1, \nu = e_n$.  There are three cases, each of which produces a universal $0 < \tilde{c} < 1$. Take $c$ to be the minimum of these three.

\medskip

{\bf Case 1:} $a_0 < -1/5$. For small $\varepsilon > 0$ we have $x_n + b_0 < 0$ on $B_{1/10}$. Therefore, by the assumed inequality on $w$, $$0 \leq v(x)\colonequals \frac{w(x) -g(0)\gamma(x_n+a_0)}{g(0)\gamma\varepsilon} \leq 1,\; \forall x\in B_{1/10}.$$ Additionally, $\Delta v = 0$ on $B_{1/10}$.

So by the Harnack inequality there are constants $1 \geq k_1 \geq k_2 \geq 0$ such that $k_1 -k_2 = 1-\tilde{c} < 1$ where $\tilde{c}$ is universal (though $k_1, k_2$ may depend on $w$) and $k_1 \geq v(x) \geq k_2$ on $B_{1/20}$. 

This implies $$U_\gamma(x_n + a_0 + k_2\varepsilon) \leq w(x) \leq U_\gamma(x_n + a_0 + k_1\varepsilon), \; \forall x\in B_{1/20}.$$ Set $a_1 = a_0 + k_2\varepsilon$ and $b_1 = a_0 + k_1\varepsilon$, so that $a_0 \leq a_1 \leq b_1 \leq b_0$ and $b_1 -a_1 \leq (k_1 -k_2)\varepsilon = (1-\tilde{c})\varepsilon$.

\medskip

{\bf Case 2:} $a_0 > 1/5$. In this case $a_0 + x_n > 0$ on $B_{1/10}$ and so $$0 \leq v(x) \colonequals \frac{w(x) - \gamma(x_n + a_0)}{\gamma \varepsilon}\leq 1$$ on $B_{1/10}$. The rest of the argument follows exactly as in {\bf Case 1}. 

\medskip

{\bf Case 3:} $|a_0| < 1/5$. We can rewrite the main assumption as $$U_\gamma(x_n + a_0) \leq w(x) \leq U_\gamma(x_n + a_0 + \varepsilon),\; x\in B_{1}(0).$$ Without loss of generality, assume that \begin{equation}w(\overline{x}) \geq U_\gamma(\overline{x}_n + a_0 + \varepsilon/2)\end{equation} where $\overline{x} = 4e_n/25 - a_0e_n$ (the case with the reverse inequality is similar). 

If $v(x)\colonequals w(x - a_0e_n)$ for $x\in B_{4/5}(0)$, then the above can be rewritten as \begin{equation}\label{equationsforv}
\begin{aligned}
U_\gamma(x_n) \leq v(x) &\leq U_\gamma(x_n + \varepsilon), \; \forall x\in B_{4/5}(0).\\
v(4e_n/25) &\geq U_\gamma(4/25 + \varepsilon/2).
\end{aligned}
\end{equation}

Note that $v$ satisfies the free boundary problem associated to $\tilde{g}$ which is a translate of $g$. Thus we can apply Lemma \ref{harnacklemma} with $\varepsilon/2$ and inside $B_{4/5}$ to get that $$v(x) \geq U_\gamma(x_n + \tilde{c}\varepsilon), x\in \overline{B}_{2/5}(0)\Rightarrow$$$$w(x) \geq U_\gamma(x_n + a_0 + \tilde{c}\varepsilon), x\in \overline{B}_{1/5}(0),$$ for some universal $0 < \tilde{c} < 1$. Letting $a_1 = a_0 + \tilde{c}\varepsilon$ and $b_1 = b_0$ we have $b_1 - a_1 = b_0 -a_0 - \tilde{c}\varepsilon \leq (1-\tilde{c})\varepsilon$.
\end{proof}

With these lemmata in hand we can prove the following regularity result. This will be crucial in the proof of Lemma \ref{iterationstep} (the iterative step).

\begin{cor}\label{convergencecor}[Compare with \cite{silvaferrarisalsa}, Corollary 8.2]
Let $w, \gamma, g, \nu, \varepsilon, x_0$ satisfy the assumptions of Lemma \ref{twosidedharnackinequality} with $r=1$. Define \begin{equation}\label{tildedefinition}\tilde{w}_\varepsilon \colonequals\left\{\begin{aligned} &\frac{w(x) - \gamma x\cdot \nu}{\gamma \varepsilon},\; x\in B_2(0) \cap \{w \geq 0\}\\
&\frac{w(x) - g(0)\gamma x\cdot \nu}{g(0)\gamma \varepsilon},\; x\in B_2(0) \cap \{w < 0\} \end{aligned}\right. \end{equation}

Then $\tilde{w}_\varepsilon$ has a H\"older modulus of continuity at $x_0$ outside the ball of radius $\varepsilon/\overline{\varepsilon}$, i.e. for all $x\in B_1(x_0)$ with $|x-x_0| \geq \varepsilon/\overline{\varepsilon}$ $$|\tilde{w}_\varepsilon(x) - \tilde{w}_\varepsilon(x_0)| \leq C |x-x_0|^\chi$$ where $C, \chi$ depend only $n, \tilde{k}$. 
\end{cor}

\begin{proof}
Let $\nu = e_n$. Repeated application of Lemma \ref{twosidedharnackinequality} gives $$U_\gamma(x_n + a_m) \leq w(x) \leq U_\gamma(x_n + b_m), \; x\in B_{20^{-m}}(x_0),$$ with $b_m - a_m \leq  (1-c)^m\varepsilon$. However, we may only apply Lemma \ref{twosidedharnackinequality} when $m$ is such that $(1-c)^m20^m\varepsilon \leq \overline{\varepsilon}$ (as we are taking $r = 20^{-m}$ at the $m$th step). 

If $20^{-\chi} = (1-c)$ then we have, for each acceptable $m$, that $x\in B_{20^{-m}}(x_0)\backslash B_{20^{-m-1}}(x_0)$ implies $|\tilde{w}_\varepsilon(x) - \tilde{w}_\varepsilon(x_0)| \leq C|x-x_0|^\chi$. As above, $m$ must satisfy $20^{-m} \geq (1-c)^m \frac{\varepsilon}{\overline{\varepsilon}}$, which is true if $20^{-m} \geq  \frac{\varepsilon}{\overline{\varepsilon}}$. So we have the desired continuity outside $B_{\frac{\varepsilon}{\overline{\varepsilon}}}(x_0)$. 
\end{proof}

\subsection{The Transmission Problem and Proof of Lemma \ref{iterationstep}}

In order to prove Lemma \ref{iterationstep}, we will argue by contradiction and analyze the limit of the $\tilde{w}_\varepsilon$ (see  \eqref{tildedefinition}) as $\varepsilon \downarrow 0$. This limit will be the solution to a transmission problem which we introduce now.  

\begin{defin}\label{transmissionproblem} We say that $W\in C(B_\rho)$ is a {\bf classical solution to the transmission problem at 0} in $B_\rho$ if:
\begin{itemize}
\item $W \in C^{\infty}(B_\rho \cap \{x_n \geq 0\})\cap C^\infty(B_\rho \cap \{x_n \leq 0\}) $
\item $W$ satisfies \begin{equation}\label{transmissionequation}\begin{aligned} \Delta W &=0,\; x\in B_\rho(0) \cap \{x_n \neq 0\}\\
 \lim_{t\downarrow 0}W_n(x', t) - \lim_{t\uparrow 0}W_n(x', t) &= 0,\; x\in B_\rho(0) \cap \{x_n = 0\} \end{aligned} \end{equation}
 \end{itemize}
 
 When no confusion is possible, we will simply say that $W$ is a classical solution to the transmission problem or a classical solution to  \eqref{transmissionequation}.
 \end{defin}
 
We can deduce the following immediately from the definition:

\begin{lem}\label{growthoftranssolution}
Let $W$ be a classical solution to the transmission problem in $B_1$.  Then there is a universal constant $C$ and a constant $p$ (which depend on $W$) such that \begin{equation}\label{estimateontransgrowth}
|W(x) - W(0) - (\nabla_{x'}W(0) \cdot x' + px_n^+ - px_n^-)| \leq C\|W\|_{L^\infty(B_1)}r^2,\; \forall x = (x', x_n)\in B_r(0).
\end{equation}
\end{lem}
 
Unfortunately, the conditions of Definition \ref{transmissionproblem} are too difficult to verify directly. It will be more convenient to work with viscosity solutions.

\begin{defin}\label{viscositytransmission} Let $\widetilde{W} \in C(B_\rho)$. We say that $\widetilde{W}$ is a viscosity solution to the transmission problem \eqref{transmissionequation} if:
\begin{itemize}
\item $\Delta \widetilde{W}(x) = 0$, in the viscosity sense, when $x\in \{x_n \neq 0\} \cap B_\rho$.
\item Let $\phi$ be any function of the form $$\phi(x) = A + px_n^+ - qx_n^- + BQ(x-y)$$ where $$Q(x) = \frac{1}{2}[(n-1)x_n^2 - |x'|^2],\; y= (y', 0),\; A\in \R, B> 0$$ and $p - q > 0$. Then $\phi$ cannot touch $\widetilde{W}$ strictly from below at a point $x_0 = (x_0', 0) \in B_\rho$. 
\item  If $p - q < 0$ then $\phi$ cannot touch $\widetilde{W}$ strictly from above on $\{x_n =0\}$. 
\end{itemize}
\end{defin}

The following result allows us to estimate the growth rate of viscosity solutions. We will omit the proof as it is identical to the one provided by De Silva, Ferrari and Salsa in  \cite{silvaferrarisalsa}. 

\begin{thm}\label{viscosityisclassical}[Theorem 3.3 and Theorem 3.4 in \cite{silvaferrarisalsa}]
Let $\widetilde{W}$ be a viscosity solution to \eqref{transmissionequation} in $B_1$ such that $\|\widetilde{W}\|_{L^\infty} \leq 1$. Then, in $B_{1/2}$, $\widetilde{W}$ is actually a classical solution to \eqref{transmissionequation}.  In particular, $\widetilde{W}$ satisfies the estimate \eqref{estimateontransgrowth}. 
\end{thm}

With this machinery in hand we are ready to prove Lemma \ref{iterationstep}.

\begin{proof}[Proof of Lemma \ref{iterationstep}.]
It suffices to assume that $x_0 = 0$ and $\nu = e_n$ (by the rotation invariance of the conditions).  Fix any $r > 0$ small and let $\{\gamma_k\}, \{\varepsilon_k\}, \{w_k\}, \{g_k\}$ be such that $C_1 > \gamma_k > c_1, \varepsilon_k \downarrow 0$ and $w_k$ is a classical solution to the free boundary problem associated to $g_k$. Furthermore, $\inf_{x\in B_1(0)} g_k(x) \geq \tilde{k},\; \sup_{x,y \in B_1(0)}\frac{|g_k(x) - g_k(y)|}{|x-y|^\alpha} < \varepsilon_k^2$ and $w_k(x)$ satisfies \begin{equation}\label{sandwichinequality}U_{\gamma_k}^{(0)}(x_n -\varepsilon_k) \leq w_k(x) \leq U_{\gamma_k}^{(0)}(x_n+ \varepsilon_k),\; x\in B_1(0).\end{equation} However, to obtain a contradiction, assume the desired $\nu_k, \gamma_k'$ do not exist. 

Define $\tilde{w}_k$ as in \eqref{tildedefinition}. Then \eqref{sandwichinequality} implies that $\{\tilde{w}_k = 0\} \rightarrow \{x_n =0\}$ in the Hausdorff distance norm and  $\|\tilde{w}_k\|_{L^\infty} \leq 1$. These observations, combined with Corollary \ref{convergencecor} and the Arzel\`a-Ascoli theorem, show that $\tilde{w}_k\rightarrow \tilde{w}$ uniformly in $C(B_1(0))$ (after passing to subsequences). Furthermore, Corollary \ref{convergencecor} implies that $\tilde{w}$ is a $C^{0,\chi}$ function defined on $B_{1/2}(0)$. 

\medskip

\noindent {\bf Claim:} $\tilde{w}$ {\it is a viscosity solution in $B_{1/2}$ to the transmission problem.}

\medskip

If this is the case, $\tilde{w}$ satisfies the estimate \eqref{estimateontransgrowth}. So there is a $p$ such that $$|\tilde{w}(x) - \tilde{w}(0) - (\nabla_{x'}\tilde{w}(0) \cdot x' + px_n^+ - px_n^-)| \leq Cr^2,\; \forall x = (x', x_n)\in B_r(0).$$  Because $\|\tilde{w}\|_{L^\infty} \leq 1$ we have $|p| < 10$.  We will also pick $r$ small enough so that $8Cr < 1$.  

As $\tilde{w}_k$ converges uniformly to $\tilde{w}$, for large enough $k$ (depending on $r$ possibly) we have \begin{equation}\label{transmissionestimatefork} |\tilde{w}_k(x) - (\nabla_{x'}\tilde{w}(0) \cdot x' + px_n^+ - px_n^-)| \leq 2Cr^2,\; \forall x = (x', x_n)\in B_r(0).\end{equation} Let $\nu_k \colonequals \frac{1}{\sqrt{1+ \varepsilon_k^2 |\nabla_{x'}\tilde{w}(0)|^2}}(\varepsilon_k \nabla_{x'}\tilde{w}(0), 1)$ and $\gamma_k' \colonequals \gamma_k(1+\varepsilon_kp)$. We will now prove \begin{equation}\label{condition A}\tag{A} U_{\gamma_k'}(x\cdot \nu_k - r\frac{\varepsilon_k}{2}) \leq w_k(x) \leq U_{\gamma'_k}(x\cdot \nu_k + r \frac{\varepsilon_k}{2}),\; x\in B_r(0)\end{equation} and also \begin{equation}\label{conditionb}\tag{B}|\gamma_k'-\gamma_k| \leq \tilde{C}\varepsilon_k\gamma_k,\; |e_n - \nu_k| \leq \tilde{C}\varepsilon_k,\end{equation} for some universal $\tilde{C}$. This is the desired contradiction. 

\medskip

\noindent {\bf Proof of \eqref{condition A}:} Assume $w_k(x) \geq 0$ (the other case follows similarly). \eqref{transmissionestimatefork} implies $$(\nabla_{x'}\tilde{w}(0) \cdot x' + px_n^+ - px_n^-) - 2Cr^2 \leq \frac{w_k(x) - \gamma_k x_n}{\gamma_k\varepsilon_k} \leq 2Cr^2 + (\nabla_{x'}\tilde{w}(0) \cdot x' + px_n^+ - px_n^-)$$ for $x\in B_r(0)$. Consider the inequality on the left. Some algebraic manipulation yields $$\gamma_k x_n + \gamma_k\varepsilon_k ((\nabla_{x'}\tilde{w}(0) \cdot x' + px_n^+ - px_n^-) - 2Cr^2) \leq w_k(x),\; \forall x\in B_r(0)\cap \{w_k \geq 0\}.$$ We can rewrite this again to obtain $$\sqrt{1+\varepsilon_k^2|\nabla_{x'}\tilde{w}(0)|^2} U_{\gamma_k'}(x\cdot \nu_k)- \gamma_kp\varepsilon_k^2 |\nabla_{x'}\tilde{w}(0)\cdot x|- 2Cr^2 \gamma_k\varepsilon_k \leq w_k(x),\; \forall x\in B_r(0)\cap \{w_k \geq 0\}.$$

The Cauchy-Schwartz inequality, followed by some more algebraic manipulation, gives $$U_{\gamma'_k}(x\cdot \nu_k) - \gamma'_k r \frac{\varepsilon_k}{2}\left(\frac{2p\varepsilon_k |\nabla_{x'} \tilde{w}(0)| + 4Cr}{1+\varepsilon_kp} \right)\leq w_k(x),\; \forall x\in B_r(0)\cap \{w_k \geq 0\}.$$ Recall that $r$ was chosen so that $8Cr < 1$.  Now pick $k$ large enough so that $20\varepsilon_k |\nabla_{x'}\tilde{w}(0)| < 1/2$. Together this implies $\left(\frac{2p\varepsilon_k |\nabla_{x'} \tilde{w}(0)| + 4Cr}{1+\varepsilon_kp} \right) < 1$. In conclusion, $$U_{\gamma'_k}(x\cdot \nu_k-r\frac{\varepsilon_k}{2}) \leq w_k(x), \; \forall x\in B_r(0)\cap \{w_k \geq 0\}.$$ The upper bound on $w_k$ and the inequalities for when $w_k < 0$ follow in the same fashion.  

\medskip

\noindent {\bf Proof of \eqref{conditionb}:} We compute  $|\gamma_k'-\gamma_k| = \varepsilon_k p \gamma_k \leq 10 \varepsilon_k \gamma_k$.  Also $|\nu_k- e_n|^2 = (\nu_k - e_n, \nu_k-e_n) = 2 -2(e_n, \nu_k) = 2-\frac{2}{\sqrt{1+\varepsilon_k^2|\nabla_{x'}\tilde{w}(0)|^2}}$. For large $k$ (so that $\varepsilon_k|\nabla_{x'}\tilde{w}(0)| < 1/2$)  the taylor series expansion of $\sqrt{1+x^2}$ yields the estimate $|\nu_k - e_n|^2 \leq \varepsilon_k^2|\nabla_{x'}\tilde{w}(0)|^2$. Let $\tilde{C} = \max\{ |\nabla_{x'}\tilde{w}(0)|, 10\}$ and we are done. 

\medskip

{\bf Proof of Claim:} We want to establish that $\tilde{w}$ is a viscosity solution to the transmission problem. As $\Delta \tilde{w}_k = 0$, wherever $\{\tilde{w}_k \neq 0\}$, it is clear that $\Delta \tilde{w} = 0$, in the viscosity sense, when $\{x_n \neq 0\}$. It remains to verify the boundary condition. 

So assume, in order to reach a contradiction, that there is a function $$\tilde{\phi}(x) \colonequals A + px_n^+ - qx_n^- + BQ(x-y),$$ with $p-q > 0$, which touches $\tilde{w}$ strictly from below at $x_0 = (x_0', 0)$ (the case where $p-q < 0$ and $\tilde{\phi}$ touches from above follows similarly). Recall $Q(x) \colonequals \frac{1}{2}[(n-1)x_n^2 - |x'|^2], y = (y',0), B > 0$ and $A \in \R$. We now construct a family of functions which converge uniformly to $\tilde{\phi}$. Define $$\Gamma(x) \colonequals \frac{1}{n-2}[(|x|'^2 + |x_n - 1|^2)^{\frac{2-n}{2}} - 1]\: \mathrm{and}\:\Gamma_k(x) \colonequals \frac{1}{B\varepsilon_k} \Gamma(B\varepsilon_k(x-y) + AB\varepsilon_k^2e_n).$$ Additionally, let $$\phi_k(x) \colonequals \gamma_k(1+\varepsilon_kp)\Gamma_k^+(x) - g(0)\gamma_k(1+\varepsilon_kq)\Gamma^-_k(x) + \gamma_k(d_k^+(x))^2\varepsilon_k^{3/2} + g(0)\gamma_k(d^-_k(x))^2\varepsilon_k^{3/2},$$ where $d_k$ is the signed distance from $x$ to $\partial B_{\frac{1}{B\varepsilon_k}}(y+ e_n(A\varepsilon_k-\frac{1}{B\varepsilon_k}))$. Finally, we can define $\tilde{\phi}_k$ as in \eqref{tildedefinition}.  

A taylor series expansion gives $\Gamma(x) = x_n + Q(x) + O(|x|^3)$ and thus $$\Gamma_k(x) = A\varepsilon_k + x_n + B\varepsilon_kQ(x-y) + O(\varepsilon_k^2),\; x\in B_1.$$  Therefore, $\tilde{\phi}_k$ converges uniformly to $\tilde{\phi}$. The existence of a touching point $x_0$ implies a sequence of constants, $c_k$, and points, $x_k\in B_{1/2}$, such that $\psi_k(x) \colonequals \phi_k(x + \varepsilon_kc_ke_n)$ touches $w_k$ from below at $x_k$. We will get the desired contradiction if $\psi_k$ is a strict subsolution to the free boundary problem associated to $g_k$. 

When $\psi_k \neq 0$ we have $\Delta \psi_k \gtrsim \Delta d_k^2(x + \varepsilon_kc_ke_n) > 0$.  If $\psi_k = 0$ a straightfoward computation shows $\Gamma_k(x+\varepsilon_kc_ke_n) = d_k(x + \varepsilon_kc_ke_n) = 0$. Thus, $|\nabla d_k^2| = 0$ whenever $\psi_k = 0$. We can also compute $(\nabla \Gamma_k^\pm)_\nu = \pm 1$ on $\psi_k = 0$. Putting this together, $g_k(x)(\psi_k(x)^+)_\nu + (\psi_k(x)^-)_\nu = g_k(x)\gamma_k(1+\varepsilon_kp) - g(0)\gamma_k(1+\varepsilon_kq)$. Recall, $|g_k(x) - g(0)| = |g_k(x)-g_k(0)| \leq \varepsilon_k^2$ which implies, $g_k(x) \geq g(0) - \varepsilon_k^2$. Therefore, $g_k(x)(\psi_k(x)^+)_\nu + (\psi_k(x)^-)_\nu \geq g(0)\gamma_k\varepsilon_k(p-q) - \varepsilon_k^2\gamma_k(1+\varepsilon_kp)$. We are done if this last term is $> 0$. It is easy to see $$g(0)\gamma_k\varepsilon_k(p-q) - \varepsilon_k^2\gamma_k(1+\varepsilon_kp) > 0 \Leftrightarrow g(0)(p-q) > \varepsilon_k(1+\varepsilon_kp)$$ which is clearly true for $k$ large enough. 

\end{proof}

\section{Optimal H\"older regularity and higher regularity}\label{sec: higherholder}

 Proposition \ref{initialholderregularity} tells us that if $\log(h) \in C^{0,\alpha}(\partial \Omega)$ then $\partial \Omega$ is locally the graph of a $C^{1,s}$ function for some $s > 0$. In this section we will introduce tools from elliptic regularity theory in order to establish the sharp estimate $s = \alpha$. These tools will also allow us to analyze the case when $\log(h) \in C^{k,\alpha}(\partial \Omega)$ for $k \geq 1$.

 \subsection{Partial Hodograph Transform and Elliptic Systems} We begin by recalling the partial hodograph transform (see \cite{kinderlehrerstampacchia}, Chapter 7 for a short introduction).  Here, and throughout the rest of the paper, we assume that $0 \in \partial \Omega$ and that, at 0, $e_n$ is the inward pointing normal to $\partial \Omega$.
 
Define $F^+: \Omega^+ \rightarrow \mathbb H^+$ by $(x', x_n) = x \mapsto y = (x', u^+(x))$. Because $u^+_n(0) = \frac{d\omega^+}{d\sigma}(0) \neq 0$ (Proposition \ref{nondegeneracy}), $DF^+(0)$ is invertible. So, by the inverse function theorem, there is some neighborhood, $\mathcal O^+$, of $0$ in $\Omega^+$ that is mapped diffeomorphically to $U$, a neighborhood of $0$ in the upper half plane. Furthermore, this map extends in a $C^1$ fashion from $\overline{\mathcal O^+}$ to $\overline{U}$ (by Corollary \ref{c1extension}).

Similarly, define $F^-: \Omega^- \rightarrow \mathbb H^+$ by $(x', x_n) = x\mapsto y = (x', u^-(x))$. Again $u^-_n(0) \neq 0$ so $DF^-(0)$ is invertible. We can conclude, as above, that there is a neighborhood, $\mathcal O^-$, of $0$ in $\Omega^-$ that is mapped diffeomorphically to $U$ (perhaps after shrinking $U$) and that this map extends in a $C^1$ fashion from $\overline{\mathcal O^-}$ to $\overline{U}$.

Let $\psi: \overline{U} \rightarrow \R$ be given by $\psi(y) = x_n$, where $F^+(x) = y$. Because $F^+$ is locally one-to-one, $\psi$ is well defined. Similarly, define $\phi: \overline{U} \rightarrow \R$ by $\phi(y) = -x_n$ where $F^-(x) = y$. Again, $F^-$ is locally one-to-one, so $\phi$ is well defined. 

If $\nu_y$ denotes the normal vector to $\partial \Omega$ pointing into $\Omega$ at $y$, then $u$ satisfies \begin{eqnarray}
\Delta u^+(x) &=& 0, \; x\in \Omega^+\nonumber\\
\Delta u^-(x) &=& 0,\; x\in \Omega^-\nonumber\\
(u^+)_{\nu_x}(x)h(x) &=& -(u^-)_{\nu_x}(x),\; x\in \partial \Omega.\nonumber
\end{eqnarray}

After our change of variables these equations become \begin{equation}\label{transformedequations}
\begin{aligned}
0 &= \frac{1}{2}\left(\frac{1}{\psi_n^2}\right)_n + \sum_{i=1}^{n-1} \left(-\left(\frac{\psi_i}{\psi_n}\right)_i + \frac{1}{2}\left(\frac{\psi_i^2}{\psi_n^2}\right)_n\right)\\
0&= \frac{1}{2}\left(\frac{1}{\phi^2_n}\right)_n + \sum_{i=1}^{n-1} \left(-\left(\frac{\phi_i}{\phi_n}\right)_i + \frac{1}{2}\left(\frac{\phi_i^2}{\phi_n^2}\right)_n\right),
\end{aligned}
\end{equation} with both equations taking place for $y\in U$. On the boundary we have  \begin{equation}\label{transformedboundary}\begin{aligned}
\phi(y) + \psi(y)  &= 0,\; y\in \{y_n = 0\}\cap \overline{U}\\
\left(\frac{\tilde{h}(y)}{\psi_n(y)}\right) - \frac{1}{\phi_n(y)} &= 0, \; y\in \{y_n =0\}\cap \overline{U},
\end{aligned}\end{equation} where $\tilde{h}((y',0)) = h((y', \psi(y)))$. 

\begin{rem}\label{transformremarks} The following are true of $\phi, \psi$:
\begin{itemize}
\item Assume $\psi, \phi \in C^{k,s}(\overline{U}\cap \{y_n =0\})$ with $k\geq 1, s\in (0,1)$  Then $u^\pm \in C^{k,s}(\overline{\mathcal O}^\pm) \Leftrightarrow \psi, \phi \in C^{k,s}(\overline{U})$. 
\item If $h \in C^{k,\alpha}(\partial \Omega)$ and $\psi, \phi \in C^{k+1,s}(\overline{U})$  for any $s,\alpha \in (0,1)$, then $\tilde{h} \in C^{k,\alpha}(\overline{U}\cap\{y_n = 0\})$ with norm depending only on the H\"older norms of $h$ and $\psi, \phi$. 
\item $\phi_n, \psi_n > 0$ in $\overline{U}$. 
\end{itemize}
\end{rem}

\begin{proof}
Let us address the first statement; when $k \geq 2$ this follows from standard elliptic regularity applied to the function $\tilde{u}^+(x) = u^+(x+ \phi(x',0))$ (and a similarly defined $\tilde{u}^-$).  When $k = 1$, a theorem of Kellogg \cite{kellogg} says that $\nabla u^{\pm}$ has non-tangential limit everywhere on $\partial \Omega\cap \overline{\mathcal O^{\pm}}$ and that this non-tangential limit is in $C^{0,s}$. We can then argue as in the proof of Corollary \ref{c1extension} to see that $\nabla u^{\pm} \in C^{0,s}(\overline{\mathcal O}^{\pm})$; the desired result.

To prove the second statement when $k = 0$, one computes $|\tilde{h}(y_1)-\tilde{h}(y_2)| = |h((y_1', \psi(y_1))) - h((y_2', \psi(y_2)))| \leq C |(y_1', \psi(y_1)) - (y_2', \psi(y_2))|^\alpha \leq C'|y_1-y_2|^\alpha$ where that last inequality follows because $\psi \in C^{1,s}(\overline{U})$. So $\tilde{h} \in C^{0,\alpha}(\{y_n =0\}\cap \overline{U})$. When $k \geq 1$ we note that $\partial_i \tilde{h}(y, 0) = \partial_i h(y, \psi(y)) + \partial_n h(y,\psi(y)) \partial_i \psi(y)$. By assumption  $\partial_i \psi(y)$ is at least as regular as $\partial_n h(y,\psi(y))$ so the result follows by induction.

The third claim follows immediately from construction. 
\end{proof}

We now recall the concepts of an elliptic system of equations and coercive boundary conditions.  For the sake of brevity, our Definition \ref{weakellipticsystem} is not fully general---it considers only a specific type of system in ``divergence form". A comprehensive introduction to elliptic systems can be found in Morrey (\cite{morrey}), Chapter 6 (weak solutions in particular are covered in Section 6.4).

\begin{defin}\label{weakellipticsystem}
Let $u^k$, $k = 1,2$, satisfy \begin{equation}\label{weakformulation}
\begin{aligned}
\int_{U} \sum_{\stackrel{|\chi| \leq m_1}{|\gamma| \leq t_1 + s_1-m_1}} a^1_{\chi\gamma}(x) D^\gamma u^1 D^\chi \zeta &= \int_U \sum_{|\chi| \leq m_1} f_\chi^1 D^\chi\zeta\\
\int_{U} \sum_{\stackrel{|\chi| \leq m_2}{|\gamma| \leq t_2 + s_2-m_2}} a^2_{\chi\gamma}(x) D^\gamma u^2 D^\chi \zeta &= \int_U \sum_{|\chi| \leq m_2} f_\chi^2 D^\chi\zeta
\end{aligned}
\end{equation}
for all $\zeta \in C_0^\infty(U)$. Additionally assume,
\begin{equation}\label{weakformulationbdry}
\begin{aligned}
\int_{\partial U\cap \{y_n = 0\}} \sum_{|\chi| \leq p_1}\left(\sum_{k=1}^2 B^1_{k\chi}(D_x, D_y, x) u^k\right) D_x^\chi \xi dx &= \int_{\partial U\cap \{y_n = 0\}} \sum_{|\chi| \leq p_1} g_\chi^1  D_x^\chi \xi dx\\
\int_{\partial U\cap \{y_n = 0\}} \sum_{|\chi| \leq p_2}\left(\sum_{k=1}^2 B^2_{k\chi}(D_x, D_y, x) u^k\right) D_x^\chi \xi dx &= \int_{\partial U\cap \{y_n = 0\}} \sum_{|\chi| \leq p_2} g_\chi^2  D_x^\chi \xi dx
\end{aligned}
\end{equation}
for all $\xi \in C_0^\infty(\partial U\cap \{y_n = 0\})$. Throughout, $\gamma,\chi$ are multi-indices.  Let $h_1, h_2,$ be such that $B^1_{k\chi}$ is of order $\leq t_k - h_1-p_1$ and $B^2_{k\chi}$ is of order $\leq t_k - h_2-p_2$. This system has {\bf a proper assignment of weights} if there exists an $h_0$ such that $h_0$ and the $t_k, m_j, s_j, h_r, p_r$, $k,j,r = 1,2$ satisfy the following conditions: 
\begin{itemize}
\item $\min_{j,k} s_j + t_k \geq 1$ and $\min_{j,k} t_k + s_j -m_j \geq 0$
\item  $\min m_j \geq 0$ and $\max s_j = 0$. 
\item $\min p_r \geq 0$ and $\min h_0 + h_r + p_r \geq 1$
\item $\min t_k + h_0 \geq 0$ and $\min h_0 -s_j + m_j \geq 0$. 
\end{itemize}

We say the above system is {\bf elliptic} if the block diagonal matrix $$M = \left( \begin{array}{cc}
(a^1_{\gamma\chi})_{|\chi| = m_1, |\gamma| = t_1 + s_1-m_1}  & 0 \\
0 & (a^2_{\gamma\chi})_{|\chi| = m_2, |\gamma| = t_2 + s_2-m_2}  \\
 \end{array} \right)$$ is an elliptic matrix for any $x_0 \in U$. Additionally, when $n = 2$, we require that, for any linearly independent $\xi, \eta \in \R^2$, half the roots of the equation $$\det \left( \begin{array}{cc}
\sum_{|\chi| = m_1, |\gamma| = t_1 + s_1-m_1} a^1_{\gamma\chi}\cdot(\xi + z\eta)_{\chi + \gamma}  & 0 \\
0 & \sum_{|\chi| = m_2, |\gamma| = t_2 + s_2-m_2}a^2_{\gamma\chi}\cdot(\xi + z\eta)_{\chi + \gamma}  \\
 \end{array} \right) = 0$$ have positive imaginary part and the other half have negative imaginary part. 
 
Finally, we say that the boundary equations are {\bf coercive} if for all $y_0 \in \overline{U}\cap \{y_n = 0\}$ the system \begin{equation}\label{coerciveboundarycondition}
\begin{aligned}
\sum_{|\chi| = m_1, |\gamma| = t_1 + s_1 - m_1}a^1_{\chi\gamma}(y_0) D^{\gamma+\chi} v^1(y) &= 0\\
\sum_{|\chi| = m_2, |\gamma| = t_2 + s_2 - m_2}a^2_{\chi\gamma}(y_0) D^{\gamma+\chi} v^2(y) &=0\\
\sum_{|\chi| = p_1} \sum_{k=1}^2 \tilde{B}^1_{k\chi}(D_x, D_y, y_0) v^k((y',0)) &= 0\\
\sum_{|\chi| = p_2} \sum_{k=1}^2 \tilde{B}^2_{k\chi}(D_x, D_y, y_0) v^k((y',0)) &= 0
\end{aligned}
\end{equation}
has no solutions of the form $v^k((y', y_n)) = e^{i y' \cdot \xi'}\tilde{v}^k(y_n), k = 1,2$ where $\tilde{v}^k(y_n)\rightarrow 0$ as $y_n \rightarrow +\infty$ and $\xi' \in \R^{n-1}$. Above, $\tilde{B}^r_{k\chi}$ denotes the part of the operator $B^r_{k\chi}$ which has order $t_k-h_r-p_r$ (the {\bf principle part}).
\end{defin}

\begin{defin}\label{regularityoncoefficients}
We define the $h-\mu$-conditions on the coefficients above:

\noindent {\bf (1) The $a^j_{\chi\gamma}$ satisfy the $h-\mu$-conditions}, $0 < \mu < 1$, in some open $\Gamma$:
\begin{enumerate}
\item  if $|\gamma| = t_j + s_j - m_j$ and $|\chi| = m_j$ then $a^j_{\chi\gamma} \in C^{0,\mu}(\overline{\Gamma})$
\item if  $h - s_j + |\chi| > 0$ then $a^j_{\chi\gamma} \in C^{h-s_j + |\chi|, \mu}(\overline{\Gamma})$
\item else, the $a$s are in $C^{0,\mu}(\overline{\Gamma})$.
\end{enumerate}
\medskip
\noindent {\bf (2) The operators $B^r_{k\chi}$ satisfy the $h-\mu$-conditions}, $0 < \mu < 1$, in some open $\Gamma$, if $B^r_{k\gamma}(D_x, D_y, -)\in C^{h+h_r + p_r, \mu}( \overline{\Gamma\cap \{y_n = 0\}})$.
\end{defin}

With these definitions in mind, we can state Theorem 6.4.8 of \cite{morrey} (note the theorem in Morrey refers to a slightly more general class of elliptic systems). Our wording differs in order to comport with the notation used above.

\begin{thm}\label{morreytheorem1}[Theorem 6.4.8, \cite{morrey}]
Let $u^k, k=1,2$ satisfy an elliptic and coercive system of equations on $U$ (a neighborhood of $0$ in the upper half plane with $C^\infty$ boundary) with a proper assignment of weights $h_0, h_r, p_r, t_k, s_j, m_j$. Let $\Gamma \supset \overline{U}$ be an open domain. Suppose the $a$'s and the coefficients in the $B_{rk\gamma}$ satisfy the $h-\mu$-conditions on $\Gamma$, $0 < \mu < 1$, and suppose the {\rm a priori} estimates: $f_{\alpha}^j \in C^{\rho, \mu}(U)$, $\rho = \max\{0, h-s_j+|\alpha|\}, g^r_{\gamma}\in C^{\tau, \mu}(U)$ with $\tau = \max\{0, h+h_r + |\gamma|\}$ and $u^k \in C^{t_k + h, \mu}(U)$. Then \begin{equation}\label{weakholderestimate}\sum_k \|u^k\|_{C^{t_k + h, \mu}(U)} \leq C\left(\sum_{j,\alpha} \|f^j_\alpha\|_{C^{\rho,\mu}(U)} + \sum_{r,\gamma} \|g^r_{\gamma}\|_{C^{\tau,\mu}(U)} + \sum_{k}\|u^k\|_{C^0(U)}\right).\end{equation} Here $C$ is, again, independent of $u^k$ the $f$'s and the $g$'s.
\end{thm}

\subsection{Sharp $C^{1,\alpha}$ regularity and $C^{2,\alpha}$ regularity} It should be noted that in \cite{morrey} it is not explicitly made clear if Theorem \ref{morreytheorem1} applies when $h < h_0$ (nor if there should be additional restrictions on $h$). For the sake of completeness we include a proof of Theorem \ref{morreytheorem1} with $h_0 = 0, h = -1$ in Appendix \ref{sec: prooffornegativeh}. This is exactly the result we need to establish optimal $C^{1,\alpha}$ regularity. 

\begin{prop}\label{optimalregularity}
Let $\Omega \subset \R^n$ be a 2-sided NTA domain with $\log(h) \in C^{0,\alpha}(\partial \Omega), \alpha \in (0,1)$. In addition, if $n\geq 3$ also assume that $\Omega$ is $\delta$-Reifenberg flat, for $\delta > 0$ small, or that $\Omega$ is a Lipschitz domain. Then $\partial \Omega$ is locally the graph of a $C^{1, \alpha}$ function. 
\end{prop}

\begin{proof}
Recall the functions $\phi, \psi$ which satisfy the system \eqref{transformedequations} with boundary conditions \eqref{transformedboundary}. For $t = (t', 0) \in \R^n$ we consider $u^{1,t}(x) \colonequals \psi(x+t) - \psi(x)$ and $u^{2,t}(x) \colonequals \phi(x+t) - \phi(x)$; our plan is to show that $u^{1,t}, u^{2,t}$ satisfy a system like the one in Definition \ref{weakellipticsystem}. Repeated applications of Theorem \ref{morreytheorem1} will then give the desired result. Our proof has three steps. 

\medskip

{\bf Step 1: constructing the elliptic and coercive system} Both $\phi$ and $\psi$ satisfy $$\mathrm{div} \vec{A}(Du) = 0$$ where $\vec{A}(Du)\colonequals \left(-\frac{u_1}{u_n}, -\frac{u_2}{u_n},..., \frac{1}{2}\left(\sum_{i=1}^{n-1} \left(\frac{u_i}{u_n}\right)^2 + \frac{1}{u_n^2}\right)\right)$. As such $$\mathrm{div} \int_0^1 \frac{d}{ds} \vec{A}(D(\psi(x) + s(\psi(x+t)-\psi(x)))) ds= 0\Rightarrow $$$$\mathrm{div} \int_0^1 a_{ij}(D(\psi(x) + s(\psi(x+t)-\psi(x))))D_iu^{1,t}(x)ds = 0$$ where $a_{ij}(\vec{p}) = \frac{d}{dp_j}A_i(\vec{p})$. $\phi$ and $u^{2,t}$ satisfy an analogous equation. Therefore, $u^{1,t}, u^{2,t}$ satisfy \eqref{weakformulation} with $a^1_{ij}(x) \colonequals a_{ij}(D\psi(x))$ and $$f^1_j\colonequals \sum_i\left(a_{ij}(D\psi(x))- \int_0^1 a_{ij}(D(\psi(x) + s(\psi(x+t)-\psi(x))))ds\right)D_iu^{1,t}$$ (and with corresponding definitions for $f^2, a^2$ in terms of $\phi$). Note $m_1= m_2 = 1, t_1 = t_2 = 2$ and $s_1 = s_2 = 0$. On the boundary $u^{1,t} + u^{2,t} = 0$ and $\frac{\tilde{h}(x)}{\phi_n(x+t)}D_nu^{2,t} - \frac{1}{\phi_n(x+t)}D_nu^{1,t} = \tilde{h}(x) - \tilde{h}(x+t)$. Therefore, $h_1 = 2, h_2 = 1$ and $p_1 = p_2 = 0$. Set $h_0 = 0$. It is then easy to see that this is a system with a proper assignment of weights. We will check in Step 3 that our system satisfies the ellipticity, coercivity and regularity conditions of Definition \ref{regularityoncoefficients}.

\medskip

{\bf Step 2: the iterative process} By Proposition \ref{initialholderregularity}, $u^{i,t} \in C^{1,s}(\overline{U})$. In particular, the $a^k_{ij}$'s and the $B$s satisfy the $h-\mu$-conditions with $h = -1$ and $\mu = s$. It is also easy to see that the $f$'s and $g$'s satisfy the conditions of Theorem \ref{morreytheorem1} (we assume, of course, that $\alpha \geq s$; otherwise the result is immediate). We conclude \begin{equation}\label{seconddifferenceestimates}
\|u^{i,t}\|_{C^{1,s}(\overline{U})} \leq K\left(\sum_{i=1}^2\sum_{j=1}^n \|f_j^i\|_{C^{0, s}(\overline{U})} + \|\tilde{h}(-)-\tilde{h}(-+t)\|_{C^{0, s}(\overline{U}\cap \{y_n = 0\})} + \sum_{k=1}^2 \|u^{k,t}\|_{C^0(\overline{U})}\right),
\end{equation} where $K$ is a constant independent of $t$. Some additional justification is needed here: in Theorem \ref{morreytheorem1} the constant may depend on the $C^{0,s}$ norm of the $a$'s and $B$'s. However, these coefficients have norms which can be bounded independently of $t$ and so $K$ may be taken to be independent of $t$.

For any $x,y\in U,$ \begin{equation}\label{tildehestimate}\begin{aligned}2\|\tilde{h}\|_{C^\alpha} |x-y|^s |t|^{\alpha-s} &\geq \min\{2|x-y|^\alpha \|\tilde{h}\|_{C^\alpha}, 2|t|^\alpha\|\tilde{h}\|_{C^{\alpha}}\}\\ &\geq |\tilde{h}(x) -\tilde{h}(x+t) - \tilde{h}(y) +\tilde{h}(y+t)|.\end{aligned}\end{equation} Thus $\|\tilde{h}(-)-\tilde{h}(-+t)\|_{C^{0, s}(\overline{U}\cap \{y_n = 0\})}\leq C|t|^{\alpha-s}$. We also claim that if $w,v\in C^{0,s}$ then $$\|(w(-) - w(-+t))(v(-)- v(-+t))\|_{C^{0,s}} \leq 4|t|^s\|w\|_{C^{0,s}}\|v\|_{C^{0,s}}$$ (this follows immediately from the triangle inequality and the fact that $\sup |w(-)-w(-+t)| < |t|^s\|w\|_{C^{0,s}}$). From here we conclude $\|f^i_j\|_{C^{0,s}(\overline{U})} \leq C|t|^s$. Plugging these estimates into \eqref{seconddifferenceestimates} we obtain $\|u^{i,t}\|_{C^{1,s}(\overline{U})} \leq K(|t|^s + |t|^{\alpha-s} + |t|)$ (as $\|u^{k,t}\|_{C^0(\overline{U})}\leq C|t|$). 

Therefore, for $j = 1,...,n$, we have that \begin{equation}\label{djpsiestimate}\begin{aligned}|D_j \psi(x+t) + D_j \psi(x-t) - 2D_j \psi(x)|&= |D_j u^{1,t}(x) - D_ju^{1,t}(x-t)|\\  &\leq \|u^{1,t}\|_{C^{1,s}} |t|^s \leq K(|t|^{2s} + |t|^\alpha).\end{aligned}\end{equation} This implies $\psi|_{\overline{U}\cap \{y_n = 0\}} \in C^{1, \beta}$ where $\beta = \min\{\alpha, 2s\}$ (see \cite{steinsingularintegrals}, Chapter 5, Proposition 8). Remark \ref{transformremarks} gives $\psi,\phi \in C^{1,\beta}(\overline{U})$.  Iterate until $\beta = \alpha$.

\medskip

{\bf Step 3: verifying the conditions of Definition \ref{weakellipticsystem}} It is easy to calculate the symmetric $(n\times n)$-matrix $$DA(\vec{p}) =  \left( \begin{array}{ccccc}
\frac{-1}{p_n} & 0 & 0 &\ldots & \frac{p_1}{p_n^2} \\
0 & \frac{-1}{p_n} & 0 &\ldots & \frac{p_2}{p_n^2} \\
\vdots & 0 & \ddots & \ldots & \vdots \\
\frac{p_1}{p_n^2} & \ldots & \frac{p_i}{p_n^2} & \ldots & -\left(\frac{1}{p_n}\right)^3\left(1+ \sum_{i=1}^{n-1} p_i^2\right)  \end{array} \right). $$ If $\vec{p} = D\phi, D\psi,$ then $p_n > 0$ in $\overline{U}$. Thus the matrices $DA(D\phi)$ and $DA(D\psi)$ are both elliptic (justifying our above use of the Schauder estimates) and the system is also elliptic (with the obvious weights $t_1 = t_2 = 2, s_1 = s_2 = 0$). Addtionaly, when $n=2$ we have the equation $$-\frac{1}{p_2} (\xi_1 + z\eta_1)^2 + 2\frac{p_1}{p_2^2}(\xi_1 + z\eta_1)(\xi_2 + z\eta_2) - \frac{1}{p_2^3}(1+p_1^2)(\xi_2+z\eta_2)^2 = 0.$$ All the coefficients of this polynomial are real, so if $\alpha, \beta$ are its roots it must be the case that $\alpha = \overline{\beta}$ which is exactly the desired result. 

We must check coercivity at an arbitrary $y_0 \in \overline{U} \cap \{y_n = 0\}$. If $u^1 = e^{i y' \cdot \xi'} \tilde{u}^1(y_n)$ solves $a_{ij}(D\psi(y_0))D_{ij} u^1 = 0$ then $\tilde{u}^1(y_n)$ is a linear combination of functions of the form $e^{ry_n}$ where $r$ is a root of  $$\frac{\sum |\xi'|^2}{p_n} + 2\frac{p_j}{p_n^2}\sum i \xi'_j x- \frac{1}{p_n^3}(1+\sum p_i^2)x^2 = 0.$$ This equation has at most one root, call it $r_1$, with strictly negative real part (as the sum of the roots is purely imaginary). That $\tilde{u}^1(y_n)\rightarrow 0$ as $y_n \rightarrow \infty$ implies $\tilde{u}^1(y_n) = \alpha_1 e^{y_n r_1}$. Similarly, we define $\tilde{u}^2(y_n)$ and conclude $\tilde{u}^2(y_n) = \alpha_2e^{y_n r_2}$, where $r_2$ has strictly negative real part (if such an $r_1$ or $r_2$ does not exist then we are done). 

As $u^1 + u^2 = 0$ on the boundary it must be true that $\alpha_1 + \alpha_2 = 0$. Furthermore $$\tilde{h}(0)D_nu^2 - D_nu^1 = 0\Rightarrow$$$$\tilde{h}(0)\alpha_2r_2 - \alpha_1r_1 = 0 \Rightarrow \tilde{h}(0)r_2 + r_1 = 0.$$ But $\tilde{h}(0)r_2$ has strictly negative real part and $r_1$ has strictly negative real part, so their sum must have strictly negative real part and the system is coercive. 
\end{proof}

 If $\log(h) \in C^{k,\alpha}$ for $k \geq 1$, the above argument can be modified slightly to give that $\partial \Omega$ is locally the graph of a $C^{2,\frac{\alpha}{2+\alpha}}$ function. 

\begin{prop}\label{c2regularity}
Let $\partial \Omega$ be a 2-sided NTA domain with $\log(h) \in C^{1,\alpha}(\partial \Omega)$ for $0 < \alpha < 1$. If $n \geq 3$ also assume either that $\Omega$ is $\delta$-Reifenberg flat for $\delta > 0$ small or that $\Omega$ is a Lipschitz domain. Then $\partial \Omega$ is locally the graph of a $C^{2, \frac{\alpha}{2+\alpha}}$ function.
\end{prop}

\begin{proof}
We follow the proof of Proposition \ref{optimalregularity}; consider again $u^{1,t}, u^{2,t}$.  We have already shown these functions satisfy an elliptic system with coercive boundary conditions.  Note, by Proposition \ref{optimalregularity}, $u^{i,t} \in C^{1,s}(\overline{U})$ for all $s\in (0,1)$. In particular, the $a_{ij}^k$'s and the $B$s satisfy the $h-\mu$-conditions with $h = -1$ and $\mu = s\in (0,1)$ to be choosen later. Furthermore, the $f$'s and $g$'s satisfy the conditions of Theorem \ref{morreytheorem1}.

Follow {\bf Step 2} in the proof of Proposition \ref{optimalregularity} until we reach \eqref{tildehestimate}. Here we need an estimate which incorporates the higher regularity of $\tilde{h}$. By Remark \ref{transformremarks}, $\tilde{h} \in C^{1,\alpha}(\overline{U}\cap \{y_n =0\})$. For any $x, y\in \mathbb R^n$ write, for the sake of brevity, $$\delta^2_y f(x) \equiv f(x+y) + f(x-y) - 2f(x).$$ We can then estimate, for $x, y\in \overline{U}\cap \{y_n =0\}$, \begin{equation}\label{secondhtildeestimate}\begin{aligned} |\delta^2_y\tilde{h}(x + t) -\delta_y^2 \tilde{h}(x)| &\leq \|\tilde{h}\|_{C^{1+\alpha}}\min\{3|t|, 2|y|^{1+\alpha}\}\\
& \leq C\|\tilde{h}\|_{C^{1+\alpha}} |y|^s|t|^{1-\frac{s}{1+\alpha}}.\end{aligned}\end{equation} Consequently, $\|\tilde{h}(-) - \tilde{h}(-+t)\|_{C^{0,s}} \leq C|t|^{1-\frac{s}{1+\alpha}}.$

Proceed as in {\bf Step 2} of the proof of Proposition \ref{optimalregularity} until we reach \eqref{djpsiestimate}, which now reads  \begin{equation}\label{djpsiestimate2}\begin{aligned}|D_j \psi(x+t) + D_j \psi(x-t) - 2D_j \psi(x)|&= |D_j u^{1,t}(x) - D_ju^{1,t}(x-t)|\\  &\leq \|u^{1,t}\|_{C^{1,s}} |t|^s \leq K(|t|^{2s} + |t|^{1+s-\frac{s}{1+\alpha}}).\end{aligned}\end{equation} Pick $s\in (0,1)$ such that $$1+s -\frac{s}{1+\alpha} = 2s\Rightarrow s  = \frac{1+\alpha}{2+\alpha}.$$ Then $\psi|_{\overline{U}\cap \{y_n = 0\}} \in C^{2, \frac{\alpha}{2+\alpha}}$. By Remark \ref{transformremarks} we can conclude that $u \in C^{2, \frac{\alpha}{2+\alpha}}(\overline{\Omega})$ and, ergo, $\psi, \phi \in C^{2, \frac{\alpha}{2+\alpha}}(\overline{U})$. 
\end{proof}

\subsection{Higher regularity} Once we have shown $\phi, \psi \in C^{2,s}(\overline{U})$ for some $s\in (0,1)$, we can apply classical non-linear ``Schauder" type estimates (which require the $C^{2,s}$ {\it a priori} assumption). First we need to define a non-linear elliptic and coercive system. 

\begin{defin}\label{nonlinearsystem}
Let $u^k, k = 1,2$ satisfy \begin{equation}\label{nonlinearinterior}
\begin{aligned}
F_1(y, u^1, u^2, Du^1, Du^2...,D^{t_1+s_1}u^1, D^{t_2+s_1}u^2) &= 0, y\in U\\
F_2(y, u^1, u^2, Du^1, Du^2...,D^{t_1+s_2}u^1, D^{t_2+s_2}u^2) &=0, y \in U\end{aligned}\end{equation}
 and on the boundary satisfy \begin{equation}\label{nonlinearboundary}\begin{aligned}
B_1(y, u^1, u^2, Du^1, Du^2...,D^{t_1-h_1}u^1, D^{t_2-h_1}u^2) &= 0, y\in \overline{U}\cap \{y_n =0\}\\
B_2(y, u^1, u^2, Du^1, Du^2...,D^{t_1-h_2}u^1, D^{t_2-h_2}u^2) &= 0, y\in \overline{U}\cap \{y_n =0\}.\end{aligned}\end{equation} Where, $\max s_i = 0$ and $\min\{t_k + s_i\}, \min\{t_k - h_i\} \geq 0$. 

For a solution, $v$, to \eqref{nonlinearinterior}, we say that the system is {\bf elliptic} along $v$ at a point $y_0 \in U$ if the linear system \begin{equation}\label{variationinterior}\begin{aligned} L^1_1(y_0,D)\phi^1 + L^2_1(y_0,D)\phi^2 &= \frac{d}{dt}F_1(y_0, v^1+ t\phi^1,..., D^{t_2+s_1}(v^2+t\phi^2))|_{t=0}\\
L^1_2(y_0, D)\phi^1 + L^2_2(y_0,D)\phi^2 &= \frac{d}{dt}F_2(y_0, v^1+ t\phi^1,..., D^{t_2+s_2}(v^2+t\phi^2))|_{t=0}\end{aligned}\end{equation} is elliptic. That is to say, if the block matrix $A$, where $A_{ij} = \tilde{L}^i_j$, is elliptic. Here $\tilde{L}^i_j$ is the principle part of the operator $L$ (for more details see Definition 3.1, Chapter 6 of \cite{kinderlehrerstampacchia}). 

For a solution, $v$, to both equations \eqref{nonlinearinterior} and \eqref{nonlinearboundary} we say that the boundary conditions are {\bf coercive} along $v$ at a point $y_0 \in \overline{U}\cap \{y_n=0\}$ if the linear boundary conditions \begin{equation}\label{variationboundary}\begin{aligned} \Phi^1_1(y_0,D)\phi^1 + \Phi^2_1(y_0,D)\phi^2 &= \frac{d}{dt}B_1(y_0, v^1+ t\phi^1,..., D^{t_2-h_1}(v^2+t\phi^2))|_{t=0}\\
\Phi^1_2(y_0, D)\phi^1 + \Phi^2_2(y_0,D)\phi^2 &= \frac{d}{dt}B_2(y_0, v^1+ t\phi^1,..., D^{t_2-h_2}(v^2+t\phi^2))|_{t=0}\end{aligned}\end{equation} are coercive for the \eqref{variationinterior} (see Definition \ref{weakellipticsystem}, above, for the definition of coercive linear boundary values. See Definition 3.2, Chapter 6 in \cite{kinderlehrerstampacchia} for more details).  Note in all of the above $D^nv$ is short hand for all $n$th-order derivatives of $v$. 
\end{defin}

We now recall the Schauder estimates for non-linear elliptic systems. 

\begin{thm}\label{adntheorem}[see Theorem 12.2, \cite{adn2}, Theorem 3.3 in Chapter 6, \cite{kinderlehrerstampacchia} and Chapter 6.8, \cite{morrey}]
Assume $u^k, k = 1,2$ satisfy an elliptic and coercive non linear system with proper weights like in Definition \ref{nonlinearsystem}. Let $0 < \alpha < 1$ and $\ell_0  = \max (0, -h_r)$ and assume $u^k \in C^{\ell_0 + t_k, \alpha}(\overline{U})$ for $k = 1,2$. Then for any $\ell \geq \ell_0$ if $F_i \in C^{\ell - s_i, \alpha}$ and $B_r \in C^{\ell +h_r, \alpha}$, in all arguments, then $u^k \in C^{\ell + t_k, \alpha}(\overline{U})$. 

Additionally if $F, G$ are $C^{\infty}$ (analytic)  functions in all of their arguments then $u^k$ is $C^{\infty}$ (analytic). 
\end{thm}

Our main theorem follows:

\begin{thm*}[Main Theorem]
Let $\Omega$ be a 2-sided NTA domain with $\log(h) \in C^{k,\alpha}(\partial \Omega)$ where $k \geq 0$ is an integer and $\alpha \in (0,1)$. Then:
\begin{itemize}
\item when $n =2$: $\partial \Omega$ is locally given by the graph of a $C^{k+1,\alpha}$ function.
\item when $n\geq 3$: there is some $\delta_n > 0$ such that if $\delta < \delta_n$ and $\Omega$ is $\delta$-Reifenberg flat or if $\Omega$ is a Lipschitz domain then $\partial \Omega$ is locally given by the graph of a $C^{k+1, \alpha}$ function.
\end{itemize}
Similarly, if $\log(h) \in C^{\infty}$ or $\log(h)$ is analytic we can conclude (under the same flatness assumptions above) that $\partial \Omega$ is locally given by the graph of a $C^\infty$ (resp. analytic) function. 
\end{thm*}

\begin{proof}
For $k = 0$ this result is contained in Proposition \ref{optimalregularity}. For $k = 1$ Proposition \ref{c2regularity} tells us that $\partial \Omega$ is $C^{2,s}$, $u^\pm \in C^{2,s}(\overline{\Omega}^{\pm})$ for some $0 < s < \alpha$. Theorem \ref{adntheorem}, applied as below, combined with a standard difference quotient argument, like the ones above, gives the optimal regularity; $\partial \Omega$ given by the graph of a $C^{2,\alpha}$ function and $u^\pm \in C^{2,\alpha}(\overline{\Omega}^{\pm})$

Let $k \geq 2$, and set $\ell_0 = 0, \ell = k-1, t_1 = t_2 = 2, s_1 = s_2 = 0$ and $h_1 = 2, h_2 = 1$. First, we will show that $\psi, \phi$ satisfy an elliptic and coercive non-linear system with the above weights (as defined in Definition \ref{nonlinearsystem}). This same method also works to prove $C^\infty$ or analytic regularity. 

Recall both $\phi$ and $\psi$ satisfy $$\mathrm{div} \vec{A}(Du) = 0$$ where $\vec{A}(Du)\colonequals \left(-\frac{u_1}{u_n}, -\frac{u_2}{u_n},..., \frac{1}{2}\left(\sum_{i=1}^{n-1} \left(\frac{u_i}{u_n}\right)^2 + \frac{1}{u_n^2}\right)\right)$. Therefore, the associated linear system at $y_0$ is $L^1_1v^1 = \frac{d}{dp_i}A_j(\psi(y_0))v^1_{ij}, L^2_1 \equiv 0, L_2^1 \equiv 0$ and $L^2_2v^2 = \frac{d}{dp_i}A_j(\phi(y_0))v^2_{ij}$. We have already established, in the proof of Proposition \ref{optimalregularity}, that this is an elliptic system. 

We have $B_1(y, \psi,\phi,....) =\phi + \psi$ and $B_2(y, \psi, \phi,...) = h((y',\psi(y)))\phi_n - \psi_n$, which are unchanged by linearization. Again, in the proof of Proposition \ref{optimalregularity}, we have shown that these boundary conditions are coercive for the above linear equations. Furthermore, the above values give a proper assignment of weights. 

Finally, $F_1, F_2, B_1$ are analytic in all arguments (recall that $\psi_n, \phi_n \neq 0$ in $U$) and $B_2$ is analytic in $D\psi, D\phi$ but has the same regularity in $y$ and $\psi$ that $h$ has in $x$. By assumption, $h\in C^{k,\alpha} = C^{\ell + h_2, \alpha}$ so $B_2$ has the desired regularity. Additionally, by Proposition \ref{c2regularity}, $u$ has the required initial smoothness. Thus, applying Theorem \ref{adntheorem} yields the desired result. 
\end{proof}

\appendix
\section{Proof of Theorem \ref{morreytheorem1} for $h < h_0$}\label{sec: prooffornegativeh}

Let us recall the statement we are trying to prove:

\medskip
\noindent {\bf Theorem \ref{morreytheorem1}} {\it 
Let $u^k, k =1,2$ satisfy a system of coercive and elliptic equations with proper weights. Suppose the coefficients in\eqref{weakformulation} and the $B_{rk\gamma}$ satisfy the $h-\mu$-conditions on a domain $\Gamma \supset U$, where $0 < \mu < 1$. Additionally, assume the following regularity: $f^{\alpha}_j \in C^{\rho, \mu}(U)$, $\rho = \max\{0, h-s_j+|\alpha|\}, g_{r\gamma}\in C^{\tau, \mu}(U)$ with $\tau = \max\{0, h+h_r + |\gamma|\}$ and $u^k \in C^{t_k + h, \mu}(U)$. Then \begin{equation}\label{weakholderestimate2}\sum_k \|u^k\|_{C^{t_k + h, \mu}(U)} \leq C\left(\sum_{j,\alpha} \|f_j^\alpha\|_{C^{\rho,\mu}(U)} + \sum_{r,\gamma} \|g_{r\gamma}\|_{C^{\tau,\mu}(U)} + \sum_{k}\|u^k\|_{C^0(U)}\right).\end{equation} Here $C$ is independent of the $u^k$'s, the $f$'s and the $g$'s.}

\medskip

For simplicity's sake, we establish the above in the special case where $h_0 = 0, h = -1, t_1 = t_2 = 2, s_1 = s_2 = 0$ and $p_1 = p_2 = 0$ (which is the case that is applied in the proof of Proposition \ref{optimalregularity}). However, our techniques work for $h_0 \geq 0, h \geq h_0-1$ and any proper assignment of weights. To further simplify the proof, we will make the assumptions that $U$ is bounded and that $u^k \in C^\infty(\overline{U}\backslash \{y_n = 0\})$, i.e. that $u^k$ is infinitely smooth away from $\{y_n = 0\}$. In the context of the paper, these assumptions are clearly satisfied.  This simplification can be avoided through the use of cutoff functions (e.g. in the proof of Theorem 6.2 in \cite{adn1}). 

Here we will follow closely the work of Agmon, Douglis and Nirenberg (\cite{adn1}, \cite{adn2}). Our proof has three steps; first, we present a representation formula for solutions to constant coefficient systems and show how this formula implies the desired result in that circumstance. Second, we analyze the variable coefficient case. Finally, we will justify the representation formula introduced in the first step.

\subsection{The constant coefficient case} We present a formula for solutions to constant-coefficient systems of the form \eqref{weakformulation} with boundary conditions \eqref{weakformulationbdry}. 

If every function involved is $C^\infty$ with compact support, then integration by parts and \cite{adn2} Theorem 6.1 tell us  \begin{equation}\label{kernelrep}D_i u^k(y', y_n) + C_i^k =  D_i v^k(y',y_n) + 
D_i\int_{\R^{n-1}} \sum_{r=1}^2K_{kr}(y'-x', y_n)(\tilde{g}^r(x')-\phi^r(x'))dx'\end{equation} 

for any $i = 1,...,n-1$ (this is essentially equation 6.7 in \cite{adn2} with the addition of a constant to compensate for $h = h_0 -1$).  We need to define some of the above terms:

\begin{itemize}
\item The $C_i^k$s are constants.
\item Let $\Gamma$ be the fundamental solution to the linear operator $(-1)^\chi a^k_{\chi\gamma} D^{\gamma+\chi}$.  We define $$v^k(Y) = \int_{\R^n}\sum_{|\chi| \leq m_k}(-1)^\chi \Gamma_k(Y-X)D_X^\chi \tilde{f}_\chi^k(X)dX.$$ Here $\tilde{f}^k_\chi $ is a smooth, compactly supported extension of $f_\chi^k$ to all of $\R^n$. How the extension is created is not particularly important.
\item Similarly $\tilde{g}^r(x)$ is a smooth, compactly supported extension of $g^r$ to all of $\R^{n-1}$. We will abuse notation and refer to $\tilde{g}$ as $g$ (similarly with $\tilde{f}$). 
\item $\phi^r(x') \colonequals \sum_{k=1}^2 B^r_{k}(D_{x'}, D_{x_n})v^k(x',0)$. 
\item $K_{kr}$ are kernels so that if the $\psi^r$s have sufficient smoothness/growth properties and $$U^k(y',y_n) \colonequals \int_{\R^{n-1}}\sum_r K_{kr}(y'-x', y_n)\psi^r(x')dx'$$ then $(-1)^\chi a^k_{\chi\gamma} D^{\gamma+\chi}U^k = 0$ and $$\sum_{k=1}^2 B^r_{k}(D_{y'}, D_{y_n})U^k(y', 0) = \psi^r(y').$$
\end{itemize}

Classical results imply that $\Gamma(Z), D\Gamma(Z)$ are integrable (at zero) and that $D^\chi \Gamma$ (for any $|\chi| = 2$) is a Calderon-Zygmund kernel which integrates to zero on $\R^n$. For the Poisson kernels, $K$, we turn to \cite{adn1}, Sections 2 and 3. When $s = \mathrm{ord}\;B^r_k = t_k - h_r - p_r = 2-h_r$, we can deduce that $D^sK_{kr}$ is homogenous of degree $-(n-1)$ (see \cite{adn1}, equation (2.13)'). In this case, we can write $$D^sK_{kr}(y', y_n) = \frac{\Omega(\frac{y'}{|Y|}, \frac{y_n}{|Y|})}{|Y|^{-n+1}},\; Y = (y', y_n).$$ As $D^sK_{kr}$ satisfies the same differential equation as $u^k$ we conclude that $$\int_{|y'| = 1} \Omega(y', 0) d\sigma(y') = 0$$ (see the corollary on pg 645 of \cite{adn1}). Furthermore, $D^sK$ has bounded first derivatives away from zero, so $\Omega$ is smooth. In particular, $D^sK_{kr}(y', 0)$ is a Calderon-Zygmund kernel.

As the $u$'s, $f$'s and $g$'s are assumed to be $C_c^\infty$, we can differentiate under the integral sign and rewrite \eqref{kernelrep} as \begin{equation}\label{kernelrep2}\begin{aligned}D_i u^k(y', y_n) + C_k = \int_{\R^n}\sum_{|\chi| \leq m_k}\tilde{\Gamma}_{k\chi}(Y-X)f_\chi^k(X)dX +& \\
\sum_{r=1}^2\int_{\R^{n-1}} D_Y^{2-h_r}K_{kr}(y'-x', y_n)D_{x'}^{h_r -1} (g^r(x')-\phi^r(x'))dx'.&\end{aligned}\end{equation} Where we define $$\tilde{\Gamma}_{k\chi}(Y-X) \colonequals  D_Y^{e_i + \chi}\Gamma(Y-X)$$ (depending on the parity of $t_k-h_r-p_r$ the above equation may be missing some minus signs, these omissions are irrelevant to future analysis). It should also be noted all the kernels above are either integrable or  Calderon-Zygmund kernels.  We now make a crucial claim:

\medskip

\noindent {\bf Claim:} The above \eqref{kernelrep2} holds for weak solutions of the constant-coefficient system \eqref{weakformulation} and \eqref{weakformulationbdry} under the regularity assumptions $f^{\alpha}_j \in C^{\rho, \mu}(U)$ where $\rho = \max\{0, |\alpha|-1\}$, $g_{r\gamma}\in C^{\tau, \mu}(U)$ with $\tau = \max\{0, h_r-1\}$, and $u^k \in C^{1, \mu}(U)$. 

\medskip

From this one can conclude:

\begin{lem}\label{constantcoefficientestimate}
Let $u^k, k =1,2$ satisfy a system of constant-coefficient coercive and elliptic equations with proper weights. Additionally, assume that for some for $0 < \mu < 1$: $f^{\chi}_j \in C^{\rho, \mu}(U)$ where $\rho = \max\{0, |\chi|-1\}$, $g_{r}\in C^{\tau, \mu}(U)$ with $\tau = \max\{0, h_r-1\}$ and $u^k \in C^{1, \mu}(U)$. Then \begin{equation}\label{weakholderestimate2}\sum_k \|u^k\|_{C^{1, \mu}(U)} \leq C_1\left(\sum_{j,\chi} \|f_j^\chi\|_{C^{\rho,\mu}(U)} + \sum_{r} \|g_{r}\|_{C^{\tau,\mu}(U)} + \sum_{k}\|u^k\|_{C^0(U)}\right).\end{equation} Here, $C$ is independent of the $u^k$'s, the $f$'s and the $g$'s.
\end{lem}

\begin{proof}[Proof assuming the Claim]
It suffices to estimate the $C^{1,\mu}$ norm of  $u^k|_{\{y_n = 0\}}$ (as each $u^k$ satisfies an elliptic equation in $U$, the full estimate can be obtained using weighted Schauder estimates. See, e.g., \cite{gandh} Theorem 5.1 or \cite{adn1} Theorem 9.1). 

 We use the classical fact that \begin{equation}\label{holderinterpolation}\|f\|_{C^1} \leq \varepsilon [Df]_{\alpha} + C_{\varepsilon,\alpha}\sup |f|\end{equation} where $f\in C^{1,\alpha}(\R^{n-1})$ and $[f]_\alpha = \sup_{x\neq y} \frac{|f(x) - f(y)|}{|x-y|^\alpha}$ (see equations 7.4, 7.5 in \cite{adn1}). From here it follows that we need only estimate $[D_iu^k|_{\{y_n = 0\}}]_{\mu}, i = 1,...,n-1$ in terms of the norms on the right hand side. That such an estimate exists, follows immediately from the theory of singular integrals and the fact that the kernels in \eqref{kernelrep2} are either Calderon-Zygmund kernels or integrable at 0.  
 \end{proof}

\subsection{The variable coefficient case} Given Lemma \ref{constantcoefficientestimate}, the standard way to handle variable coefficients is to ``freeze" the coefficients at a point. For any $y_0 = (y'_0, 0) \in \overline{U}$, we write:

\begin{equation}\label{weakformulationapp}
\begin{aligned}
\int_{U} \sum_{\stackrel{|\chi| \leq m_1}{|\gamma| \leq 2-m_1}} a^1_{\chi\gamma}(y_0) D^\gamma u^1 D^\chi \zeta  dx &= \int_U \sum (f_\chi^1+ [a^1_{\chi\gamma}(y_0)-a^1_{\chi\gamma}(x)]D^\gamma u^1) D^\chi\zeta dx\\
\int_{U} \sum_{\stackrel{|\chi| \leq m_2}{|\gamma| \leq 2-m_2}} a^2_{\chi\gamma}(y_0) D^\gamma u^2 D^\chi \zeta  dx &= \int_U \sum (f_\chi^2+ [a^2_{\chi\gamma}(y_0)-a^2_{\chi\gamma}(x)]D^\gamma u^2) D^\chi\zeta dx
\end{aligned}
\end{equation}
for all $\zeta \in C_0^\infty(U)$. On the boundary 
\begin{equation}\label{weakformulationbdryapp}
\begin{aligned}
\int_{ \{y_n = 0\}}\left(\sum_{k=1}^2 B^1_{k}(D_{x'}, D_{x_n}, y'_0) u^k\right) \xi dx' &= \int_{\{y_n = 0\}} (g^1+G^1) \xi dx'\\
\int_{ \{y_n = 0\}} \left(\sum_{k=1}^2 B^2_{k}(D_{x'}, D_{x_n}, y'_0) u^k\right) \xi dx' &= \int_{ \{y_n = 0\}}(g^2 + G^2)  \xi dx'
\end{aligned}
\end{equation}
for all $\xi \in C_0^\infty(\partial U\cap \{y_n = 0\})$. Here $G^r \colonequals \sum_{k=1}^2 (B^r_{k}(D_{x'}, D_{x_n}, y'_0)-B^r_{k}(D_{x'}, D_{x_n}, x')) u^k$. 

\medskip

However, na\"ive application of Lemma \ref{constantcoefficientestimate} will not work as the semi-norms $[Du^k]_{\mu}$ may appear with large coefficients on the wrong side of the inequality.

\eqref{holderinterpolation} allows us to argue $$[Du^k]_{\mu} \leq \frac{1}{2} \|u^k\|_{C^{1,\mu}}\Rightarrow \|u^k\|_{C^{1,\mu}} \leq C\|u^k\|_{C^0},$$ for $k = 1,2$ (which renders our desired estimate trivially true).  So, without loss of generality, it suffices to consider the case \begin{equation}\label{bigcalpha}  \exists k = 1,2\; \mathrm{ s.t.}\; \exists P,Q \in U\; \mathrm{with} \; \frac{|Du^k(P) - Du^k(Q)|}{|P-Q|^\mu} > \frac{1}{2}\|u^k\|_{C^{1,\mu}}.\end{equation}

Let $\lambda > 0$ be determined later and assume, without loss of generality, $P = (0, t), k = 1$. We have three cases:

\medskip

\noindent {\bf Case 1:} $|P-Q| \geq \lambda$. This easily implies $2\sup |Du^1| \geq \lambda^\mu \frac{1}{2}\|u^1\|_{C^{1,\mu}}$. From here, if $\lambda$ is sufficiently small, use \eqref{holderinterpolation} to get $$C_\lambda \|u^1\|_{C^0} \geq \|u^1\|_{C^{1,\mu}}$$ which, as stated above, yields the desired estimate. 
\smallskip

\noindent {\bf Case 2:} $|P-Q| < \lambda$ but $t \geq 2\lambda$. In this case $u^k, k = 1,2$ are solutions to an elliptic system of equations in $B_{3\lambda/2}(P) \subset U$.  Interior Schauder estimates for weak solutions (see e.g. \cite{morrey}, Theorem 6.4.3 or \cite{gandt} Chapter 8) give $$\sum_k \|u^k\|_{C^{1,\mu}(B_{5\lambda/4}(P))} \leq C_\lambda \left(\sum_{j,\alpha} \|f_j^\alpha\|_{C^{0,\mu}(B_{3\lambda/2}(P))}  + \sum_{k}\|u^k\|_{C^0(B_{3\lambda/2}(P))}\right).$$ By assumption, $$\frac{1}{2}\|u^1\|_{C^{1,\mu}(U)} < \frac{|Du^1(P) - Du^1(Q)|}{|P-Q|^\mu} \leq \|u^1\|_{C^{1,\mu}(B_{5\lambda/4}(P))}$$ and so, once we have fixed $\lambda$, we have the desired result. 
\smallskip

\noindent {\bf Case 3:} $|P-Q| < \lambda$ and $t < 2\lambda$. Consider a smooth cutoff function, $\eta \in C^\infty(\R^n)$, such that $\eta(Y) \equiv 1$ when $|Y| \leq 3\lambda$ and $\eta(Y) \equiv 0$ when $|Y| \geq 5 \lambda$. Additionally, $\eta$ can be chosen such that $|D^\ell\eta| \leq C\lambda^{-\ell}$. Now consider $V^k \colonequals \eta u^k$. $V^k$ satisfies equations similar to \eqref{weakformulationapp} and \eqref{weakformulationbdryapp} but with different right hand sides. 

We can use the representation \eqref{kernelrep2} and thus Lemma \ref{constantcoefficientestimate} on the $V^k$s. We need to estimate each term on the right. The term that comes from the interior equations is dominated by $$\|\sum_{\chi, \gamma} \eta (f_\chi^k+ [a^k_{\chi\gamma}(y_0)-a^k_{\chi\gamma}(x)]D^\gamma u^1)\|_{C^{0,\mu}} + \|\sum_{|\gamma| = 1}\sum_{\chi}[a^k_{\chi\gamma}(y_0)-a^k_{\chi\gamma}(x)]u^kD^\gamma \eta\|_{C^{0,\mu}}.$$ Note that $\eta$ is supported on $B_{5\lambda}$ so $\sup |a^k_{\chi\gamma}(y_0)-a^k_{\chi\gamma}(x)| < C\lambda^\mu$.  Also recall that the $h-\mu$-conditions imply the $a^k_{\chi\gamma}$ are H\"older continuous. Thus, the first term in the offset equation above can be dominated by $\sum_{\chi, k} \|f^k_\chi\|_{C^{0,\mu}} + C\lambda^\mu[Du^k]_\mu + C \sup |Du^k|$, where the constants above are independent of $\lambda$. Similarly, the second term can be bounded by $\sum_{\chi, k} C\lambda^{\mu-1}[u]_\mu + C\sup |u^k|\lambda^{-2} + C\lambda^{-1}\sup |u^k|$.

From the boundary terms we get $$\sum_r\left( \|\sum_{k=1}^2 (B^r_{k}(D_{x'}, D_{x_n}, y'_0)-B^r_{k}(D_{x'}, D_{x_n}, x')) \eta u^k\|_{C^{h_r -1, \mu}} + \|\eta g^r\|_{C^{h_r-1, \mu}}\right).$$ As we have seen above, we need not worry when the derivatives in the boundary operators land on $\eta$ (as these terms will all be bounded by the $C^{0,\mu}$ norms of the $f$s, $g$s and $u$s and the $C^1$ norm of the $u$s). When the derivatives all land on the $u^k$ term, we argue just as above (recalling that that $h-\mu$ conditions imply that the $B$s are H\"older continuous in position) and conclude that the coefficient of $[Du^k]_\mu$ contains a positive power of $\lambda$. 

We can then pick $\lambda$ small enough so that the coefficient of $[Du^k]_\mu$ on the right hand side is less than $1/4$. This yields the estimate $$\sum_k \|V^k\|_{C^{1, \mu}(U)} \leq \frac{1}{4} \sum_k [Du^k]_\mu +C\left(\sum_{j,\chi} \|f_j^\chi\|_{C^{\rho,\mu}(U)} + \sum_{r} \|g_{r}\|_{C^{\tau,\mu}(U)} + \sum_{k}\|u^k\|_{C^0(U)}\right).$$

But $V^k = u^k$ on $P,Q$ so we have that $$\frac{1}{2}\|u^1\|_{C^{1,\mu}(U)} < \frac{|Du^1(P) - Du^1(Q)|}{|P-Q|^\mu} \leq \|V^1\|_{C^{1,\mu}(U)}$$$$\Rightarrow \frac{1}{2}\|u^1\|_{C^{1,\mu}(U)} \leq \frac{1}{4} \sum_k [Du^k]_\mu +C\left(\sum_{j,\chi} \|f_j^\chi\|_{C^{\rho,\mu}(U)} + \sum_{r} \|g_{r}\|_{C^{\tau,\mu}(U)} + \sum_{k}\|u^k\|_{C^0(U)}\right).$$ From here the desired estimate follows immediately. As such, we are done modulo the proof that \eqref{kernelrep2} holds for non-$C^\infty$ functions.

\subsection{Justifying \eqref{kernelrep2}} It remains to prove our claim above: namely, that the representation in \eqref{kernelrep2} is valid without the {\it a priori} assumption of $C^\infty$ regularity. Here we follow closely the discussion on pages 673-674 of \cite{adn1}.  It should first be noted that the integrals on the right hand side of \eqref{kernelrep2} converge if $f^{\alpha}_j \in C^{\rho, \mu}(U)$ and $g_{r\gamma}\in C^{\tau, \mu}(U)$. 

Let $j(r)$ be an approximation to the identity and then define $$J_\varepsilon u(y', y_n) \colonequals \varepsilon^{-n+1} \int \prod_{i=1}^{n-1} j\left(\frac{y_i-x_i}{\varepsilon}\right)u(x_1,..., x_{n-1}, x_n)dx'.$$ Similarly, we can define $$J_{\varepsilon, \tilde{\varepsilon}}u(y', y_n) \colonequals \frac{1}{\tilde{\varepsilon}}\int_0^\infty j\left(\frac{y_n+\tilde{\varepsilon}-s}{\tilde{\varepsilon}}\right)J_\varepsilon u(y', s)ds.$$ For any $u$ it is clear that $J_{\varepsilon, \tilde{\varepsilon}}u$ is a $C^\infty$ function in the closed upper half plane. 

Now assume the $u^k$'s satisfy a coercive and elliptic system with constant coefficients and let the $f$'s and $g$'s be as in Definition \ref{weakellipticsystem}. Then (as the system has constant coefficients) it is true that $J_{\varepsilon, \tilde{\varepsilon}}u^k$ satisfies \eqref{weakformulation} with $J_{\varepsilon, \tilde{\varepsilon}}f^k_\chi$ on the right hand side. So, with $v^k$ defined as above, \eqref{kernelrep2} becomes  \begin{equation}\label{peturbedkernelrep1}\begin{aligned}J_{\varepsilon, \tilde{\varepsilon}} D_iu^k(y', y_n) + C^k_i(\varepsilon, \tilde{\varepsilon}) = \int_{\R^n}\sum_{|\chi| \leq m_k}\tilde{\Gamma}_{k\chi}(Y-X)&J_{\varepsilon, \tilde{\varepsilon}}f_\chi^k(X)dX + \\
\sum_{r=1}^2\int_{\R^{n-1}} D_Y^{2-h_r}K_{kr}(y'-x', y_n)G_{\varepsilon, \tilde{\varepsilon}}(x', 0)dx'&\end{aligned}\end{equation}  where $$G_{\varepsilon, \tilde{\varepsilon}}(x',0) \colonequals (J_{\varepsilon, \tilde{\varepsilon}}D_{x'}^{h_r -1} \sum_{k}B^r_k(D_{x'}, D_{x_n})(u^k(x', x_n)-v^k(x',x_n)))_{x_n = 0}.$$

Note that, for $H \in C^{0,\mu}, J_{\varepsilon, \tilde{\varepsilon}}H \stackrel{\tilde{\varepsilon}\downarrow 0}{\rightarrow} J_\varepsilon H$ uniformly (by Arzel\`a-Ascoli). By assumption $f^k_\chi$ is H\"older continuous. To analyze the boundary terms, note first that $D_{x'}^{h_r-1}B_k^r$ is an operator of order 1 and, as such, $D_{x'}^{h_r -1} \sum_{k}B^r_k(D_{x'}, D_{x_n})(u^k(x', x_n)-v^k(x',x_n))$ is at least as regular as $C^{0,\mu}$. So $J_{\varepsilon, \tilde{\varepsilon}}D_{x'}^{h_r -1} \sum_{k}B^r_k(D_{x'}, D_{x_n})(u^k(x', x_n)-v^k(x',x_n))$ (and thus its restriction to $\{x_n = 0\}$) converges in the uniform topology.

Let $\tilde{\varepsilon} \downarrow 0$ to obtain \begin{equation}\label{peturbedkernelrep1}\begin{aligned}J_{\varepsilon} D_iu^k(y', y_n) + C_k(\varepsilon) = \int_{\R^n}\sum_{|\chi| \leq m_k}\tilde{\Gamma}_{k\chi}(Y-X)J_{\varepsilon}&f_\chi^k(X)dX + \\
\sum_{r=1}^2\int_{\R^{n-1}} D_Y^{2-h_r}K_{kr}(y'-x', y_n)G_{\varepsilon}(x', 0)dx'&\end{aligned}\end{equation}  where $$G_{\varepsilon}(x',0) \colonequals (J_{\varepsilon}D_{x'}^{h_r -1} \sum_{k}B^r_k(D_{x'}, D_{x_n})(u^k(x', x_n)-v^k(x',x_n)))_{x_n = 0}.$$ Since $J_\varepsilon$ is a convolution in only the $\R^{n-1}$ directions, we can set $x_n = 0$  to obtain $$G_{\varepsilon}(x',0) = J_\varepsilon D_{x'}^{h_r-1}(g^r(x') - \phi^r(x')).$$ We note, by assumption, that $g^r$ is at least H\"older continuous. As such, we can use the same argument as above to justify taking $\varepsilon \downarrow 0$; the validity of our claim follows.

\end{document}